\documentclass[11pt,a4paper]{article}
\usepackage{amsmath,amssymb,amsthm}
\usepackage{color}
\theoremstyle{definition}
 \newtheorem{dfn}{Definition}
 \newtheorem{remark}[dfn]{Remark}

\theoremstyle{plain}
 \newtheorem{thm}[dfn]{Theorem}
 
 \newtheorem{lem}[dfn]{Lemma}

\numberwithin{equation}{section}
\newcommand{\bn}{{\bold n}}
\newcommand{\bd}{{\bold d}}
\newcommand{\bk}{{\bold k}}
\newcommand{\bu}{{\bold u}}
\newcommand{\bz}{{\bold z}}
\newcommand{\bv}{{\bold v}}
\newcommand{\bw}{{\bold w}}
\newcommand{\ba}{{\bold a}}
\newcommand{\bb}{{\bold b}}

\newcommand{\bff}{{\bold f}}
\newcommand{\bA}{{\bold A}}
\newcommand{\bB}{{\bold B}}

\newcommand{\bD}{{\bold D}}

\newcommand{\bH}{{\bold H}}
\newcommand{\bI}{{\bold I}}

\newcommand{\bK}{{\bold K}}

\newcommand{\bN}{{\bold N}}

\newcommand{\bV}{{\bold V}}

\newcommand{\bT}{{\bold T}}
\newcommand{\bU}{{\bold U}}

\newcommand{\DV}{{\rm Div}\,}
\newcommand{\dv}{{\rm div}\,}

\newcommand{\BR}{{\Bbb R}}
\newcommand{\BC}{{\Bbb C}}

\newcommand{\BN}{{\Bbb N}}

\newcommand{\CA}{{\mathcal A}}
\newcommand{\CB}{{\mathcal B}}
\newcommand{\CC}{{\mathcal C}}
\newcommand{\CD}{{\mathcal D}}
\newcommand{\CE}{{\mathcal E}}
\newcommand{\CF}{{\mathcal F}}

\newcommand{\CK}{{\mathcal K}}
\newcommand{\CL}{{\mathcal L}}

\newcommand{\CN}{{\mathcal N}}

\newcommand{\CR}{{\mathcal R}}
\newcommand{\CS}{{\mathcal S}}
\newcommand{\CT}{{\mathcal T}}
\newcommand{\CH}{{\mathcal H}}

\newcommand{\CP}{{\mathcal P}}

\newcommand{\CX}{{\mathcal X}}

\newcommand{\fq}{{\frak q}}
\newcommand{\fp}{{\frak p}}

\newcommand{\bg}{{\bold g}}
\newcommand{\bh}{{\bold h}}
\newcommand{\pd}{\partial}
\newcommand{\curl}{{\rm curl\,}}
\newcommand{\HS}{\BR^N_+}

\newcommand{\rot}{{\rm rot}\,}

\begin{document}

\title{\bf Local Well-posedness for Free Boundary Problem of Viscous
 Incompressible Magnetohydrodynamics}
\author{Kenta {Oishi} \thanks{
Department of Graduate School of Mathematics, Nagoya University \endgraf 
Furocho, Chikusaku, Nagoya, Aichi 464-8602 Japan. 
email address: m16011b@math.nagoya-u.ac.jp
}\enskip and 
\enskip  Yoshihiro Shibata
\thanks{Department of Mathematics,  Waseda University, 
\endgraf
Department of Mechanical Engineering and Materials Science,
University of Pittsburgh, USA \endgraf
mailing address: 
Department of Mathematics, Waseda University,
Ohkubo 3-4-1, Shinjuku-ku, Tokyo 169-8555, Japan. \endgraf
e-mail address: yshibata325@gmail.com \endgraf
Partially supported by Top Global University Project, 
JSPS Grant-in-aid for Scientific Research (A) 17H0109, and Toyota
 Central Research Institute Joint Research Fund }
}
\date{}
\maketitle

\begin{abstract}
In this paper, we consider the motion of incompressible magnetohydrodynamics (MHD) with 
resistivity in a  domain bounded by a free surface. 
An electromagnetic field generated by
some currents in an external domain keeps an MHD flow in a bounded domain.
On the free surface, free boundary conditions 
for MHD flow and transmission conditions for 
electromagnetic fields are imposed. We proved the local well-posedness in
 the general setting of 
domains from a mathematical point of view. The solutions are obtained in an 
anisotropic space $H^1_p((0, T), H^1_q) \cap L_p((0, T), H^3_q)$ for the velocity field and in an anisotropic space
$H^1_p((0, T), L_q) \cap L_p((0, T), H^2_q)$ for the magnetic fields with $2 < p < \infty$, $N < q < \infty$
and $2/p + N/q <1$. To prove our main result, we used $L_p$-$L_q$ maximal regularity theorem
for the Stokes equations with free boundary conditions and for the magnetic field equations with 
transmission conditions which have been obtained in \cite{FS1, S1, S2}.   
\end{abstract}
{\bf 2010 Mathematics Subject Classification}:  35K59, 76W05 \\
{\bf Keywords}: free boundary problem, transmission condition, magnetohydorodynamics, local well-posedness, 
$L_p$-$L_q$ maximal regularity, 

\section{Introduction}\label{sec:1} 
In this paper, we prove the local well-posedness of a free boundary problem for the 
viscous non-homogeneous incompressible magnetohydrodynamics.  The problem here is 
formulated as follows:  
Let $\Omega_+$ be a domain  in 
 the $N$-dimensional Euclidean space
$\BR^N$ ($N \geq 2$), and let $\Gamma$ be the boundary of $\Omega_+$.
Let $\Omega_-$ be also a domain in $\BR^N$ whose boundary is 
$\Gamma$ and $S_-$.  We assume that $\Omega_+ \cap \Omega_-
=\emptyset$.   Throughout the paper, we assume 
that $\Omega_\pm$ are uniform $C^2$ domains, that the weak Dirichlet problem
is uniquely solvable in $\Omega_+$ 
\footnote{The definition of uniform $C^2$ domains and 
the weak Dirichlet problem will be given in Sect. \ref{sec:3} below.}  and  that 
${\rm dist}\,(\Gamma, S_-) \geq {2}d_-$ with some positive constants $d_-$, where
 ${\rm dist}(A, B)$ denotes the 
distance of any two subsets $A$ and $B$ of $\BR^N$
 defined by setting ${\rm dist}(A, B)
= \inf\{|x-y| \mid x \in A, y \in B\}$.
Let $\Omega = \Omega_+ \cup \Gamma \cup \Omega_-$
and $\dot\Omega = \Omega_+ \cup \Omega_-$. 
The boundary of $\Omega$ is $S_-$.  We may consider 
the case that $S_-$ is an empty set, and in this case 
$\Omega= \BR^N$.  Physically, we consider the case where
 $\Omega_+$ is filled by a non-homogeneous 
incompressible magnetohydrodynamic (MHD) fluid
 and $\Omega_-$ is filled by 
an insulating gas.  
We consider a motion of an MHD fluid in a time dependent domain $\Omega_{t+}$ 
whose boundary is $\Gamma_t$ 
subject to an electomagnetic field generated in a domain 
$\Omega_{t-} = \Omega\setminus
(\Omega_{t+} \cup\Gamma_t)$ by some currents located 
on a fixed boundary $S_-$ of $\Omega_{t-}$. 
Let $\bn_t$ be the unit  outer normal to
$\Gamma_t$ oriented from $\Omega_{t+}$ into $\Omega_{t-}$, and 
let $\bn_-$ be  respective the unit outer normals to $S_-$. 
 Given any functions, $v_\pm$,  defined on $\Omega_{t\pm}$, $v$ is
defined by $v(x) = v_{\pm}(x)$ for $x \in \Omega_{t\pm}$ for $t \geq 0$,
where $\Omega_{0\pm} = \Omega_\pm$. Moreover, what
$v = v_\pm$ denotes that $v(x) = v_+(x)$ for $x \in \Omega_{t+}$ and 
$v(x) = v_{t-}(x)$ for $x \in \Omega_{t-}$.  Let
$$[[v]](x_0) = \lim_{x\to x_0 \atop x \in \Omega_{t+}}v_+(x)
- \lim_{x\to x_0 \atop x \in \Omega_{t-}}v_-(x)
$$
for every point $x_0 \in \Gamma_t$, which is the jump quantity of $v$ across
$\Gamma$.

The purpose of this paper is to prove the local well-posedness of the  
free boundary problem  formulated by the set of
 the following equations:
\begin{equation}\label{mhd.1}\begin{aligned}
\rho(\pd_t\bv + \bv\cdot\nabla\bv) - \DV(\bT(\bv, \fp) + \bT_M(\bH_+))  = 0,
\quad \dv\bv= 0&
&\quad &\text{in $\Omega^T_+$}, 
 \\
\mu_+\pd_t\bH_+ + \DV\{\alpha_+^{-1}\curl\bH_+ - \mu_+
(\bv\otimes\bH_+ - \bH_+\otimes\bv)\} = 0,
\quad \dv \bH_+ = 0&&\quad &\text{in $\Omega^T_+$}, 
 \\
\mu_-\pd_t\bH_- + \DV\{\alpha_-^{-1}\curl\bH_-\} = 0,
\quad \dv \bH_- = 0&&\quad &\text{in $\Omega^T_-$}, 
 \\
(\bT(\bv, \fp)+ \bT_M(\bH_+))\bn_t = 0, 
\quad V_{\Gamma_t}  = \bv\cdot\bn_t& 
&\quad &\text{on $\Gamma^T$}, 
 \\
[[(\alpha^{-1}\curl\bH) \bn_t]]- \mu_+(\bv\otimes\bH_+ - \bH_+\otimes\bv)
\bn_t = 0&
&\quad &\text{on $\Gamma^T$},   \\
[[\mu\bH\cdot\bn_t]]=0,
\quad [[\bH-<\bH, \bn_t>\bn_t]]=0&
&\quad &\text{on $\Gamma^T$},  \\
\bn_-\cdot\bH_- =0, \quad
(\curl\bH_-)\bn_- = 0&
&\quad&\text{on $S_-^T$}, \\
(\bv, \bH_+)|_{t=0}  = (\bv_0, \bH_{0+}) \quad \text{in $\Omega_{+}$}, 
\quad \bH_{-}|_{t=0} = \bH_{0-}&&\quad&\text{in $\Omega_{-}$}.
\end{aligned}\end{equation}
Here, 
$$\Omega^T_\pm = \bigcup_{0 < t < T}\, \Omega_{t\pm} \times\{t\}, \quad 
\Gamma^T = \bigcup_{0 < t < T}\, \Gamma_t \times\{t\},
\quad {S_-^T} = {S_-} \times (0, T);
$$
$\bv  = (v_1(x, t), \ldots, v_N(x, t))^\top$
is the velocity vector field, where $M^\top$ stands for the 
transposed $M$, $\fp = \fp(x, t)$ the pressure fields, and 
$\bH = \bH_\pm = (H_{\pm 1}(x, t), \ldots, H_{\pm N}(x, t))^\top$
the magnetic vector field.  The $\bv$, $\fp$, and $\bH$  are unknows,
while $\bv_0$ and  $\bH_0$  are prescribed $N$-component vectors
of functions.
As for the remaining symbols,  
$\bT(\bv, \fp) = \nu\bD(\bv)-\fp\bI$ is the viscous stress
tensor, $\bD(\bv) = \nabla \bv
+ (\nabla\bv)^\top$ is the doubled deformation tensor whose
$(i, j)$th component is $\pd_j v_{i} + \pd_i v_{j}$ with 
$\pd_i = \pd/\pd x_i$, $\bI$ the $N\times N$ unit matrix, 
$\bT_M(\bH_+) = \mu_+(\bH_+\otimes\bH_+
- \frac12|\bH_+|^2\bI)$
the magnetic stress tensor, $\curl \bv = 
 {(\nabla \bv)^\top - \nabla \bv}$ the doubled rotation
tensor whose $(i,j)$th component is $\pd_j v_{i} - \pd_i v_{j}$,
$V_{\Gamma_t}$ the velocity of the evolution of $\Gamma_t$ in the direction of 
$\bn_t$.
Moreover,  
$\rho$, $\mu_\pm$, $\nu$, and $\alpha_\pm$ are positive
constants describing respective the mass density, 
the magnetic permability,  the kinematic 
viscosity, and the conductivity. Finally, for any
matrix field $\bK$ with $(i, j)$th component 
$K_{ij}$, the quantity $\DV K$ is an $N$-vector of functions
with the $i$th  component $\sum_{j=1}^N \pd_jK_{ij}$, 
and for any $N$-vectors of functions $\bu = (u_1, \ldots,
u_N)^\top$ and $\bw = (w_1,\ldots, w_N)^\top$,  
$\dv\bu = \sum_{j=1}^N\pd_ju_j$,  
$\bu\cdot\nabla\bw$ is an $N$-vector of functions
with  the $i$th component  $\sum_{j=1}^N u_j\pd_jw_i$,
and $\bu\otimes\bw$ an $N\times N$ matrix with the $(i, j)$th component
 $u_iw_j$. We notice that  in the three dimensional case 
\begin{equation}\label{lap:1}\begin{split}
&\Delta \bv = -\DV\curl \bv + \nabla\dv \bv, 
\quad 
\DV(\bv\otimes\bH- \bH\otimes\bv)
= \bv\dv\bH - \bH\dv \bv + \bH\cdot\nabla\bv - \bv\cdot\nabla\bH, \\
&\rot\rot \bH = \DV\curl\bH, \quad
\rot(\bv\times\bH)= \DV(\bv\otimes\bH - \bH\otimes\bv), 
\end{split}\end{equation}
where $\times$ is the exterior product. 
In particular, in the three dimensional case, the set of equations 
for the  magnetic vector field in equations \eqref{mhd.1} are written by
\begin{alignat*}2
\mu_+\pd_t\bH_+ + \rot(\alpha_+^{-1}\rot\bH_+ - \mu_+\bv\times\bH_+)
=0, \quad\dv \bH_+ = 0& &\quad &\text{in $\Omega^T_+$}, \nonumber \\ 
\mu_-\pd_t\bH_- + \rot(\alpha_-^{-1}\rot\bH_- )
=0, \quad\dv \bH_- = 0& &\quad &\text{in $\Omega^T_-$}, \nonumber \\ 
[[\bn_t\times (\alpha^{-1}\rot\bH)]]- \bn_t\times (\mu_+\bv\times\bH_+)  = 0, \quad 
[[\mu\bH\cdot\bn_t]]=0,
\quad [[\bH-<\bH, \bn_t>\bn_t]]=0&
&\quad &\text{on $\Gamma^T$}. 
\end{alignat*}
This is a standard description, and so the set of equations for the magnetic field in equations \eqref{mhd.1} is 
the $N$-dimensional mathematical description of equations for the
 magnetic vector field with transmission 
conditions.

In equations \eqref{mhd.1}, there is one equation for the magnetic fields
$\bH_\pm$ too many, so that in this paper instead of \eqref{mhd.1}, 
we consider the following equations:
\begin{equation}\label{mhd.2}\begin{aligned}
\rho(\pd_t\bv + \bv\cdot\nabla\bv) - \DV(\bT(\bv, \fp) + \bT_M(\bH_+))  = 0,
\quad \dv\bv= 0&
&\quad &\text{in $\Omega^T_+$},
 \\
\mu_+\pd_t\bH_+ - \alpha^{-1}_+\Delta \bH_+-\DV\mu_+(\bv\otimes\bH_+ - \bH_+\otimes\bv) = 0
&&\quad&\text{in $\Omega^T_+$},  \\
\mu_-\pd_t\bH_- - \alpha^{-1}_- \Delta \bH_- = 0 &&\quad&\text{in $\Omega^T_-$}, 
 \\
(\bT(\bv, \fp)+ \bT_M(\bH_+))\bn_t = 0, \quad {V_{\Gamma_t} = \bv\cdot\bn_t}&&
\quad&\text{on $\Gamma^T$}, \\
[[(\alpha^{-1}\curl\bH)\bn_t]] - \mu_+(\bv\otimes\bH_+ - \bH_+\otimes\bv)
\bn_t  = 0, \quad [[\mu\dv\bH]]  =0&
&\quad &\text{on $\Gamma^T$}, 
\\
[[\mu\bH\cdot\bn_t]]=0,
\quad [[\bH-<\bH, \bn_t>\bn_t]]=0&
&\quad &\text{on $\Gamma^T$}, 
\\
\bn_-\cdot\bH_- =0, \quad
(\curl\bH_-)\bn_- = 0&
&\quad&\text{on $S_-^T$}, 
\\
(\bv, \bH_+)|_{t=0}  = (\bv_0, \bH_{0+}) \quad \text{in $\Omega_{+}$}, 
\quad \bH_{t-}|_{t=0} = \bH_{0-}&&\quad&\text{in $\Omega_-$}. 
\end{aligned}\end{equation}
Namely, two equations: $\dv \bH_\pm = 0$ in $\Omega_\pm^T$ is replaced with 
one transmission condition:
 $[[\mu\dv\bH]]=0$ on $\Gamma^T$. 
 Employing the same argument as in Frolova and Shibata \cite[Appendix]{FS1},  
 we see that in equations \eqref{mhd.2}
 if $\dv\bH= 0$ initially, then $\dv \bH=0$ in $\dot\Omega$ 
follows automatically for any $t > 0$ as long as solutions
exist.  Thus, the local well-posedness of equations \eqref{mhd.1} 
follows from that of equations \eqref{mhd.2} provided that
the initial  data $\bH_{0\pm}$ satisfy the divergence zero condition:
$\dv \bH_{0\pm} = 0$.
This paper devotes to proving the local well-posedness of equations
\eqref{mhd.2} in the maximal $L_p$-$L_q$ regularity framework.

The MHD equations can be found in \cite{C, LL}. 
The solvability of MHD equations was first obtained 
by Ladyzhenskaya and Solonnikov \cite{LS}.
The initial-boundary value problem for MHD equations 
with non-slip conditions for the velocity vector field 
and perfect wall conditions for the magnetic vector field was studied by 
Sermange and Temam \cite{ST}
in a bounded domain and by Yamaguchi \cite{Y} in an exterior domain. 
In their studies \cite{ST, Y}, the boundary is fixed. 
On the other hand, 
in the field of engineering, when a thermonuclear reaction is caused artificially,
a high-temperature plasma is sometimes subjected 
to a magnetic field and held in the air, 
and the boundary of the fluid at this time is a free one. From this point of view,
free boundary problem for MHD equations is important. The local well-posedness
for free boundary problem for MHD equations was first proved by Padula and Solonnikov
\cite{PS} in the case where $\Omega_{t+}$ is a bounded domain surrounded by
a vacuum area, $\Omega_{t-}$. In \cite{PS}, 
the solution was obtained in Sobolev-Slobodetskii
spaces in the $L_2$ framework of fractional order greater than 2.  Later on,
the global well-posedness was proved by Frolova \cite{Frolova1} and Solonnikov
and Frolova \cite{SE1}.  Moreover, the $L_p$ approach to the same problem
was done by Solonnikov {\cite{Sol15, Sol14-2}}.  When $\Omega_{t+}$ is a bounded 
domain which is surrounded by an electromagnetic field generated in a domain,  
$\Omega_{t-}$, Kacprzyk proved the local well-posedness in \cite{K1} and global
well-posedness in \cite{K2}. In \cite{K1, K2}, the solution 
was also obtained in Sobolev-Slobodetskii
spaces in the $L_2$ framework of fractional order greater than 2.

Recently, the $L_p$-$L_q$ maximal regularity theorem 
for the initial boundary value problem of the system
of parabolic equations with non-homogeneous boundary conditions 
have been studied   
by using  $\CR$-solver  in \cite{S20} and references therein and 
by using $H^\infty$ calculus in \cite{Pr-Sim} and references therein.  They are completely different approaches.  
In particular, Shibata \cite{S1, S2} proved the $L_p$-$L_q$ maximal regularity for the Stokes equations with 
non-homogeneous free boundary conditions by $\CR$-solver theory and 
Frolova and Shibata \cite{FS1} proved it for linearized equations for the magnetic vecor fields 
with transmission conditions
on the interface and {perfect} wall conditions on the fixed boundary arising in the study of 
 two phase problems for the MHD flows  also 
by using  $\CR$-solver.  The results in \cite{S1, S2, FS1} enable us to prove the 
local well-posedness for equations \eqref{mhd.2} in the $L_p$-$L_q$ maximal 
regularity class. 

Aside from dynamical boundary conditions on $\Gamma_t$, a kinematic condition, $V_{\Gamma_t} = \bv\cdot
\bn_t$, is satisfied on $\Gamma_t$, which represents $\Gamma_t$ as a set of 
points $x = x(\xi, t)$ for $\xi \in \Gamma$, where $x(\xi, t)$ is the solution of 
the Cauchy problem:
\begin{equation}\label{kin:1}
\frac{dx}{dt} = \bv(x, t) ,\quad x|_{t=0}= \xi.
\end{equation}
This expresses the fact that the free surface $\Gamma_t$ consists for all $t > 0$ of the 
same fluid particules, which do not leave it and are not incident on it 
from inside $\Omega_{t+}$. Problem \eqref{mhd.2} can be written as an initial-boundary problem with
transmission conditions on $\Gamma$ if we go over the 
Euler coordinates $x \in \dot\Omega_t = \Omega_{t+} \cup \Omega_{t-}$ to the 
Lagrange coordinates $\xi \in \dot \Omega = \Omega_+ \cup \Omega_-$ connected with
$x$ by \eqref{kin:1}. Since the velocity field, $\bu_+(\xi, t) =\bv(x, t)$,  is given only in $\Omega_+$, 
we extend it to $\bu_-$ defined on $\Omega_-$ in such a way that 
\begin{equation}\label{ext:1}\begin{aligned}
\lim_{\xi\to \xi_0 \atop \xi \in \Omega_+} {\pd_\xi^\alpha}\bu_+(\xi, t) 
&= \lim_{\xi\to \xi_0 \atop \xi\in \Omega_-} {\pd_\xi^\alpha} \bu_-(\xi, t) 
\quad\text{for $|\alpha| \leq 3$ and $(\xi_0,  t) \in \Gamma\times(0, T)$}, \\
\|\bu_-(\cdot, t)\|_{H^i_q(\Omega_-)} & \leq C_q\|\bu_+(\cdot, t)\|_{H^i_q(\Omega_+)}
\quad\text{for $i=0, 1, 2, 3$ and $t \in (0, T)$.}
\end{aligned}\end{equation}
Let $\varphi(\xi)$ be a $C^\infty(\BR^N)$ function which equals $1$ when 
${\rm dist}\,(\xi, S_-) \geq 2d_-$ and equals $0$ 
when ${\rm dist}\,(\xi, S_-) \leq  d_-$.
The connection between Euler coordinates $x$ and Lagrangian coordinates $\xi$ is
defined by setting 
\begin{equation}\label{kin:3}
x = \xi + \varphi(\xi)\int^t_0 \bu(\xi, s)\,ds = X_\bu(\xi, t).
\end{equation}
{Define $\fq_+(\xi,t):=\fp(x,t)$ and $\tilde{\bH}(\xi,t)=\tilde{\bH}_\pm(\xi,t)={\bf H}_\pm(x,t)$.} 
Problem \eqref{mhd.2}  is transformed by \eqref{kin:3} to 
the following equations:
\begin{equation}\label{mhd.3}\begin{aligned}
\rho\pd_t\bu_+ - \DV \bT(\bu_+, \fq_+) = \bN_1(\bu_+, \tilde\bH_+)
&&\quad&\text{in $\Omega_+\times(0, T)$}, \\
\dv \bu_+ = N_2(\bu_+) = \dv \bN_3(\bu_+) 
&&\quad&\text{in $\Omega_+\times(0, T)$},\\
T(\bu_+, \fq_+)\bn = \bN_4(\bu_+, \tilde\bH_+)
&&\quad&\text{on $\Gamma\times(0, T)$}, \\
\mu\pd_t \tilde\bH - \alpha^{-1}\Delta \tilde\bH = \bN_5(\bu, \tilde\bH)&&\quad&\text{in $\dot\Omega\times(0, T)$}, \\
[[\alpha^{-1}{\rm \curl}\tilde\bH]]\bn 
= \bN_6(\bu, \tilde\bH)&&\quad&\text{on $\Gamma\times(0, T)$}, \\
[[\mu\dv\tilde\bH]]= N_7(\bu, \tilde\bH) 
&&\quad&\text{on $\Gamma\times(0, T)$}, \\
[[\mu\tilde\bH\cdot\bn]]= N_8(\bu, \tilde\bH) 
&&\quad&\text{on $\Gamma\times(0, T)$}, \\
[[\tilde\bH_\tau]] = \bN_9(\bu, \tilde\bH)
&&\quad&\text{on $\Gamma\times(0, T)$}, \\
\bn_-\cdot\tilde\bH_-=0, \quad 
({\rm curl}\,\tilde\bH_-)\bn_-=0&&\quad&\text{on $S_-\times(0, T)$}, \\
{\bu_+}|_{t=0} = \bu_{0+} \quad\text{in $\Omega_+$}, \quad
\tilde\bH|_{t=0} = \tilde\bH_0&&\quad&\text{in $\dot\Omega$}.
\end{aligned}\end{equation}
Here, $\bn$ is the unit outer normal to $\Gamma$ oriented from $\Omega_+$ into $\Omega_-$,
 $\bd_\tau = \bd - <\bd, \bn>\bn$ for any $N$-vector $\bd$,   and
the $\bN_1(\bu_+, \tilde\bH_+), \ldots, \bN_9(\bu, \tilde\bH)$ are nonlinear terms defined in
Sect. \ref{sec:2} below.

Our main result is the following theorem.
\begin{thm}\label{thm:main}
Let $1 < p, q < \infty$ and $B \geq 1$. Assume that $2/p + N/q < 1$, that
$\Omega_+$ is a uniform $C^3$ domain and 
$\Omega = \Omega_+ \cup \Gamma \cup \Omega_-$ 
a uniform $C^2$ domain, 
and that
the weak  problem is uniquely solvable in $\Omega_+$ for $q$ and 
$q' = q/(q-1)$. Let initial data $\bu_{0+}$ and $\bH_{0\pm}$ 
with
$$\bu_{0+} \in B^{3-2/p}_{q,p}(\Omega_+), \quad \tilde\bH_{0\pm} \in 
B^{2(1-1/p)}_{q,p}(\Omega_\pm)$$
satisfy the conditions: 
\begin{equation}\label{initial:1}
\|\bu_{0+}\|_{B^{3-2/p}_{q,p}(\Omega_+)}
+ \|\tilde\bH_0\|_{B^{2(1-1/p)}_q(\dot\Omega)} \leq B
\end{equation}
and compatibility conditions
\begin{equation}\label{initial:2}\begin{aligned}
\dv\bu_{0+} = 0&&\quad&\text{in $\Omega_+$}, \\
((\nu\bD(\bu_{0+}) + T_M(\tilde\bH_{0+})\bn)_\tau
= 0&&\quad&\text{on $\Gamma$}, \\
([[\alpha^{-1}\curl\tilde\bH_0]]-\mu_+(\bu_{0+}\otimes \tilde\bH_{0+} - \tilde\bH_{0+}\otimes\bu_{0+}))\bn=0
&&\quad&\text{on $\Gamma$}, \\
[[\mu\dv\tilde\bH_0]]=0, \quad [[\mu\tilde\bH_0\cdot\bn]]=0,
\quad [[(\tilde\bH_0)_\tau]]=0&&\quad&\text{on $\Gamma$}, \\
\bn_-\cdot\tilde\bH_{0-} = 0, 
\quad(\curl\tilde\bH_{0-})\bn_-=0
&&\quad&\text{on $S_-$}.
\end{aligned}\end{equation}
Then, the there exists a time $T > 0$ for which
problem \eqref{mhd.3} admits unique solutions $\bu_+$ and $\tilde\bH_\pm$ with
\begin{align*}
\bu_+&\in L_p((0, T), H^3_q(\Omega_+)^N) \cap H^1_p((0, T), H^1_q(\Omega_+)^N), \\ 
\tilde\bH_\pm &\in  L_p((0, T), H^2_q(\Omega_\pm)^N) \cap H^1_p((0, T), L_q(\Omega_\pm)^N) 
\end{align*}
possessing the estimate: 
\begin{align*}
&\|\bu_+\|_{L_p((0, T), H^3_q(\Omega_+))} + 
\|\pd_t\bu_+\|_{L_p((0, T), H^1_q(\Omega_+))}\\
&\quad + \|\tilde\bH\|_{L_p((0, T), H^2_q(\dot\Omega))} + 
\|\pd_t\tilde\bH\|_{L_p((0, T), L_q(\dot\Omega))}
\leq f(B)
\end{align*}
with some polynomial $f(B)$ with respect to $B$. 
\end{thm}
\begin{remark} As was mentioned after equations \eqref{mhd.2}, if we assume that 
$\dv\tilde\bH_{0\pm} = 0$ in $\Omega_\pm$ in addition, then $\dv \bH_\pm=0$ in $\Omega^T_\pm$, and so
$\bv$ and $\bH$ are solutions of equations \eqref{mhd.1}.  Thus, we obtain the 
local well-posedness of equations \eqref{mhd.1} from Theorem \ref{thm:main}. 
\end{remark}
Finally, we explain some symbols used throughout the paper. \\
{\bf Notation} ~  We denote the set of all naturall numbers, real numbers, complex numbers by
$\BN$, $\BR$, and $\BC$, respectively, and set $\BN_0 ~ \BN \cup\{0\}$. For any 
multi-index $\kappa=(\kappa_1, \ldots, \kappa_N)$, $\kappa_j \in \BN_0$, we set
$\pd_x^\kappa = \pd_1^{\kappa_1}\cdots\pd_N^{\kappa_N}$ and $|\kappa| = \sum_{j=1}^N \kappa_j$. 
For a scalar function, $f$,  and an $N$-vector of functions, $\bg=(g_1, \ldots, g_N)^\top$, we set 
$\nabla^nf = \{\pd^\kappa_x f \mid |\kappa|=n\}$ and $\nabla^n\bg = 
\{\pd_x^\kappa g_j \mid |\kappa|=n, j=1, \ldots, N\}$.  In particular, $\nabla^0f=f$, $\nabla^0\bg=\bg$, $\nabla^1f=\nabla f$,
and $\nabla^1\bg = \nabla \bg$. For notational convention, $\nabla\bg$ and $\nabla^2\bg$ are
sometimes considered as $N^2$ and $N^3$ column vectors, respecively,  in the following way:
$$\nabla\bg = (\pd_1g_1,\ldots, \pd_Ng_1, \ldots, \pd_1g_N, \ldots, \pd_Ng_N)^\top, 
\quad \nabla^2\bg= (\cdots, \pd_1\pd_1g_\ell, \ldots, \pd_i\pd_jg_\ell, \ldots, \pd_N\pd_Ng_\ell,
\cdots)^\top$$
for $\ell=1, \ldots, N$, and $1 \leq i \leq j \leq N$. 

For $1 \leq q \leq \infty$, $m \in \BN$, $s \in \BR$,
and any domain $D \subset \BR^N$, we denote  the standard Lebesgue space, Sobolev space, 
and Besov space by $L_q(D)$, $H^m_q(D)$, and
$B^s_{q,p}(D)$  respectively, while 
 $\|\cdot\|_{L_q(D)}$, $\|\cdot\|_{H^m_q(D)}$,
and $\|\cdot\|_{B^s_{q,p}(D)}$ denote their norms. 
We write $W^s_q(D) = B^s_{q,q}(D)$ and $H^0_q(D)=L_q(D)$.
What $f= f_\pm$ means that
$f(x) = f_\pm(x)$ for $x \in \Omega_\pm$. 
 For $\CH \in \{L_q, H^m_q, B^s_{q,p}\}$, 
the function spaces $\CH(\dot \Omega)$ 
($\dot \Omega = \Omega_+ \cup \Omega_-$) 
and their norms   are  defined by setting
$$\CH(\dot \Omega) = \{f = f_\pm \mid f_\pm \in \CH(\Omega_\pm)\},
\quad \|f\|_{\CH(\dot D)} = \|f_+\|_{\CH(D_+)} +
\|f_-\|_{\CH(D_-)}. 
$$

For any Banach space $X$, 
$\|\cdot\|_X$ being its norm,
$X^d$ denotes the $d$ product space defined by
$\{x = (x_1, \ldots, x_d) \mid x_i \in X\}$, while
 the norm of $X^d$ is simply written by $\|\cdot\|_X$, which is defined by setting 
$\|x\|_X = \sum_{j=1}^d\|x_j\|_X$.
For any time interval $(a, b)$,
$L_p((a, b), X)$ and $H^m_p((a, b), X)$ denote respective the standard
$X$-valued Lebesgue space and $X$-valued Sobolev space,
while $\|\cdot\|_{L_p((a, b), X)}$ and
$\|\cdot\|_{H^m_p((a, b), X)}$ denote their norms. 
Let $\CF$ and $\CF^{-1}$ be respective the Fourier transform
and the Fourier inverse transform. Let $H^s_p(\BR, X)$, $s>0$,
be the Bessel potential space of order $s$ defined
by
\begin{align*}
H^s_p(\BR, X) &= \{f \in \CS'(\BR, X) \mid \|f\|_{H^s_p(\BR, X)}=
\|\CF^{-1}[(1+|\tau|^2)^{s/2}\CF[f](\tau)]\|_{L_p(\BR, X)} < \infty\}.
\end{align*}
where $\CS'$ denotes the set of all $X$-valued tempered distributions on
$\BR$.

Let $\ba\cdot \bb =<\ba, \bb>= \sum_{j=1}^Na_jb_j$
for any $N$-vectors $\ba=(a_1, \ldots, a_N)$  and $\bb
=(b_1, \ldots, b_N)$.  
For any $N$-vector $\ba$, let $\ba_\tau
: = \ba - <\ba, \bn>\bn$.  For any two $N\times N$-matrices 
$\bA=(A_{ij})$ and $\bB=(B_{ij})$, the quantity $\bA:\bB$ is defined by
$\bA:\bB 
= \sum_{i,j=1}^NA_{ij}B_{ji}$. 
For any domain $G$ with boundary $\pd G$, 
we set
$$(\bu, \bv)_G = \int_G\bu(x)\cdot\overline{\bv(x)}\,dx, 
\quad (\bu, \bv)_{\pd G} = \int_{\pd G}\bu(x)\cdot\overline{\bv(x)}\,d\sigma,
$$
where $\overline{\bv(x)}$ is the complex conjugate of $\bv(x)$ and 
$d\sigma$ denotes the surface element of $\pd G$. 
Given  
$1 < q < \infty$, let $q' = q/(q-1)$. Throughout the paper, the letter $C$ denotes
generic constants and   $C_{a,b, \cdots}$ 
the constant  which depends on $a$, $b$, $\cdots$.
 The values of constants $C$, $C_{a,b,\cdots}$
may be changed from line to line.

When we describe nonlinear terms $\bN_1(\bu_+, \tilde\bH_+),
\ldots, \bN_9(\bu, \tilde\bH)$ in \eqref{mhd.3},  we use the following notational conventions. 
Let $\bu_i$ ($i=1, \ldots, m)$ be $n_i$-vectors whose 
$j$th component is $u_{ij}$, and then 
$\bu_1\otimes\cdots\otimes \bu_m$  denotes an $n = {\prod}_{i=1}^m n_i$
vector whose $(j_1, \ldots, j_m)$th component is $\Pi_{i=1}^m u_{ij_i}$ and 
 the set $\{(j_1, \ldots, j_m) \mid 1 \leq j_i\leq n_i, \enskip i=1, \ldots, m\}$ is 
rearranged as $\{k \mid k=1, 2, \ldots, n\}$ and $k$ is the corresponding 
number to some  $(j_1, \ldots, j_m)$. For example, 
$\bu\otimes \nabla\bv$ is an $N+N^2$ vector whose $(i, j, k)$ component is  
$ u_i\pd_jv_k$ and $\bu\otimes\nabla\bv\otimes\nabla^2\bw$ is an 
$N+N^2+N^3$ vector whose $(i,j,k,\ell,m,n)$ component is
$u_i(\pd_jv_k)\pd_\ell\pd_mw_n$. Here, the sets $\{(i, j ,k) \mid 1 \leq i, j, k \leq N\}$ and 
$\{(i, j, k, \ell, m, n) \mid 1\leq i, j, k, \ell, m, n\leq N\}$ are rearranged as 
$\{k  \mid 1 \leq k \leq N+N^2\}$ and $\{k \mid 1 \leq k \leq N+N^2+N^3\}$,
respectively.  Let $\bu^\ell_i$ $(i=1, \ldots, m_\ell, \ell=1, \ldots, n)$ be $n^\ell_i$-
vectors, let $\bA^\ell$ be $n^\ell\times N$ 
matrices, where $n^\ell = {\prod}_{i=1}^{m_\ell} n^\ell_i$,
 and set  $\bA = \{\bA^1,  \ldots, \bA^m\}$.   And then,  we write
$$\bA(\bu^1_1\otimes\cdots\otimes\bu^1_{m_1}, \ldots, \bu^m_1\otimes
\cdots\otimes\bu^m_{n_m}) = \sum_{\ell=1}^n
\bA^\ell\bu^\ell_1\otimes\cdots\otimes\bu^\ell_{m_\ell}.
$$
When there are two sets of matrices $\bA=\{\bA^1,\ldots, \bA^n\}$ and 
$\bB = \{\bB^1, \ldots, \bB^n\}$, we write
$$(\bA-\bB)(\bu^1_1\otimes\cdots\otimes\bu^1_{m_1}, \ldots, \bu^m_1\otimes
\cdots\otimes\bu^m_{n_m}) = \sum_{\ell=1}^n
(\bA^\ell-\bB^\ell)\bu^\ell_1\otimes\cdots\otimes\bu^\ell_{m_\ell} . 
$$

\section{Derivation of nonlinear terms} \label{sec:2}

Let $\bu_+(\xi, t)$ be the velocity field with respect to the Lagrange coordinates
$\xi \in \Omega_+$, and let $\bu_-(\xi, t)$ be the extension of  $\bu_+$ to 
$\xi \in \Omega_-$ satisfying the conditions given in \eqref{ext:1}. 
Let 
$$\psi_\bu(\xi, t) = \varphi(\xi)\int^t_0\bu(\xi, s)\,ds$$
 and we consider the 
correspondence: $x = \xi + \psi_\bu(\xi, t)$ for $\xi \in \Omega$,
which has been already 
given in \eqref{kin:3}.  Let $\delta \in (0, 1)$ 
be a small constant and we assume that 
\begin{equation}\label{2.1*}
\sup_{t \in (0, T)}\|\psi_\bu(\cdot, t)\|_{H^1_\infty(\Omega)} \leq \delta.
\end{equation}
And then, the correspondence: $x = \xi + \psi_\bu(\xi, t)$
is one to one. Since $\psi_\bu(\xi, t) = 0$ when ${\rm dist}\,(\xi, S_\pm) \leq d_0$, if
$\bu$ satisfies the regularity condition:
\begin{equation}\label{2.2}
\bu(\xi, t) \in H^1_p((0, T), H^1_q(\Omega)) \cap L_p((0, T), H^3_q(\Omega)),
\end{equation}
then the correspondence $x = \xi + \psi_\bu(\xi, t)$ is a bijection from $\Omega$
onto $\Omega$, and so we set
\begin{equation}\label{2.3}
\Omega_{t\pm} = \{x = \xi + \psi_\bu(\xi, t) \mid \xi \in \Omega_\pm\},
\quad \Gamma_t = \{x = \xi + \psi_\bu(\xi, t) \mid \xi \in \Gamma\}.
\end{equation}
In the following, for notational simplicity we set 
\begin{equation}\label{psi:1}
\Psi_\bu = \int^t_0\nabla(\varphi(\xi)\bu(\xi, s))\,ds, 
\end{equation}
where $\nabla = (\pd/\pd \xi_1, \ldots, \pd/\pd \xi_N)$. 
The Jacobi matrix of the correspondence: $x = \xi + \psi_\bu(\xi, t)$ is 
\begin{equation}\label{2.4*}\frac{\pd x}{\pd \xi} = \bI + \Psi_\bu.
\end{equation}
Notice that 
\begin{equation}\label{2.1}
\|\Psi_\bu\|_{L_\infty((0, T), L_\infty(\Omega))} \leq \delta
\end{equation}
as follows from the assmption \eqref{2.1*} .  Thus, we have 
\begin{equation}\label{2.4} 
\frac{\pd \xi}{\pd x} = \Bigl(\frac{\pd x}{\pd \xi}\Bigr)^{-1}
= \bI + \sum_{k=1}^\infty({-\Psi(\xi, t)})^k= \bI + \bV_0(\Psi_\bu).
\end{equation}
Here and in the following, we set 
\begin{equation}\label{2.4*}
\bV_0(\bk) = \sum_{k=1}^\infty(-\bk)^k
\end{equation}
for $|\bk| \leq \delta ( <1)$.  The  $\bk = (\bk_1, \ldots, \bk_N) \in \BR^{N^2}$ with
$\bk_i = (k_{i1}, \ldots, k_{iN}) \in \BR^N$ denotes  the  independent variables
corresponding  to
$\Psi_\bu = (\Psi_{1\bu}, \ldots, \Psi_{N\bu})$ with $\Psi_{\ell\bu} 
= \int^t_0 \nabla(\varphi(\xi)u_\ell(\xi, s))\,ds$.
The $\bV_0(\bk)$ is a matrix
of analytic functions defined on $|\bk| \leq \delta$ with $\bV_0(0) = 0$.  Using this symbol, we have
\begin{equation}\label{2.5}
\nabla_x = (\bI + \bV_0(\Psi_\bu))\nabla_\xi, \quad \frac{\pd}{\pd x_i} = \frac{\pd}{\pd \xi_i}
+ \sum_{j=1}^NV_{0ij}(\Psi_\bu)\frac{\pd}{\pd \xi_j}
\end{equation}
where $V_{0ij}$ is the $(i, j)$th component of the $N\times N$ matrix $\bV_0$. 

For any $N$-vector of functions, $\bw(x, t) = (w_1(x, t), \ldots, w_N(x, t))^\top$, 
we set $\tilde \bw(\xi, t) = \bw(x, t)$ and $\bv(x, t) = \bu(\xi, t)$. 
And then, by \eqref{2.5} 
\begin{equation}\label{2.6}\begin{aligned}
\pd_t\bw(x, {t}) + \bv(x, t)\cdot\nabla \bw(x, t) = 
\pd_t\tilde\bw (\xi, t) + (1-\varphi(\xi))\bu(\xi, t)\cdot
 ((\bI + \bV_0(\Psi_\bu))\nabla\tilde \bw(\xi, t)).
\end{aligned}\end{equation}
By \eqref{2.5}, 
\begin{equation}\label{2.7}\begin{aligned}
\bD(\bw) & = \bD(\tilde\bw) + \CD(\Psi_\bu)\nabla\tilde\bw, \quad 
\curl \bw & = \curl \tilde\bw + \CC(\Psi_\bu)\nabla \tilde\bw
\end{aligned}\end{equation}
with
\begin{align*}
\CD(\Psi_\bu)\nabla\tilde\bw = \bV_0(\Psi_\bu)\nabla\tilde\bw + (\bV_0(\Psi_\bu)\nabla\tilde\bw)^\top,
\quad \CC(\Psi_\bu)\nabla\tilde\bw  = 
\bV_0(\Psi_\bu)\nabla\tilde\bw - (\bV_0(\Psi_\bu)\nabla\tilde\bw)^\top. 
\end{align*}
By \eqref{2.5}, 
\begin{equation}\label{2.8} 
\dv \bw = \dv \tilde\bw  + V_0(\Psi_\bu):\nabla \tilde\bw
\end{equation}
with $V_0(\Psi_\bu):\nabla \tilde\bw = \sum_{i,j=1}^N V_{0ij}(\Psi_\bu)\frac{\pd \tilde w_i}{\pd \xi_j}$.  
Analogously, for any $N\times N$ matrix of functions, $A = (A_{ij})$, we set $\tilde A(\xi, t) =A(x, t)$
and $\tilde A_{ij}(\xi, t) = A_{ij}(x, t)$. Let  
$\bV_0(\Psi_\bu):\nabla \tilde A$ be an $N$-vector of functions whose $i$-th component is
$\sum_{j,k=1}^N V_{0jk}(\Psi_\bu)\frac{\pd \tilde A_{ij}}{\pd \xi_k}$, and then
\begin{equation}\label{2.9}
\DV A = \DV \tilde A + \bV_0(\Psi_\bu):\nabla \tilde A.
\end{equation}
On the other hand, using the dual form, we see that
\begin{equation}\label{2.10}
\dv \bw = \dv((\bI + \bV_0(\Psi_\bu)^\top)\tilde\bw) = \dv \tilde\bw + \dv(\bV_0(\Psi_\bu)^\top\tilde\bw).
\end{equation}
Let $\fq(\xi, t) = \fp(x, t)$, $\bu_+(\xi, t) = \bv(x, t)$, and 
$\tilde\bH_\pm = \bH(x, t)$. By \eqref{2.6}, \eqref{2.7}, and \eqref{2.9},
we see that the first equation in equations \eqref{mhd.2} is transformed to 
\begin{align*}
&\rho\pd_t\bu_+ + \rho(1-\varphi)\bu_+\cdot(\bI + \bV_0(\Psi_\bu))\nabla\bu_+
- \DV(\nu\bD(\bu_+) + \nu\CD(\Psi_\bu)\nabla\bu_+) \\
&- \bV_0(\Psi_\bu):\nabla(\nu\bD(\bu_+) + \nu\CD(\Psi_\bu)\nabla\bu_+)
+ (\bI + \bV_0(\Psi_\bu))\nabla\fq
-\DV T_M(\tilde\bH_+) - \bV_0(\Psi_\bu):\nabla T_M(\tilde\bH_+) = 0.
\end{align*}
By \eqref{2.4*} and \eqref{2.4}, 
$(\bI + \Psi_\bu)(\bI + \bV_0(\Psi_\bu)) = \bI$, and so we have
$$\rho\pd_t\bu_+ - \DV \bT(\bu_+, \fq) = \bN_1(\bu_+, \bH_+)$$
with
\begin{align*} 
\bN_1(\bu_+, \tilde\bH_+) & = -\Psi_\bu
\{\rho\pd_t\bu_+ - \DV(\nu\bD(\bu_+)\} \\
& + \Bigl(\bI + \Psi_\bu)
\{-\rho(1-\varphi)\bu_+\cdot(\bI + \bV_0(\Psi_\bu))\nabla\bu_+ 
+\nu \DV(\CD(\Psi_\bu)\nabla\bu_+) \\
&\quad +\bV_0(\Psi_\bu):\nabla(\nu\bD(\bu_+) + \nu\CD(\Psi_\bu)\nabla\bu_+)
+ \DV\bT_M(\tilde\bH_+) + \bV_0(\Psi_\bu):\nabla\bT_M(\tilde\bH_+)\}.
\end{align*}
Using notational convention given in Notation, we may write 
\begin{equation}\label{2.12}\begin{aligned}
&\bN_1(\bu_+, \tilde\bH_+) \\
&\quad = \CA_1(\Psi_\bu)(\Psi_\bu\otimes(\pd_t\bu_+, \nabla^2\bu_+),
\bu_+\otimes\nabla\bu_+ , \nabla\Psi_\bu\otimes\nabla\bu_+,
\tilde\bH_+\otimes\nabla\tilde\bH_+), 
\end{aligned}\end{equation}
where $\CA_1(\bk)$ is a set of matrices of  smooth 
functions defined for $|\bk| \leq \delta.$ 
Combining \eqref{2.8} and  \eqref{2.9}, we see that the condition $\dv \bv= 0$
in $\Omega_t$ is transformed to 
$$\dv \bu_+ = - \bV_0(\Psi_\bu):\nabla\bu_+ = -\dv(\bV_0(\Psi_\bu)^\top\bu_+).
$$
Thus, we set $N_2(\bu_+) 
= - \bV_0(\bk):\nabla\bu_+$ and $\bN_3(\bu_+) = -\bV_0(\bk)^\top\bu_+$, 
and then by using notational convenience defined in Notation, we may write 
\begin{equation}\label{2.11}
N_2(\bu_+)  = \CA_2(\Psi_\bu)\Psi_\bu\otimes\nabla\bu_+, \quad 
\bN_3(\bu_+) = \CA_3(\Psi_\bu)\Psi_\bu\otimes\bu_+
\end{equation}
where $\CA_i(\bk)$ ($i=2,3$) are matrices of smooth functions defined for $|\bk| \leq \delta$. 
By \eqref{2.5}, 
$$\Delta\bH_\pm = \nabla\cdot\nabla\bH_\pm =\Delta\tilde\bH_\pm 
+ \nabla(\bV_0(\Psi_\bu)\nabla\tilde\bH_\pm) 
+ \bV_0(\Psi_\bu)\nabla((\bI+\bV_0(\Psi_\bu))\nabla\tilde\bH_\pm).
$$
and so noting \eqref{2.6} we see that  the second and third equations in \eqref{mhd.2} 
are transformed to
\begin{align*}\mu_+\pd_t\tilde\bH_+ - \alpha_+^{-1}\Delta\tilde\bH_+ = \bN_{5+}(\bu, \tilde\bH) 
&\quad\text{in $\Omega_+\times(0, T)$}, \\
\mu_-\pd_t\tilde\bH_- - \alpha_-^{-1}\Delta\tilde\bH_- = \bN_{5-}(\bu,  \tilde\bH)
&\quad\text{in $\Omega_-\times(0, T)$}
\end{align*}
with
\begin{align*}
\bN_{5+}(\bu, \tilde\bH) & = \mu_+\varphi{\bu}\cdot(\bI + \bV_0(\Psi_\bu))\nabla\tilde\bH_+\\
&+ \alpha^{-1}_+\{{\DV}(\bV_0(\Psi_\bu)\nabla\tilde\bH_+)
+ \bV_0(\Psi_\bu)\nabla((\bI + \bV_0(\Psi_\bu))\nabla\tilde\bH_+)\}\\
&+ \DV(\mu_+(\bu_+\otimes\tilde\bH_+ - \tilde\bH_+\otimes\bu_+))
+ \bV_0(\Psi_\bu):\nabla(\mu_+(\bu_+\otimes\tilde\bH_+ - \tilde\bH_+\otimes\bu_+));\\
\bN_{5-}(\bu, \tilde\bH) & =  \mu_-\varphi{\bu}\cdot(\bI + \bV_0(\Psi_\bu))\nabla\tilde\bH_-\\
&+ \alpha^{-1}_-\{{\DV}(\bV_0(\Psi_\bu)\nabla\tilde\bH_-)
+ \bV_0(\Psi_\bu)\nabla((\bI + \bV_0(\Psi_\bu))\nabla\tilde\bH_-)\}.
\end{align*}
Thus, using notational convention given in Notation, we may write 
\begin{equation}\label{2.13}\begin{aligned}
\bN_{5\pm}(\bu, \tilde\bH) 
= \CA_{5\pm}(\Psi_\bu)(\tilde\bH_\pm\otimes\nabla\tilde\bH_\pm, \Psi_\bu\otimes\nabla^2\tilde\bH_\pm, 
\nabla\Psi_\bu\otimes\nabla\tilde\bH_\pm, \delta_\pm\nabla\bu_\pm\otimes\tilde\bH_\pm,
{\bu_\pm\otimes\nabla\tilde\bH_\pm}), 
\end{aligned}\end{equation}
where $\delta_+=1$ and $\delta_-=0$, where  $\CA_{5\pm}(\bk)$ are two sets of matrices
of  smooth functions  defined for $|\bk| \leq \delta$.  In partuclar, 
we have the fourth equation in \eqref{mhd.3}. 

We now consider the transmission conditions. The unit outer normal, $\bn_t$, to the $\Gamma_t$ is 
represented by
$$\bn_t = \frac{{(\bI + \bV_0(\Psi_\bu))^\top}\bn}{|{(\bI + \bV_0(\Psi_\bu))^\top}\bn|}.$$
Choosing $\delta > 0$ small enough, we may write
\begin{equation}\label{2.14}
\bn_t = (\bI + \bV_\bn(\Psi_\bu))\bn, 
\end{equation}
where $\bV_\bn(\bk)$ is a matrix of smooth functions defined on $|\bk| \leq \delta$ such that
$\bV_\bn(0) = 0$.  By \eqref{2.7} 
\begin{align*}
&(\bT(\bv, \fp)+\bT_M(\bH_+))\bn_t \\
&\quad = \nu(\bD(\bu_+) + \CD(\Psi_\bu)\nabla\bu)(\bI+ \bV_\bn(\Psi_\bu))\bn
-(\bI +\bV_\bn(\Psi_\bu))\fq\bn + \bT_M(\tilde\bH_+)(\bI + \bV_\bn(\Psi_\bu))\bn=0.
\end{align*}
Choosing $\delta > 0$ small if necessary, we may assume that $(\bI + \bV_\bn(\bk))^{-1}$
exists and we may write  $(\bI + \bV_\bn(\bk))^{-1}= \bI + \bV_{\bn,-1}(\bk)$, where
$\bV_{\bn,-1}(\bk)$ is a matrix of smooth functions defined on $|\bk| \leq \delta$
such that $\bV_{\bn,-1}(0) = 0$.  
Thus, setting 
\begin{align*}
\bN_4(\bu_+, \tilde\bH_+) &= - \{\nu\bD(\bu_+)\bV_\bn(\Psi_\bu))\bn
+ \bV_{\bn,-1}(\Psi_\bu)\nu\bD(\bu_+)(\bI+\bV_\bn(\Psi_\bu)\bn)\\
& + (\bI + \bV_{\bn,-1}(\Psi_\bu))(\nu\CD(\Psi_\bu)\nabla\bu_+
+\bT_M(\tilde\bH_+))(\bI + \bV_\bn(\Psi_\bu))\bn\}, 
\end{align*}
we have
$$\bT(\bu_+, \fq) = \bN_4(\bu_+, \tilde\bH_+)\quad
\text{on $\Gamma\times(0, T)$}.
$$
Using notational convention defined in Notation, we may write
\begin{equation}\label{2.15}\begin{aligned}
\bN_4(\bu_+, \tilde\bH_+) & = \CA_4(\Psi_\bu)(\Psi_\bu\otimes\nabla\bu_+, \tilde\bH_+\otimes\tilde\bH_+).
\end{aligned}\end{equation}
where $\CA_4(\bk)$ is a set of matrices of  functions consisting of  products of  elements of $\bn$ 
 and smooth functions defined for 
$|\bk| \leq \delta$. 

By \eqref{2.7} and \eqref{2.14},
\begin{align*}
&[[(\alpha^{-1}\curl\bH)\bn_t]] - \mu_+(\bv\otimes\bH_+ - \bH_+\otimes\bv)\bn_t
\\
&\quad = [[(\alpha^{-1}(\curl\tilde\bH + \CC(\Psi_\bu)\nabla\tilde\bH)(\bI + \bV_\bn(\Psi_\bu))\bn]]
-\mu_+(\bu_+\otimes\tilde\bH_+ - \tilde\bH_+\otimes\bu_+)
(\bI + \bV_\bn({\Psi_\bu}))\bn = 0.
\end{align*}
Thus, setting 
\begin{align*}
\bN_6(\bu, \tilde\bH) & = -\{[[(\alpha^{-1}\curl\tilde\bH)\bV_\bn(\Psi_\bu)\bn]]
+[[(\alpha^{-1}\CC(\Psi_\bu)\nabla\tilde\bH)(\bI + \bV_\bn(\Psi_\bu))\bn]]\} \\
& + \mu_+(\bu_+\otimes\tilde\bH_+ - \tilde\bH_+\otimes\bu_+)
(\bI + \bV_\bn(\Psi_\bu))\bn,
\end{align*}
we have
$$[[(\alpha^{-1}\curl \tilde\bH)\bn]] = \bN_6(\bu, \tilde\bH)\quad
\text{on $\Gamma\times(0, T)$}.
$$
Using notational convention defined in Notation and noting that 
$[[\Psi_\bu]]=0$ on $\Gamma$ as follows from \eqref{ext:1}, we may write
\begin{equation}\label{2.16}\begin{aligned}
\bN_6(\bu, \tilde\bH)= \CA_{61}(\Psi_\bu)[[\alpha^{-1}\nabla\tilde\bH]] 
+ (\CB+\CA_{62}(\Psi_\bu))\bu_+\otimes\tilde\bH_+
\end{aligned}\end{equation}
where $\CA_{61}(\bk)$ and $\CA_{62}(\bk)$ are  a matrix and a set of 
matrices of functions  consisting of  products of
elements of $\bn$ and smooth functions defined for 
$|\bk| \leq \delta$, and $\CB$ is a set of matrices of functions such that 
$\CB\bu_+\otimes\tilde\bH_+ 
= \mu_+(\bu_+\otimes\tilde\bH_+ - \tilde\bH_+\otimes\bu_+)\bn$.
In particular, 
\begin{equation}\label{mat:1}
{\|\CA_{6i}(\Psi_\bu)\|_{H^1_\infty(\Omega)} \leq C\|\Psi_\bu\|_{H^1_\infty(\Omega)}}.
\end{equation} 
By \eqref{2.10}, 
$$[[\mu\dv\bH]]=[[\mu\dv\tilde\bH + \mu\bV_0(\Psi_\bu):\nabla\tilde\bH]]=0,$$
and so setting 
\begin{equation}\label{2.17}
N_7(\bu, \tilde\bH) = -[[\mu\bV_0(\Psi_\bu):\nabla\tilde\bH]] = \CA_7(\Psi_\bu)
[[\mu\nabla\tilde\bH]]
\end{equation}
we have 
$$[[\mu\dv\tilde \bH]] = N_7(\bu, \tilde\bH) \quad\text{on $\Gamma\times(0, T)$},
$$
where $\CA_7(\bk)$ is a matrix of functions consisting of products of 
elements of $\bn$ and smooth functions defined for $|\bk| \leq \delta$.  
Notice that 
\begin{equation}\label{mat:2}
{\|\CA_{7}(\Psi_\bu)\|_{H^1_\infty(\Omega)} \leq C\|\Psi_\bu\|_{H^1_\infty(\Omega)}.}
\end{equation}

By \eqref{2.14}, we have
$$[[\mu\bH\cdot\bn_t]] = [[\mu\tilde\bH(\bI + \bV_\bn(\Psi_\bu))\bn]] = 0,$$
and so , setting
\begin{equation}\label{2.18*}
N_8(\bu, \tilde\bH) = -[[\mu\tilde\bH]]\bV_\bn(\Psi_\bu)\bn
\end{equation}
we have 
$$[[\mu\tilde\bH\cdot\bn]]= N_8(\bu, \tilde\bH) \quad\text{on $\Gamma\times(0, T)$}, 
$$

Finally, by \eqref{2.14} 
\begin{align*}
[[\bH-<\bH, \bn_t>\bn_t]] &= [[\tilde\bH - < \tilde\bH, (\bI + \bV_\bn(\Psi_\bu))\bn>(\bI + \bV_\bn({\Psi_\bu}))\bn]]
\\
& = [[\tilde\bH_\tau]]-[[<\tilde\bH, \bn>\bV_\bn(\Psi_\bu)\bn]] - 
[[<\tilde\bH, \bV_\bn(\Psi_\bu)\bn>(\bI + \bV_\bn(\Psi_\bu))\bn]],
\end{align*}
and so, setting
\begin{equation}\label{2.19*}
\bN_9(\bu, \tilde\bH) = <[[\tilde\bH]], \bn>\bV_\bn(\Psi_\bu)\bn +
<[[\tilde\bH]], \bV_\bn(\Psi_\bu)\bn>(\bI + \bV_\bn(\Psi_\bu))\bn,
\end{equation}
we have
$$[[\tilde\bH_\tau]] = \bN_9(\bu, \tilde\bH) \quad\text{on $\Gamma\times(0, T)$}, 
$$
For notational simplicity, we set 
\begin{equation}\label{2.18}
(N_8(\bu, \tilde\bH), \bN_9(\bu, \tilde\bH) = \CA_8(\Psi_\bu)[[\tilde\bH]],
\end{equation}
where $\CA_8(\bk)$ is  a set of matrices 
of  functions consisting of  products of
elements of $\bn$ and smooth functions defined for 
$|\bk| \leq \delta$. Notice that 
\begin{equation}\label{mat:2}
{\|\CA_{8}(\Psi_\bu)\|_{H^1_\infty(\Omega)} \leq C\|\Psi_\bu\|_{H^1_\infty(\Omega)}}.
\end{equation}

\section{Linear Theory}\label{sec:3}

Since the coupling of the velocity field and the magnetic field in 
\eqref{mhd.3} is semilinear, the linearized equations are decouple.  Namely, 
we consider the two linearlized equations:  one is  Stokes equations
with {free boundary conditions on $\Gamma$}, 
and  another is a system  of  heat 
equations with transmission conditions on $\Gamma$ and the perfect
wall conditions on ${S_-}$.  Recall that 
$\dot\Omega = \Omega_+ \cup \Omega_-$ and $\Omega
= \dot\Omega\cup \Gamma$. 
\subsection{The Stokes equations with free boundary conditions}
\label{subsec:3-1}
This subsection is devoted to presenting the $L_p$-$L_q$
maximal regularity theorem for the  
Stokes equations with free boundary conditions.  The problem considered here is formulated by
the following equations: 
\begin{equation}\label{3.1.1}
\begin{aligned}
\rho\pd_t\bv - \DV\bT(\bv, \fq) = \bff_1&
&\quad &\text{in $\Omega_+\times(0, T)$}, \\
 \dv\bv= g = \dv\bg&
&\quad &\text{in $\Omega_+\times(0, T)$}, \\
\bT(\bv, \fq)\bn= \bh&
&\quad &\text{on $\Gamma\times(0, T)$}, \\
\bv|_{t=0}  = \bu_{0+}&
&\quad&\text{in $\Omega_+$}.
\end{aligned}
\end{equation}
To state assumptions for equations \eqref{3.1.1}, we make two definitions.
\begin{dfn}\label{dfn:1} Let $\Omega_+$ be a domain given in Introduction. 
 We say that $\Omega_+$ is a uniform ${C^3}$ domain,
if there exist positive constants $a_1$, $a_2$, and $A$ such 
that the following assertion holds:  For any $x_0 = (x_{01},
\ldots, x_{0N}) \in \Gamma$ there exist a coordinate number
$j$ and a ${C^3}$ function $h(x')$ defined on $B'_{a_1}(x_0')$ such that 
$\|h\|_{H^k_\infty(B'_{a_1}(x_0'))} \leq A$ {for $k\leq3$} and
\begin{align*}
\Omega_+\cap B_{a_2}(x_0) & = \{x \in \BR^N \mid x_j > h(x') \enskip
(x' \in B_{a_1}'(x_0')) \} \cap B_{a_2}(x_0), \\
\Gamma\cap B_{a_2}(x_0) & = \{x \in \BR^N \mid x_j = h(x') \enskip
(x' \in B_{a_1}'(x_0')) \} \cap B_{a_2}(x_0).
\end{align*}
Here,  we have set 
\begin{gather*}
y' = (y_1, \ldots, y_{j-1}, y_{j+1}, \ldots, y_N) \enskip (y \in \{x, x_0\}),
 \\
B'_{a_1}(x'_0)  = \{x' \in \BR^{N-1} \mid |x' - x'_0| < a_1\}, \\
B_{a_2}(x_0) = \{x \in \BR^N \mid |x-x_0| < a_2\}.
\end{gather*}
\end{dfn}
Let $\hat H^1_{q,0}(\Omega_+)$ be a homogeneous Sobolev space defined by
letting
\begin{equation}\label{hom:1}
\hat H^1_{q,0}(\Omega_+) = \{\varphi \in L_{q, {\rm loc}}(\Omega_+)
\mid \nabla\varphi \in L_q(\Omega_+)^N, \quad\varphi|_\Gamma=0\}.
\end{equation}
Let $1 < q < \infty$. The variational equation:
\begin{equation}\label{wd:1}
(\nabla u, \nabla\varphi)_\Omega = (\bff, \nabla\varphi)_\Omega
\quad\text{for all $\varphi \in \hat H^1_{q', 0}(\Omega_+)$}
\end{equation}
is called the weak Dirichlet problem, where $q' = q/(q-1)$.  
\begin{dfn}\label{dfn:wd} We say that the weak Dirichlet  problem
\eqref{wd:1} is uniquely solvable for an index $q$ if 
for any $\bff \in L_q(\Omega)^N$, problem \eqref{wd:1} admits
a unique solution $u \in \hat H^1_{q,0}(\Omega)$ possessing the 
estimate: $\|\nabla u\|_{L_q(\Omega)} \leq C\|\bff\|_{L_q(\Omega)}$.
\end{dfn} 
We say that $\bu \in L_q(\Omega_+)$ is solenoidal if $\bu$ satisfies 
\begin{equation}\label{solenoidal:1}
(\bu, \nabla\varphi)_{\Omega_+} = 0 
\quad\text{for any $\varphi \in \hat H^1_{q',0}(\Omega_+)$}.
\end{equation}
Let $J_q(\Omega_+)$ be the set of all solenoidal vector of functions.

In this paper, we assume that 
\begin{itemize}
\item[\thetag1]~ $\Omega_+$ is a unform $C^3$ domain.
\item[\thetag2]~ The weak Dirichlet problem is uniquely solvable
in $\Omega_+$ for indices $q \in (1, \infty)$ and $q'=q/(q-1)$. 
\end{itemize}
By assumption \thetag2, 
we see that $L_q(\Omega_+)^N = J_q(\Omega_+)\oplus G_q(\Omega_+)$, 
where $G_q(\Omega_+) 
= \{\nabla\varphi \mid \varphi \in \hat H^1_{q,0}(\Omega_+)\}$ and the symbol 
$\oplus$ here denotes the direct sum of $J_q(\Omega_+)$ and $G_q(\Omega_+)$.
\begin{thm}\label{thm:3.1.1} 
Let $1 < p, q < \infty$ with ${2/p+N/q \not=1}$, and $T > 0$. 
 Let 
$\bu_{0+} \in B^{3-2/p}_{q, p}(\dot\Omega)$ and 
let $\bff$, $g$, $\bg$, $\bh$ be functions appearing in equations \eqref{3.1.1} 
satisfying the following conditions:
\begin{gather*}
 \bff \in L_p((0, T), H^1_q(\Omega_+)^{N}), 
\quad 
 g \in L_p((0, T), H^2_q(\Omega_+)) \cap 
H^1_p((0, T), L_q(\Omega_+)), \\
\bg \in H^1_p((0, T), H^1_q(\Omega_+)^N), \quad 
\bh \in L_p((0, T), H^2_q(\Omega_+)^N) \cap H^1_p((0, T), L_q(\Omega_+)^N).
\end{gather*} 
  Assume that $\bu_0$,  $g$, and
$\bh$ satisfy the following compatibility conditions:
\begin{equation}\label{compati:3}\begin{aligned}
\dv\bu_{0+} = g|_{t=0}&&\quad &\text{on $\dot\Omega$},
\quad \bu_{0+} - \bg|_{t=0} \in J_q(\Omega_+), \\
(\nu\bD(\bu_{0+})\bn)_\tau = \bh_\tau|_{t=0}&
&\quad&\text{on $\Gamma$, \phantom{${}_\pm$}provided
$2/p + 1/q < 1$},
\end{aligned}\end{equation}
where $\bd_\tau = \bd - <\bd, \bn>\bn$. Then, problem \eqref{3.1.1} admit unique
solutions $\bv$ and $\fq$ with
\begin{align*}
\bv & \in L_p((0, T), H^3_q(\Omega_+)^N) \cap
H^1_p((0, T), H^1_q(\Omega_+)^N), \quad 
\fq  \in L_p((0, T), H^1_q(\Omega_+)+ \hat H^1_{q,0}(\Omega_+)),
\end{align*}
and $\nabla^2\fq \in L_p((0, T), L_q(\Omega_+)^{N^2})$ 
possessing the estimates:
$$
\|\pd_t\bv\|_{L_p((0, T), H^1_q(\Omega_+))}
+ \|\bv\|_{L_p((0, T), H^3_q(\Omega_+))}
\leq Ce^{\gamma_1T}\{\|\bu_0\|_{B^{3-2/p}_{q,p}(\Omega_+)}
+ F_v(\bff, g, \bg, \bh)\}$$
with
$$F_v(\bff, g, \bg, \bh) =  \|\bff\|_{L_p((0, T), H^1_q(\Omega_+))}
 + \|(g, \bh)\|_{L_p(\BR, H^2_q(\Omega_+))}
+ \|\bg\|_{H^1_p(\BR, H^1_q(\Omega_+))}
+\|(g, \bh)\|_{H^1_p(\BR, L_q(\Omega_+))}
$$
for some positive constants $C$ and $\gamma_1$ independent of $T$.
\end{thm}
\begin{remark} \thetag1
Theorem \ref{thm:3.1.1} has been proved by Shibata \cite{S1} in the 
standard case where
$$\bv \in H^1_p((0, T), L_q(\Omega_+)^N) \cap L_p((0, T), H^2_q(\Omega_+)^N).
$$
But, in Theorem \ref{thm:3.1.1}  one more additional regularity is stated, which is
necessary for our approach to prove the well-posedness of equations 
\eqref{mhd.2}. 
The idea of proving how to obtain third order regularity 
of the fluid vector field will be given in Appendix below. \\
\thetag2~ The uniqueness holds in the following sense. Let $\bv$  and $\fq$ with 
\begin{align*}
\bv & \in L_p((0, T), H^2_q(\Omega_+)^N) \cap
H^1_p((0, T), L_q(\Omega_+)^N), \quad 
 \fq  \in L_p((0, T),  H^1_q(\Omega) + \hat H^1_{q,0}(\Omega_+))
\end{align*}
satisfy the homogeneous equations: 
\begin{align*}
\rho\pd_t\bv - \DV\bT(\bv, \fq) = 0, \quad \dv\bv=0
&\quad \text{in $\Omega_+\times(0, T)$}, \\
\bT(\bv, \fq)\bn= 0
&\quad \text{on $\Gamma\times(0, T)$}, \\
\bv|_{t=0}  = 0&
\quad\text{in $\Omega_+$},
\end{align*}
then $\bv=0$ and $\fq=0$. 
\end{remark}
\subsection{Two phase problem for the linear electro-magnetic vector field equations}
\label{subsec:3-2} 
This subsection is devoted to presenting the 
$L_p$-$L_q$ maximal regularity due to 
Frolova and Shibata \cite{FS1} for the linear electro-magnetic vector field equations.  The problem is formulated by a set of the following equations:
\begin{equation}\label{3.2.1}\begin{split}
\mu\pd_t\bH - \alpha^{-1}\Delta\bH = \bff 
\quad&\text{in $\dot\Omega\times(0, T)$}, \\
[[\alpha^{-1}\curl\bH]]\bn = \bh', \quad
[[\mu\dv\bH]] = h_N 
\quad&\text{on $\Gamma\times(0, T)$}, \\
[[\bH-<\bH, \bn>\bn]]= \bk', \quad 
[[\mu \bH\cdot\bn]] = k_N
\quad&\text{on $\Gamma\times(0, T)$}, \\
{\bn_-\cdot\bH_- = 0, \quad
(\curl\bH_-)\bn_- = 0} \quad
&{\text{on $S_-\times(0, T)$}}, \\
\bH|_{t=0} = \bH_0\quad
&\text{in $\dot\Omega$}.
\end{split}\end{equation}
To state the main result, we make a definition. 
\begin{dfn}\label{dfn:2} Let $\Omega=\Omega_+\cup \Gamma \cup \Omega_-$
 be a domain given in Introduction. 
 We say that $\Omega$ is a uniform $C^2$ domain with interface $\Gamma$ 
if there exist positive constants $a_1$, $a_2$, and $A$ such 
that the following assertion holds:  For any $x_0 = (x_{01},
\ldots, x_{0N}) \in \Gamma$ there exist a coordinate number
$j$ and a $C^2$ function $h(x')$ defined on $B'_{a_1}(x_0')$ such that 
$\|h\|_{H^k_\infty(B'_{a_1}(x_0'))} \leq A$ {for $k\leq2$} and
\begin{align*}
\Gamma \cap B_{a_2}(x_0)  &= \{x \in \BR^N \mid 
x_j  =  h(x') \enskip
(x' \in B_{a_1}'(x_0')) \} \cap B_{a_2}(x_0),\\
\Omega_\pm  \cap B_{a_2}(x_0) & = \{x \in \BR^N \mid \pm x_j > h(x') \enskip
(x' \in B_{a_1}(x_0')) \} \cap B_{a_2}(x_0), 
\end{align*}
and 
for any $x_0 = (x_{01},
\ldots, x_{0N}) \in S_-$ there exist a coordinate number
$j$ and a $C^2$ function $h(x')$ defined on $B'_{a_1}(x_0')$ such that 
$\|h\|_{H^k_\infty(B'_{a_1}(x_0'))} \leq A$ {for $k\leq2$} and
\begin{align*}
\Omega \cap B_{a_2}(x_0) & = \{x \in \BR^N \mid x_j > h(x') \enskip
(x' \in B_{a_1}(x_0')) \} \cap B_{a_2}(x_0), \\
S_- \cap B_{a_2}(x_0) & = \{x \in \BR^N \mid x_j = h(x') \enskip
(x' \in B_{a_1}'(x_0')) \} \cap B_{a_2}(x_0).
\end{align*}
Here,  we have set 
\begin{gather*}
y' = (y_1, \ldots, y_{j-1}, y_{j+1}, \ldots, y_N) \enskip (y \in \{x, x_0\}),
 \\
B'_{a_1}(x'_0)  = \{x' \in \BR^{N-1} \mid |x' - x'_0| < a_1\}, \\
B_{a_2}(x_0) = \{x \in \BR^N \mid |x-x_0| < a_2\}.
\end{gather*}
\end{dfn}

\begin{thm}\label{thm:3.2.1}
Let $1 < p, q < \infty$, ${2/p+N/q \not=1,2}$,  and $T > 0$.  Assume that $\Omega$ 
is a uniform $C^2$ domain with interface $\Gamma$.   
Then, there exists a $\gamma_2$ such that the following assertion holds:
Let $\bH_0 \in B^{2(1-1/p)}_{q,p}(\dot\Omega)$ and let 
$\bff \in L_p((0, T), L_q(\dot\Omega)^N)$, and let 
 $\tilde\bh = (\tilde\bh', \tilde h_N)$, and $\tilde\bk = (\tilde \bk', \tilde k_N)$
be  functions such that $\tilde \bh' = \bh'$, $\tilde h_N=h_N$, $\tilde \bk'
=\bk$, and $\tilde k_N = k_N$ for $t \in (0, T)$, where $\bh'$, $h_N$, $\bk'$, 
and $k_N$ are functions given in the right side of \eqref{3.2.1}, and 
the following conditions hold: 
\begin{align*}
e^{-\gamma t}\tilde\bh \in L_p(\BR, H^1_q(\Omega)^N) \cap 
H^{1/2}_p(\BR, L_q(\Omega)^N), \quad
e^{-\gamma t}\tilde\bk \in L_p(\BR, H^2_q(\Omega)^N) \cap 
H^1_p(\BR, L_q(\Omega)^N)
\end{align*}
for any $\gamma \geq \gamma_2$.  Moreover, we assume that $\bH_0$, $\bh$ and 
$\bk$ satisfy the following compatibility conditions:
\begin{align}
&[[\alpha^{-1}\curl\bH_0]]\bn=\bh'|_{t=0}, \quad
[[\mu\dv\bH_0]]=h_N|_{t=0} \quad\text{on $\Gamma$}, 
\quad(\curl\bH_{0-})\bn_- = 0
\quad\text{on $S_-$} \label{compati:4}
\intertext{provided ${2/p+N/q < 1}$;} 
&[[\bH_0-<\bH_0, \bn>\bn]]=\bk'|_{t=0}, \quad
[[\mu\bH_0\cdot\bn]]=k_N|_{t=0} \quad\text{on $\Gamma$}, \quad
{\bn_-}\cdot\bH_{0-} = 0 \quad\text{on $S_-$} \label{compati:5}
\end{align}
provided ${2/p+N/q < 2}$. 
Then, problem \eqref{3.2.1} admits a unique solution $\bH$
with
$$\bH
\in L_p((0, T), H^2_q(\dot\Omega)^N) \cap 
H^1_p((0, T), L_q(\dot\Omega)^N)$$
possessing the estimate:
$$
\|\pd_t\bH\|_{L_p((0, T), L_q(\dot\Omega))}
+ \|\bH\|_{L_p((0, T), H^2_q(\dot\Omega))} 
\leq Ce^{\gamma T}\{\|\bH_0\|_{B^{2(1-1/p)}_{q,p}(\dot\Omega)} 
+ F_H(\bff, \tilde\bh, \tilde\bk)\}
$$
with
\begin{align*}
F_H(\bff, \tilde\bh, \tilde\bk) &=  \|\bff\|_{L_p(\BR, L_q(\dot\Omega))} 
+ \|e^{-\gamma t}\tilde\bh\|_{L_p(\BR, H^1_q(\Omega))}
 + \|e^{-\gamma t}\tilde\bh\|_{H^{1/2}_p(\BR, L_q(\Omega))} \\
&+ \gamma^{1/2}\|e^{-\gamma t}\tilde\bh\|_{L_p(\BR, L_q(\Omega))} 
+ \|e^{-\gamma t}\tilde\bk\|_{L_p(\BR, H^2_q(\Omega))}
+ \|e^{-\gamma t}\pd_t\tilde\bk\|_{L_p(\BR, L_q(\Omega))}\}
\end{align*}
for any $\gamma \geq \gamma_2$ with some constant $C > 0$ independent of 
$\gamma$. 
\end{thm}
\begin{remark} \thetag1~Theorem \ref{thm:3.2.1}
 was proved by Froloba and Shibata \cite{FS2}. 
\vskip0.5pc
\thetag2~ The uniqueness holds in the following sense.  Let 
$\bH$ with
$$\bH
\in L_p((0, T), H^2_q(\dot\Omega)^N) \cap 
H^1_p((0, T), L_q(\dot\Omega)^N)$$
satisfies the homogeneous equations: 
\begin{equation}\label{null:1}\begin{split}
\mu\pd_t\bH - \alpha^{-1}\Delta\bH = 0
\quad&\text{in $\dot\Omega\times(0, T)$}, \\
[[\alpha^{-1}\curl\bH]]\bn = 0, \quad
[[\mu\dv\bH]] =0
\quad&\text{on $\Gamma\times(0, T)$}, \\
[[\bH-<\bH, \bn>\bn]]= 0, \quad 
[[\mu \bH\cdot\bn]] = 0
\quad&\text{on $\Gamma\times(0, T)$}, \\
\bn_\pm\cdot\bH_\pm = 0, \quad
(\curl\bH_\pm)\bn_\pm = 0 \quad
&\text{on $S_\pm\times(0, T)$}, \\
\bH|_{t=0} = 0\quad
&\text{in $\dot\Omega$}.
\end{split}\end{equation}
then $\bH=0$ in $\dot\Omega\times(0, T)$. 
\end{remark}
\section{Estimate of non-linear terms}

Let $\bu_+$ and $\bH_\pm$ be $N$-vectors of functions such that
\begin{equation}\label{cond:1}\begin{aligned}
\bu_+ &\in H^1_p((0, T), H^1_q(\Omega_+)^N) \cap L_p((0, T), H^3_q(\Omega_+)^N), 
&\quad \bu_+|_{t=0} &= \bu_{0+}, \\
\bH_\pm &\in H^1_p((0, T), L_q(\Omega_\pm)^N) \cap L_p((0, T), H^2_q(\Omega_\pm)^N), 
&\quad \bH_\pm|_{t=0} &= \tilde\bH_{0\pm}, 
\end{aligned}\end{equation}
and we shall estimate nonlinear terms 
$\bN_1(\bu_+, \bH_+),\ldots, \bN_9(\bu, \bH)$ appearing in the right side of
equations \eqref{mhd.3}.  Here, $\bw = \bw_+$ for $x \in \Omega_+$ and $\bw=\bw_-$
for $x \in \Omega_-$ ($\bw \in \{\bu, \bH\}$) and $\bu_-$ is an extension of $\bu_+$ defined in
\eqref{ext:1}.  For notational simplicity, we set
\begin{gather*}
E^1_T(\bu_+) = \|\pd_t\bu_+\|_{L_p((0, T), H^1_q(\Omega_+))} + 
\|\bu_+\|_{L_p((0, T), H^3_q(\Omega_+))}, \\
E^{2, \pm}_T(\bH_\pm)  = \|\pd_t\bH_\pm\|_{L_p((0, T), L_q(\Omega_\pm))}
+ \|\bH_\pm\|_{L_p((0, T), H^2_q(\Omega_\pm))},
\quad E^2_T(\bH) = E^{2,+}_T(\bH_+) + E^{2,-}_T(\bH_-).
\end{gather*}
Moreover, let $\bu^i_+$ and $\bH^i_\pm$ ($i=1,2$) be $N$-vectors of functions such that
\begin{equation}\label{cond:1}\begin{aligned}
\bu^i_+ &\in H^1_p((0, T), H^1_q(\Omega_+)^N) \cap L_p((0, T), H^3_q(\Omega_+)^N), 
&\quad \bu^i_+|_{t=0} &= \bu_{0+}, \\
\bH^i_\pm &\in H^1_p((0, T), L_q(\Omega_\pm)^N) \cap L_p((0, T), H^2_q(\Omega_\pm)^N), 
&\quad \bH^i_\pm|_{t=0} &= \tilde\bH_{0\pm}.
\end{aligned}\end{equation}
We also consider the differences: $\CN_1= \bN_1(\bu^1_+, \bH^1_+) - \bN_1(\bu_+^2, \bH_+^2), 
\ldots, \CN_9= \bN_9(\bu^1, \bH^1) - \bN_9(\bu^2, \bH^2)$. 
Here, $\bw = \bw_+$ for $x \in \Omega_+$ and $\bw=\bw_-$
for $x \in \Omega_-$ ($\bw \in \{\bu^1, \bu^2, \bH^1, \bH^2\}$) and $\bu_-^i$ is  
an extension
of $\bu^i_+$  defined in \eqref{ext:1}. 
For notational simplicity, we assume that 
\begin{equation}\label{initial:4.1}
\|\bu_{0+}\|_{B^{3-2/p}_{q,p}(\Omega_+)} + \|\tilde\bH_{0+}\|_{B^{2(1-1/p)}_{q,p}(\Omega_+)}
+ \|\tilde\bH_{0-}\|_{B^{2(1-1/p)}_{q,p}(\Omega_-)} \leq B
\end{equation}
for some constant $B > 0$.  
In what follows, we assume that 
$2 < p <\infty$, $N < q < \infty$ and $2/p + N/q < 1$.  
To estimate nonlinear terms, we use the following inequalities which follows from Sobolev's inequality.
\begin{equation}\label{4.3}\begin{aligned}
\|f\|_{L_\infty(\Omega_\pm)} & \leq C\|f\|_{H^1_q(\Omega_\pm)}, \\ 
\|fg\|_{H^1_q(\Omega_\pm)} & \leq C\|f\|_{H^1_q(\Omega_\pm)}
\|g\|_{H^1_q(\Omega_\pm)},  \\
\|fg\|_{H^2_q(\Omega_\pm)} & \leq C(\|f\|_{H^2_q(\Omega_\pm)}
\|g\|_{H^1_q(\Omega_\pm)}
+ \|f\|_{H^1_q(\Omega_\pm)}
\|g\|_{H^2_q(\Omega_\pm)}),  \\
\|fg\|_{W^{1-1/q}_q(\Gamma)} & \leq C\|f\|_{W^{1-1/q}_q(\Gamma)}
\|g\|_{W^{1-1/q}_q(\Gamma)},  \\
\|fg\|_{W^{2-1/q}_q(\Gamma)} & \leq C(\|f\|_{W^{2-1/q}_q(\Gamma)}
\|g\|_{W^{1-1/q}_q(\Gamma)}
+ \|f\|_{W^{1-1/q}_q(\Gamma)}
\|g\|_{W^{2-1/q}_q(\Gamma)}).
\end{aligned}\end{equation}
By H${\rm \ddot o}$lder's inequality and \eqref{ext:1}, we have
\begin{equation}\label{4.4}\begin{aligned}
\|\Psi_\bw&\|_{L_\infty((0, T), H^2_q(\Omega))} \leq CT^{1/p'}E^1_T(\bw_+)
\quad\text{for $\bw \in \{\bu, \bu^1, \bu^2\}$}, \\
&\|\Psi_{\bu^1} - \Psi_{\bu^2}\|_{L_\infty((0, T), H^2_q(\Omega))} 
\leq CT^{1/p'}E^1_T(\bu_+^1-\bu_+^2).
\end{aligned}\end{equation}
In view of  \eqref{4.3} and \eqref{4.4}, choosing $T>0$ so small, we may assume that 
\begin{equation}\label{2.1**}
\|\Psi_\bw\|_{L_\infty((0, T), L_\infty(\Omega))} \leq \delta
\quad\text{for $\bw \in \{\bu, \bu^1, \bu^2\}$}. 
\end{equation}
In the following, for simplicity, choosing $T > 0$ small,  we also assume that 
\begin{equation}\label{4.4.2} T^{1/p'}(E^1_T(\bw) +B) \leq 1
\quad\text{for $\bw \in \{\bu_+, \bu^1_+, \bu^2_+\}$}.
\end{equation}
We may assume that the unit outer normal $\bn$ to $\Gamma$ is defined on
$\BR^N$ and $\|\bn\|_{H^2_\infty(\BR^N)} < \infty$ because $\Omega_+$
is a uniform $C^3$ domain.  Thus, setting 
$$[\CA_i(\Psi_\bw)]_{T,2} := \|\CA_i(\Psi_\bw)\|_{L_\infty((0, T), L_\infty(\Omega))}
+ \|\nabla\CA_i(\Psi_\bw)\|_{L_\infty((0, T), H^1_q(\Omega))},
$$
 by \eqref{2.1**}, \eqref{4.3}, \eqref{4.4}, and \eqref{4.4.2} we have
\begin{equation}\label{4.4.3}
[\CA_i(\Psi_\bw)]_{T,2} \leq C
\end{equation}
with some constant $C > 0$ for $i = 1, \ldots, 9$ and $\bw \in \{\bu, \bu^1, \bu^2\}$,
where we have set $\CA_5(\Psi_\bw) = (\CA_{5+}(\Psi_\bw), \CA_{5-}(\Psi_\bw))$
and $\CA_6(\Psi_\bw) = (\CA_{61}(\Psi_\bw), \CA_{62}(\Psi_\bw))$. 

We first consider $\bN_1(\bu_+, \bH_+)$ and $\CN_1 = \bN_1(\bu_+^1, \bH_+^1) - \bN_1(\bu_+^2, \bH_+^2)$. 
Recall \eqref{2.12}.  Applying \eqref{4.3} and using \eqref{4.4}, we have
\begin{equation}\label{4.5.0}\begin{aligned}
&\|\Psi_\bu\otimes(\pd_t\bu_+, \nabla^2\bu_+)
\|_{L_p((0, T), H^1_q(\Omega_+))}
\leq CT^{1/p'}E^1_T(\bu_+)^2, \\
&\|\bu_+\otimes\nabla\bu_+\|_{L_p((0, T), H^1_q(\Omega_+))}
\leq CT^{1/p}\|\bu_+\|^2_{L_\infty((0, T), H^2_q(\Omega_+))}, \\
&{\|\nabla\Psi_\bu\otimes\nabla\bu_+\|_{L_p((0, T), H^1_q(\Omega_+))}
\leq CT^{1/p'}E^1_T(\bu_+)^2}, \\
&\|\bH_+\otimes\nabla\bH_+\|_{{L}_p((0, T), H^1_q(\Omega_+))}
\leq C\|\bH_+\|_{L_\infty((0, T), H^1_q(\Omega_+))}
\|\bH_+\|_{L_p((0, T), H^2_q(\Omega_+))}.
\end{aligned}\end{equation}

By real interpolation theory, we see that 
\begin{equation}\label{real.int}\begin{aligned}
&\sup_{t \in (0, T)} \|\bv_\pm({\cdot, t})\|_{B^{\ell+2-2/p}_{q,p}(\Omega_\pm)}
\\
&\quad \leq C(\|\bv|_{t=0}\|_{B^{\ell+2-2/p}_{{q,p}}(\Omega_\pm)}
+ \|\pd_t\bv_\pm\|_{L_p((0, T), H^\ell_q(\Omega_\pm))}
+ \|\bv_\pm\|_{L_p((0, T), H^{\ell+2}_q(\Omega_\pm))})
\end{aligned}\end{equation} 
($\ell=0,1$){. }
{In order to prove this, we make a few preparations.  
For a $X$-valued function $f(\cdot, t)$ defined for $t \in (0, T)$, where $X$ is a Banach space, we set
\begin{equation}\label{4.8}
e_T[f](\cdot, t) = \begin{cases} 0 & \quad\text{for $t < 0$}, \\
f(\cdot, t)  & \quad\text{for $0 < t < T$}, \\
f(\cdot, 2T-t)  & \quad\text{for $T < t < 2T$}, \\
0  & \quad\text{for $t > 2T$}.
\end{cases}
\end{equation}
Then, $e_T[f](\cdot, t) = f(\cdot, t)$ for $t \in (0, T)$.  If $f|_{t=0}=0$, then 
\begin{equation}\label{4.9}
\pd_te_T[f](\cdot, t) = \begin{cases} 0 & \quad\text{for $t < 0$}, \\
\pd_tf(\cdot, t)  & \quad\text{for $0 < t < T$}, \\
-(\pd_tf)(\cdot, 2T-t)  & \quad\text{for $T < t < 2T$}, \\
0  & \quad\text{for $t > 2T$}.
\end{cases}
\end{equation}
In partiular, we have
\begin{equation}\label{4.10}\begin{aligned}
\|e_T[f]\|_{L_p(\BR, X)} & \leq 2\|f\|_{L_p((0, T), X)}, \\
\|\pd_te_T[f]\|_{L_p(\BR, X)} & \leq 2\|\pd_tf\|_{L_p((0, T), X)}.
\end{aligned}\end{equation}}
{Let $\bw_\pm$ be a $N$-vector of function defined on $\Omega_\pm$ 
and let $E_\mp[\bw_\pm]$ be an extension of $\bw_\pm$ to $\Omega_\mp$ for which 
\begin{gather}
E_\mp[\bw_\pm] = \bw_\pm\quad\text{ for $x \in \Omega_\pm$}, 
\nonumber \\
\lim_{x\to x_0 \atop x \in \Omega_\mp}\pd_x^\alpha E_\mp[\bw_\pm](x, t)
= \lim_{x\to x_0 \atop x \in \Omega_\pm} \pd_x^\alpha \bw_\pm(x, t)
\quad\text{for $|\alpha| \leq 1$ and $x_0 \in \Gamma$}, \label{jump:1} \\
\|E_\mp[\bw_\pm](\cdot, t)\|_{H^i_q(\Omega)}
\leq C\|\bw_\pm(\cdot, t)\|_{H^i_q(\Omega_\pm)}
\quad\text{for $i=0,1,2$}. \nonumber
\end{gather}
Let $E_{\BR^N}[E_\mp[\bw_{\pm}]]$ 
be an extension of $E_\mp[\bw_{\pm}]$ to $\BR^N$ for which
\begin{equation}\label{4.11}\begin{aligned}
E_{\BR^N}[E_\mp[\bw_{\pm}]] &= E_\mp[\bw_{\pm}] \quad\text{on $\Omega$}, \\
\|E_{\BR^N}[E_\mp[\bw_{\pm}]]\|_{B^{\ell+2(1-1/p)}_{q,p}(\BR^N)}
& \leq C\|\bw_{\pm}\|_{B^{\ell+2(1-1/p)}_{q,p}(\Omega_\pm)} \quad (\ell=0,1). 
\end{aligned}\end{equation}
For $\bv_0 \in B^{\ell+2(1-1/p)}_{q,p}(\BR^N)$, let 
\begin{align} \label{semigroup}
T(t)\bv_0 = e^{(-1 + \Delta)t}\bv_0
\end{align}
be a $C^0$ analytic semigroup satisfying the condition:
$T(0)\bv_0 = \bv_0$ and possessing the estimate: 
\begin{equation} \label{4.12} \begin{aligned}
\|T(\cdot)\bv_0\|_{L_p((0, \infty), H^{\ell+2}_q(\BR^N))}
+ \|\pd_tT(\cdot)\bv_0\|_{L_p((0, \infty), H^\ell_q(\BR^N))}
\leq C\|\bv_0\|_{B^{\ell+2(1-1/p)}_{q,p}(\BR^N)} 
\end{aligned}\end{equation}
for $\ell=0,1,2$. Let $\bw_\pm$ be defined on $\Omega_\pm\times(0,T)$ and set $\bw_{0\pm}=\bw_\pm|_{t=0}$. 
Let $\psi(t)$ be a function which equals one for $t > -1$ and zero for $t < -2$
and let 
\begin{equation}\label{4.13*}\begin{aligned}
\CT[E_\mp[\bw_{0\pm}]](t) & = \psi(t)T(|t|)E_{\BR^N}[E_\mp[\bw_{0\pm}]], \\
\end{aligned}\end{equation}}
{Then, by \eqref{4.11} and \eqref{4.12} 
\allowdisplaybreaks
\begin{align}
\CT[E_\mp[\bw_{0\pm}]](0) & = E_\mp[\bw_{0\pm}] \quad\text{in $\Omega$},  
\nonumber\\
\|\CT[E_\mp[\bw_{0\pm}]]\|_{L_p(\BR, H^{2+\ell}_q(\Omega))}
&+ \|\pd_t\CT[E_\mp[\bw_{0\pm}]]\|_{L_p(\BR, H^\ell_q(\Omega))}
\leq C\|\bw_0\|_{B^{2(1-1/p)+\ell}_{q,p}(\dot\Omega)}. \label{4.13a}
\end{align}
Set 
\begin{equation}\label{4.14}\begin{aligned}
\CE[E_\mp[ \bw_\pm]] &= \CT[E_\mp[\bw_{0\pm}]] + e_T[E_\mp[\bw_\pm] - 
\CT[E_\mp[\bw_{0\pm}]]]. 
\end{aligned}\end{equation}
Obviously, $\CE[E_\mp[\bw_\pm]] = E_\mp[\bw_\pm]$ for $t \in (0, T)$. 
Then, \eqref{real.int} is guaranteed by \eqref{4.10}, \eqref{jump:1} and \eqref{4.13a} as follows: 
\begin{align*}
&\sup_{t \in (0, T)} \|\bv_\pm({\cdot, t})\|_{B^{\ell+2-2/p}_{q,p}(\Omega_\pm)}
= \sup_{t \in (0, T)} \|\CE[E_\mp[\bv_\pm]]({\cdot, t})\|_{B^{\ell+2-2/p}_{q,p}(\BR^N)}
\\
&\quad \leq C(\|\bv|_{t=0}\|_{B^{\ell+2-2/p}_{{q,p}}(\Omega_\pm)}
+ \|\pd_t\bv_\pm\|_{L_p((0, T), H^\ell_q(\Omega_\pm))}
+ \|\bv_\pm\|_{L_p((0, T), H^{\ell+2}_q(\Omega_\pm))}). 
\end{align*}}

{Combining \eqref{real.int} with \eqref{initial:4.1}} leads to 
\begin{equation}\label{4.5}\begin{aligned}
\|\bw_+\|_{L_\infty((0, T), H^2_q(\Omega_+))} & \leq C(B+E^1_T(\bw_+))
&\quad&\text{for $\bw \in \{\bu, \bu^1, \bu^2\}$}, \\
\|\bz_\pm\|_{L_\infty((0, T), H^1_q(\Omega_\pm))} & \leq C(B+E^{2,\pm}_T(\bz_\pm))
&\quad&\text{for $\bz \in \{\bH, \bH^1, \bH^2\}$},
\end{aligned}\end{equation}
because $B^{\ell+2-2/p}_{q,p}(\Omega_\pm)$ is continuously imbedded into 
$H^{\ell+1}_q(\Omega_\pm)$ as follows from $2-2/p > 1$, that is, $2 < p < \infty$. 
Moreover, 
\begin{equation}\label{4.5.1}
\|\bH_+\|_{L_\infty((0, T), H^1_q(\Omega_+))}
\leq C(B + T^{s/p'(1+s)}E^{2,+}_T(\bH_+)),
\end{equation}
provided  $0 < T < 1$. 
In fact, we write $\bH_+ = \tilde\bH_{0+} + \bv$ with $\bv = \bH_+ - \tilde\bH_{0+}$. 
Since $\bv|_{t=0}=0$, we have
$$\|\bv(\cdot, t)\|_{L_q(\Omega_+)} \leq \int^t_0\|\pd_s\bH_+(\cdot, s)\|_{L_q(\Omega_+)} \leq T^{1/p'}E^{2,+}_T(\bH_+)
$$
for $t \in (0, T)$. On the other hand,  choosing $s \in (0, 1- 2/p)$ yields that 
 $B^{2(1-1/p)}_{q,p}(\Omega_+)$
is continuously imbedded into $W^{1+s}_q(\Omega_+)$, 
and so by \eqref{real.int} 
$$\|\bv\|_{W^{1+s}_q(\Omega_+)} \leq C\|\bv\|_{B^{2(1-1/p)}_{q,p}(\Omega_+)}
\leq C(B + E^{2,+}_T(\bH_+)).
$$
Since $\|\bv\|_{H^1_q(\Omega_+)} 
\leq C_s\|\bv\|_{L_q(\Omega_+)}^{s/(1+s)}\|\bv\|_{W^{1+s}_q(\Omega_+)}^{1/(1+s)}
$, we have \eqref{4.5.1} provided $0 < T < 1$. 

Combining \eqref{4.5.0}, \eqref{4.4.3}, \eqref{4.5}, and \eqref{4.5.1}, we have
\begin{equation}\label{n:1}\begin{aligned}
&\|\bN_1(\bu_+, \bH_+)\|_{L_p((0, T), H^1_q(\Omega_+))}\\
&\quad\leq C(T^{1/p'}E^1_T(\bu)^2 + T^{1/p}(B+E^1_T(\bu_+))^2
+ (B+T^{s/p'(1+s)}E^{2,+}_T(\bH_+))E^{2,+}_T(\bH_+)).
\end{aligned}\end{equation}
We next consider $\CN_1 = \bN_1(\bu_+^1, \bH^1_+) - \bN_1(\bu_+^2, \bH_+^2)$, which is represented
by $\CN_1 = \CN_{11} + \CN_{12}$ with
\begin{align*}
\CN_{11}&= (\CA_1(\Psi_{\bu^1}) - \CA_1(\Psi_{\bu^2}))(K^1_1, K^1_2, K^1_3, K^1_4), \\
\CN_{12}&=\CA_1(\Psi_{\bu^2})(K^1_1-K^2_1, K^1_2-K^2_2, K^1_3-K^2_3, K^1_4-K^2_4), \\
K^i_1 &= \Psi_{\bu^i}\otimes(\pd_t\bu_+^i, \nabla^2\bu_+^i), 
\quad K^i_2 = \bu_+^i\otimes\nabla\bu_+^i, \quad 
K^i_3 = \nabla\Psi_{\bu^i}\otimes\nabla\bu^i_+, \quad 
K^i_4 = \bH^i_+\otimes\nabla\bH^i_+. 
\end{align*}
Representing 
\begin{equation}\label{4.6.0}\CA_i(\Psi_{\bu^1}) -\CA_i(\Psi_{\bu^2}) = \int^1_0(d_{\bk}\CA_i)(\Psi_{\bu^2} + \theta(
\Psi_{\bu^1}-\Psi_{\bu^2}))\,d\theta(\Psi_{\bu^1}-\Psi_{\bu^2}),
\end{equation}
by \eqref{4.3}, \eqref{2.1**}, \eqref{4.4}, and \eqref{4.4.2}, we have
\begin{equation}\label{4.6.1}
\|\CA_i(\Psi_{\bu^1}) - \CA_i(\Psi_{\bu^2})\|_{L_\infty((0, T), H^2_q(\Omega_+))}
\leq CT^{1/p'}\|\bu^1_+-\bu^2_+\|_{L_p((0, T), H^3_q(\Omega_+))}
\end{equation}
for $i=1, \ldots, 9$. Estimating $\|(K^1_1, \ldots, K^1_4)\|_{L_p((0, T), H^1_q(\Omega_+))}$ in the 
same manner as in proving \eqref{n:1}, we have
\begin{equation}\label{nd:1.1} \|\CN_{11}\|_{L_p((0, T), H^1_q(\Omega_+))}
\leq CT^{1/p'}(B^2 + E^1_T(\bu^1_+)^2 + E^{2,+}_T(\bH^1_+)^2)E^1_T(\bu^1_+-\bu^2_+),
\end{equation}
where we have used $0 < T < 1$ and $(B+T^{s/p'(1+s)}E^{2,+}_T(\bH_+))E^{2,+}_T(\bH_+)
\leq (1/2)B^2 + (3/2)E^{2,+}_T(\bH_+)^2$.  To estimate $\CN_{12}$ we use the fact:
\begin{equation}\label{4.6.2} |f_1f_2-g_1g_2| \leq |f_1-g_1||f_2| + |f_2-g_2||g_1|.
\end{equation}
Thus, by \eqref{4.3}, \eqref{4.4}, \eqref{4.4.2}, 
 \eqref{4.5}, and \eqref{4.5.1}, we have
\begin{align*}
&\|K^1_1-K^2_1\|_{L_p((0, T), H^1_q(\Omega_+))} \\
&\quad \leq C(\|\Psi_{\bu^1}-\Psi_{\bu^2}\|_{L_\infty(0, T), H^1_q(\Omega_+))}
\|(\pd_t\bu^1_+, \nabla^2\bu^1_+)\|_{L_p((0, T), H^1_q(\Omega_+))} \\
&\hskip2cm 
+\|\Psi_{\bu^2}\|_{L_\infty((0, T), H^1_q(\Omega_+))}E^1_T(\bu_+^1-\bu_+^2))\\
&\quad\leq CT^{1/p'}(E^1_T(\bu^1_+) + E^2_T(\bu^2_+))E^1_T(\bu_+^1-\bu_+^2);\\
&\|K^1_2-K^2_2\|_{L_p((0, T), H^1_q(\Omega_+))} 
 \leq CT^{1/p}( E^1_T(\bu^1_+) + E^2_T(\bu^2_+)+B)E^1_T(\bu^1_+-\bu^2_+); \\
&\|K^1_3-K^2_3\|_{L_p((0, T), H^1_q(\Omega_+))} 
\leq CT^{1/p'}(E^1_T(\bu^1_+) + E^2_T(\bu^2_+)+B)E^1_T(\bu^1_+-\bu^2_+); \\
&\|K^1_4-K^2_4\|_{L_p((0, T), H^1_q(\Omega_+))} \\
&\quad\leq C(\|\bH^1_+-\bH^2_+\|_{L_\infty((0, T), H^1_q(\Omega_+))}
E^{2,+}_T(\bH^1_+) + \|\bH^2_+\|_{L_\infty((0, T), H^1_q(\Omega_+))}E^{2,+}_T(\bH^1_+-\bH^2_+))\\
&\quad \leq C(T^{\frac{s}{p'(1+s)}}E^{2,+}_T(\bH^1_+-\bH^2_+)E^{2,+}_T(\bH^1_+)
 +(B+T^{\frac{s}{p'(1+s)}}E^{2,+}_T(\bH^2_+))E^{2,+}_T(\bH^1_+-\bH^2_+)\\
 &\quad \leq C\{B+T^{\frac{s}{p'(1+s)}}(E^{2,+}_T(\bH^1_+)
+E^{2,+}_T(\bH^2_+))\}E^{2,+}_T(\bH^1_+-\bH^2_+),
\end{align*}
which, combined with \eqref{4.4.3} and  \eqref{nd:1.1}  gives that
\begin{equation}\label{nd.1}\begin{aligned}
\|\CN_1\|_{L_p((0, T), H^1_q(\Omega_+))}
& \leq C[\{
T^{1/p'}(B^2 + E^1_T(\bu^1_+)^2 + E^1_T(\bu^2_+)^2 + E^{2,+}_T(\bH^1_+)^2)
+ \\
&
+ (T^{1/p'}+T^{1/p})(E^1_T(\bu^1_+) + E^1_T(\bu^2_+)+B)\}E^1_T(\bu^1_+-\bu^2_+) \\
&+ (B+T^{\frac{s}{p'(1+s)}}(E^{2,+}_T(\bH^1_+) + E^{2,+}_T(\bH^2_+)))
E^{2,+}_T(\bH^1_+-\bH^2_+)]. 
\end{aligned}\end{equation}
We now estimate $N_2$, $\bN_3$, $\CN_2$ and $\CN_3$. By \eqref{4.3}, \eqref{4.4}, 
\eqref{4.5},  and  \eqref{4.4.3},
to the formula of $\bN_2(\bu_+, \bH_+)$ given in \eqref{2.11}, we have
\begin{equation}\label{n:2.1} \begin{aligned}
\|N_2(\bu_+)\|_{L_p((0, T), H^2_q(\Omega_+)} & \leq C
[\CA_2(\Psi_{\bu})]_{T,2}\|\Psi_\bu\|_{L_\infty((0, T), H^2_q(\Omega_+))}
\|\nabla\bu_+\|_{L_p((0, T), H^2_q(\Omega_+))} \\
& \leq CT^{1/p'}E^1_T(\bu_+)^2. 
\end{aligned}\end{equation}
By \eqref{2.1**}, \eqref{4.4}, \eqref{4.6.0}, \eqref{4.4.2},  and \eqref{4.5}, 
\begin{equation}\label{4.6.3}\begin{aligned}
\|\pd_t\CA_i(\Psi_\bw)&\|_{L_\infty((0, T), H^1_q(\Omega))} 
\leq C\|\bu_+\|_{L_\infty((0, T), H^2_q(\Omega_+))} \leq C{(B+E^1_T(\bw_+))}
\quad\text{for $\bw \in \{\bu, \bu^1, \bu^2\}$}; \\
\|\pd_t(\CA_i(\Psi_{\bu^1})&-\CA_i(\Psi_{\bu^2}))\|_{L_\infty((0, T), H^1_q(\Omega))} \leq
C(\|\bu^1_+-\bu^2_+\|_{L_\infty((0, T), H^2_q(\Omega_+))} \\
&\quad + (\|\bu^1_+\|_{L_\infty((0, T), H^2_q(\Omega_+))} + 
\|\bu^2_+\|_{L_\infty((0, T), H^2_q(\Omega_+))})\|\Psi_{\bu^1}-\Psi_{\bu^2}\|_{L_\infty((0, T), H^1_q(\Omega_+))})
\\
& \leq CE^1_T(\bu^1_+-\bu^2_+)
\end{aligned}\end{equation}
By \eqref{4.3}, \eqref{4.4.2}, \eqref{4.4}, 
\eqref{4.5}, \eqref{4.4.3}, and \eqref{4.6.3}, we have
\allowdisplaybreaks
\begin{align*}
\|\pd_t{N}_2&(\bu_+)\|_{L_p((0, T), L_q(\Omega_+))} \\
&\leq C\{\|\pd_t\CA_2(\Psi_{\bu})\|_{L_\infty((0, T), H^1_q(\Omega))}
\|\Psi_\bu\|_{L_\infty((0, T), H^1_q(\Omega_+))}
T^{1/p}\|\nabla\bu_+\|_{L_\infty((0, T), L_q(\Omega_+))}\\
&+
 T^{1/p}[\CA_2(\Psi_{\bu})]_{T,2}\|\pd_t\Psi_\bu\|_{L_\infty((0, T), H^1_q(\Omega_+))}
\|\nabla\bu_+\|_{L_\infty((0, T), L_q(\Omega_+))}
\\
& + [\CA_2(\Psi_{\bu})]_{T,2}\|\Psi_\bu\|_{L_\infty((0, T), H^1_q(\Omega_+))}
\|\pd_t\nabla\bu_+\|_{L_p((0, T), L_q(\Omega_+))}\}
\\
& \leq C\{T^{1/p'}E^1_T(\bu_+)^2 +T^{1/p}(B+E^1_T(\bu_+))^2\};\\
{\|\bN_3(\bu}&{{}_+)\|_{L_p((0, T), H^1_q(\Omega_+))}} \\
&{\leq CT^{1/p}[\CA_3(\Psi_{\bu})]_{T,2}
 \|\Psi_\bu\|_{L_\infty((0, T), H^1_q(\Omega_+))}\|\bu_+\|_{L_\infty((0, T), H^1_q(\Omega_+))}\}} \\
& {\leq CT^{1/p}(B+E^1_T(\bu_+))^2,} \\
\|\pd_t\bN_3&(\bu_+)\|_{L_p((0, T), H^1_q(\Omega_+))} \\
&\leq C\{\|\pd_t\CA_3(\Psi_{\bu})\|_{L_\infty((0, T), H^1_q(\Omega))}
\|\Psi_\bu\|_{L_\infty((0, T), H^1_q(\Omega_+))}
T^p\|\bu_+\|_{L_\infty((0, T), H^1_q(\Omega_+))}\\
&\quad + T^{1/p}[\CA_3(\Psi_{\bu})]_{T,2}
\|\pd_t\Psi_\bu\|_{L_\infty((0, T), H^1_q(\Omega_+))}\|\bu_+\|_{L_\infty((0, T), H^1_q(\Omega_+))}
\\
&\quad  +[\CA_3(\Psi_{\bu})]_{T,2}
 \|\Psi_\bu\|_{L_\infty((0, T), H^1_q(\Omega_+))}\|\pd_t\bu_+\|_{L_p((0, T), H^1_q(\Omega_+))}\}
\\
& \leq C\{T^{1/p'}E^1_T(\bu_+)^2 +T^{1/p}(B+E^1_T(\bu_+))^2\},
\end{align*}
which, combined with \eqref{n:2.1}, yields that
\begin{equation}\label{n:2}\begin{aligned}
&\|N_2(\bu_+)\|_{L_p((0, T), H^2_q(\Omega_+))} + \|\pd_tN_2(\bu_+)\|_{L_p((0, T), L_q(\Omega_+))}
+ \|\bN_3(\bu_+)\|_{H^1_p((0, T), H^1_q((\Omega_+))}\\
&\quad \leq C\{T^{1/p'}E^1_T(\bu_+)^2 +T^{1/p}(B+E^1_T(\bu_+))^2\}.
\end{aligned}\end{equation}
To estimate $\CN_2$ and $\CN_3$, we write 
$\CN_2 = \CN_{21} + \CN_{22}$ and $\CN_3 = \CN_{31}+\CN_{32}$ with 
\begin{align*}
\CN_{21} & = (\CA_2(\Psi_{\bu^1}) -\CA_2(\Psi_{\bu^2}))\Psi_{\bu^1}\otimes\nabla\bu^1_+, \quad
\CN_{22}  = \CA_2(\Psi_{\bu^2})(\Psi_{\bu^1}\otimes\nabla\bu^1_+-\Psi_{\bu^2}\otimes\nabla\bu^2_+)\\
\CN_{31} & = (\CA_3(\Psi_{\bu^1}) -\CA_3(\Psi_{\bu^2}))\Psi_{\bu^1}\otimes\bu^1_+, \quad
\CN_{32}  = \CA_3(\Psi_{\bu^2})(\Psi_{\bu^1}\otimes\bu^1_+-\Psi_{\bu^2}\otimes\bu^2_+).
\end{align*}
Employing the same argument as in proving \eqref{n:2} and using \eqref{4.4.2},
\eqref{4.6.1} and \eqref{4.6.3}, 
we have
\begin{equation}\label{nd:2.1}\begin{aligned}
\|\CN_{21}\|_{L_p((0, T), H^2_q(\Omega_+))} &
\leq CT^{1/p'}E^1_T({\bu^1_+})E^1_T(\bu^1_+-\bu^2_+), \\
\|\pd_t\CN_{21}\|_{L_p((0, T), L_q(\Omega_+))}
&\leq C(T^{1/p'}E^1_T(\bu^1_+) + T^{1/p}(B+E^1_T(\bu^1_+)))
E^1_T(\bu^1_+-\bu^2_+); \\
{\|\CN_{31}\|_{L_p((0, T), H^1_q(\Omega_+))}}
&\leq {CT(B+E^1_T(\bu^1_+))E^1_T(\bu^1_+-\bu^2_+),} \\
\|\pd_t\CN_{31}\|_{L_p((0, T), H^1_q(\Omega_+))}
&\leq C(T^{1/p'}E^1_T(\bu^1_+) + T^{1/p}(B+E^1_T(\bu^1_+)))
E^1_T(\bu^1_+-\bu^2_+).
\end{aligned}\end{equation}
By \eqref{4.4.2}, \eqref{4.4}, \eqref{4.4.3}, \eqref{4.6.2}, \eqref{4.6.3}, and \eqref{4.5}, 
we have
\begin{equation}\label{nd:2.2}\begin{aligned}
\|\CN_{22}&\|_{L_p((0, T), H^2_q(\Omega_+))}\\
& \leq C[\CA_2(\Psi_{\bu^2})]_{T,2}\{\|\Psi_{\bu^1}-\Psi_{\bu^2}\|_{L_\infty((0, T), H^2_q(\Omega_+))}
\|\nabla\bu^1_+\|_{L_p((0, T), H^2_q(\Omega_+))} \\
& + \|\Psi_{\bu^2}\|_{L_\infty((0, T), H^2_q(\Omega_+))}
\|\nabla(\bu^1-\bu^2)\|_{L_p((0, T), H^2_q(\Omega_+))}\}\\
& \leq CT^{1/p'}({E^1_T(\bu^1_+) + E^1_T(\bu^2_+)})E^1_T(\bu^1_+-\bu^2_+); \\
\|\pd_t\CN_{22}&\|_{L_p((0, T), L_q(\Omega_+))} \\
& \leq 
C\{\|\pd_t\CA_2(\psi_{\bu^2})\|_{L_\infty((0, T), H^1_q(\Omega))}
(\|\Psi_{\bu^1}-\Psi_{\bu^2}\|_{L_\infty((0, T), H^1_q(\Omega_+))}
T^{1/p}\|\nabla\bu^1_+\|_{L_\infty((0, T), L_q(\Omega_+))}\\
&\qquad + T^{1/p}\|\Psi_{\bu^2}\|_{L_\infty((0, T), H^1_q(\Omega_+))}
\|\nabla(\bu^1_+-\bu^2_+)\|_{L_\infty((0, T), L_q(\Omega_+))})\\
&+T^{1/p}[\CA_2(\Psi_{\bu^2})]_{T,2}
(\|\pd_t(\Psi_{\bu^1}-\Psi_{\bu^2})\|_{L_\infty((0, T), H^1_q(\Omega_+))}
\|\nabla\bu^1_+\|_{L_\infty((0, T), L_q(\Omega_+))}\\
&\qquad + \|\pd_t\Psi_{\bu^2}\|_{L_\infty((0, T), H^1_q(\Omega_+))}
\|\nabla(\bu^1_+-\bu^2_+)\|_{L_\infty((0, T), L_q(\Omega_+))})\\
&+[\CA_2(\Psi_{\bu^2})]_{T,2}( \|\Psi_{\bu^1}-\Psi_{\bu^2}\|_{L_\infty((0, T), H^1_q(\Omega_+))}
\|\pd_t \nabla\bu^1_+\|_{L_p((0, T), L_q(\Omega_+))})\\
&\qquad + \|\Psi_{\bu^2}\|_{L_\infty((0, T), H^1_q(\Omega_+))}
\|\pd_t\nabla(\bu^1_+-\bu^2_+)\|_{L_p((0, T), L_q(\Omega_+))})\}\\
& \leq C\{T^{1/p'}(E^1_T(\bu^1_+) + E^1_T(\bu^2_+))
+T^{1/p}(B+E^1_T(\bu^1_+) + E^1_T(\bu^2_+))\}
E^1_T(\bu^1_+-\bu^2_+).
\end{aligned}\end{equation}
Employing the same argument as in proving the second inequality in
\eqref{nd:2.2}, we also have
\begin{align*}
&{\|\CN_{32}\|_{L_p((0, T), H^1_q(\Omega_+))} 
\leq CT^{1/p'}(E^1_T(\bu^1_+) + E^1_T(\bu^2_+))E^1_T(\bu^1_+-\bu^2_+),} \\
&\|\pd_t\CN_{32}\|_{L_p((0, T), {H^1_q}(\Omega_+))} \\
&\quad\leq C\{T^{1/p'}(E^1_T(\bu^1_+) + E^1_T(\bu^2_+))
+T^{1/p}(B+E^1_T(\bu^1_+) + E^1_T(\bu^2_+))\}
E^1_T(\bu^1_+-\bu^2_+), 
\end{align*}
which, combined with \eqref{nd:2.1} and \eqref{nd:2.2}, yields that
\begin{equation}\label{nd:2}\begin{aligned}
&\|\CN_2\|_{L_p((0, T), H^2_q(\Omega_+))} + \|\pd_t\CN_2\|_{L_p((0, T), L_q(\Omega_+))}
\|\pd_t\CN_3\|_{{H^1_p}((0, T), {H^1_q}(\Omega_+))}\\
&\quad \leq C(T^{1/p'}E^1_T(\bu^1_+) + T^{1/p}(B+E^1_T(\bu^1_+)))
E^1_T(\bu^1_+-\bu^2_+).
\end{aligned}\end{equation}

We now esstimate $\bN_4$ and $\CN_4$.  Applying \eqref{4.3} to the formula given in
\eqref{2.15}, we have
\begin{align*}
&\|\bN_4(\bu_+, \bH_+)\|_{L_p((0, T), H^2_q(\Omega_+))} \\ 
&\quad\leq C[\CA_4(\Psi_\bu)]_{T,2}(\|\Psi_\bu\|_{L_\infty((0, T), H^2_q(\Omega_+))}
\|\nabla\bu_+\|_{L_p((0, T), H^2_q(\Omega_+))} \\
&\hskip4cm+ \|\bH_+\|_{L_\infty((0, T), H^1_q(\Omega_+))}
\|\bH_+\|_{L_p((0, T), H^2_q(\Omega_+))}); \\
&\|\pd_t\bN_4(\bu_+, \bH_+)\|_{L_p((0, T), L_q(\Omega_+))} \\
&\quad \leq C\{\|\pd_t\CA_4(\Psi_\bu)\|_{L_\infty((0, T), H^1_q(\Omega_+))}
(\|\Psi_\bu\|_{L_\infty((0, T), H^1_q(\Omega_+))}
\|\nabla\bu_+\|_{L_p((0, T), L_q(\Omega_+))} \\
&\hskip4cm+ \|\bH_+\|_{L_\infty((0, T), H^1_q(\Omega_+))}
\|\bH_+\|_{L_p((0, T), L_q(\Omega_+))})\\
&\quad +T^{1/p}[{\CA_4(\Psi_\bu)}]_{T,2}
\|\pd_t\Psi_\bu\|_{L_\infty((0, T), H^1_q(\Omega_+))}
\|\nabla\bu_+\|_{L_\infty((0, T), L_q(\Omega_+))})\\
&\quad + [{\CA_4(\Psi_\bu)}]_{T,2}(
\|\Psi_\bu\|_{L_\infty((0, T), H^1_q(\Omega_+))}\|\pd_t\nabla\bu_+\|_{L_p((0, T), L_q(\Omega_+))}\\
&\hskip4cm+ \|\bH_+\|_{L_\infty((0, T), H^1_q(\Omega_+))}
\|\pd_t\bH_+\|_{L_p((0, T), L_q(\Omega_+))})\}.
\end{align*}
Thus, by \eqref{4.4}, \eqref{4.4.3}, \eqref{4.5.1}, \eqref{4.4.2}, and \eqref{4.6.3}, 
we have
\begin{equation}\label{n:3}\begin{aligned}
\|\bN_4(\bu_+, \bH_+)\|_{L_p((0, T), H^2_q(\Omega_+))}
&\leq C(T^{1/p'}E^1_T(\bu_+)^2 
+ (B+T^{\frac{s}{p'(1+s)}}E^{2,+}_T(\bH_+))E^{2,+}_T(\bH_+)); \\
\|\pd_t\bN_4(\bu_+, \bH_+)\|_{L_p((0, T), L_q(\Omega_+))}
&\leq C(T^{1/p'}E^1_T(\bu_+)^2 + T^{1/p}(B+E^1_T(\bu_+))^2 \\
&\quad + {(1+B+E^1_T(\bu_+))}(B+T^{\frac{s}{p'(1+s)}}E^{2,+}_T(\bH_+))E^{2,+}_T(\bH_+)).
\end{aligned}\end{equation}
To estimate $\CN_4$, we write $\CN_4 = \CN_{41} + \CN_{42}$ with 
\begin{align*}
\CN_{41} & = (\CA_4(\Psi_{\bu^1}) - \CA_4(\Psi_{\bu^2}))(\Psi_{\bu^1}\otimes\nabla\bu^1_+, 
\bH^1_+\otimes\bH^1_+), \\
\CN_{42} &= \CA_4(\Psi_{\bu^2})((\Psi_{\bu^1}-\Psi_{\bu^2})\otimes\nabla \bu^1_+
+ \Psi_{\bu^2}\otimes\nabla(\bu^1_+-\bu^2_+), 
(\bH^1_+-\bH^2_+)\otimes\bH^1_+ + \bH^2_+\otimes(\bH^1_+-\bH^2_+)).
\end{align*}
By \eqref{4.3}, we have 
\allowdisplaybreaks
\begin{align*}
&\|\CN_{41}\|_{L_p((0, T), H^2_q(\Omega_+))} \\
&\quad \leq C\|\CA_4(\Psi_{\bu^1})-\CA_4(\Psi_{\bu^2})\|_{L_\infty((0, T), H^2_q(\Omega_+))}
( \|\Psi_{\bu^1}\|_{L_\infty((0, T), H^2_q(\Omega_+))}\|\nabla\bu^1_+\|_{L_p((0, T), H^2_q(\Omega_+))}\\
&\hskip2cm+\|\bH^1_+\|_{L_\infty((0, T), H^1_q(\Omega_+))}
\|\bH^1_+\|_{L_p((0, T), H^2_q(\Omega_+))}); \\
&\|\pd_t\CN_{41}\|_{L_p((0, T), L_q(\Omega_+))} \\
&\quad \leq C\{T^{1/p}\|\pd_t(\CA_4(\Psi_{\bu^1})-\CA_4(\Psi_{\bu^2}))\|_{L_\infty((0, T), H^1_q(\Omega_+))}
(\|\Psi_{\bu^1}\|_{L_\infty((0, T), H^1_q(\Omega_+))}\|\nabla\bu^1_+\|_{L_\infty((0, T), L_q(\Omega_+))} \\
&\hskip2cm+ \|\bH^1_+\|_{L_\infty((0, T), H^1_q(\Omega_+))}\|\bH^1_+\|_{L_\infty((0, T), L_q(\Omega_+))})\\
&\quad + \|\CA_4(\Psi_{\bu^1})-\CA_4(\Psi_{\bu^2})\|_{L_\infty((0, T), H^1_q(\Omega_+))}
({T^{1/p}} \|\pd_t\Psi_{\bu^1}\|_{L_\infty((0, T), H^1_q(\Omega_+))}\|\nabla\bu^1_+\|_{L_\infty((0, T), L_q(\Omega_+))} \\
&\quad+\|\Psi_{\bu^1}\|_{L_\infty((0, T), H^1_q(\Omega_+))}
\|\pd_t\nabla\bu^1_+\|_{L_p((0, T), L_q(\Omega_+))} 
+ \|\bH^1_+\|_{L_\infty((0, T), H^1_q(\Omega_+))}\|\pd_t\bH^1_+\|_{L_p((0, T), L_q(\Omega_+))})\}.\\
\end{align*}
Thus, by \eqref{4.4}, \eqref{4.4.2}, \eqref{4.4.3}, \eqref{4.5}, \eqref{4.6.1}, 
and \eqref{4.6.3}, we have 
\allowdisplaybreaks
\begin{align}\label{nd4.1}\begin{aligned}
\|\CN_{41}\|_{L_p((0, T), H^2_q(\Omega_+))} 
&\leq CT^{1/p'}\{E^1_T(\bu^1_+) + (B+E^{2,+}_T(\bH^1_+))E^{2,+}_T(\bH^1_+)\}E^1_T(\bu^1_+-\bu^2_+); \\
\|\pd_t\CN_{41}\|_{L_p((0, T), L_q(\Omega_+))}  &\leq C[T^{1/p}\{(B+E^1_T(\bu^1_+)) 
+ (B+E^{2,+}_T(\bH^1_+))^2\} \\
&+ T^{1/p'}\{E^1_T(\bu^1_+) + (B+E^{2,+}_T(\bH^1_+))E^{2,+}_T(\bH^1_+)\}]
E^1_T(\bu^1_+-\bu^2_+).
\end{aligned}\end{align}
On the other hand, by \eqref{4.3}, 
\allowdisplaybreaks
\begin{align*}
\|\CN_{42}\|_{L_p((0, T), H^2_q(\Omega_+))}
&\leq C[\CA_4(\Psi_{\bu^2})]_{T,2}(\|\Psi_{\bu^1}-\Psi_{\bu^2}\|_{L_\infty((0, T), H^2_q(\Omega_+))}
\|\nabla\bu^1_+\|_{L_p((0, T), H^2_q(\Omega_+))}\\
&\quad + \|\Psi_{\bu^2}\|_{L_\infty((0, T), H^2_q(\Omega_+))}
\|\nabla(\bu^1_+-\bu^2_+)\|_{L_p((0, T), H^2_q(\Omega_+))}\\
&\quad+\|\bH^1_+-\bH^2_+\|_{L_\infty((0, T), H^1_q(\Omega_+))}
\|(\bH^1_+, \bH^2_+)\|_{L_p((0, T), H^2_q(\Omega_+))}\\
&\quad +\|\bH^1_+-\bH^2_+\|_{L_p((0, T), H^2_q(\Omega_+))}
\|(\bH^1_+, \bH^2_+)\|_{L_\infty((0, T), H^1_q(\Omega_+))});\\
\|\pd_t\CN_{42}\|_{L_p((0, T), L_q(\Omega_+))} &
\leq C[\|\pd_t\CA_4(\Psi_{\bu^2}]\|_{L_\infty((0, T), H^1_q(\Omega_+))}\\
&\times(\|\Psi_{\bu^1}-\Psi_{\bu^2}\|_{L_\infty((0, T), H^1_q(\Omega_+))}
T^{1/p}\|\nabla\bu^1_+\|_{L_\infty((0, T), L_q(\Omega_+))}\\
&\quad + \|\Psi_{\bu^2}\|_{L_\infty((0, T), H^1_q(\Omega_+))}
T^{1/p}\|\nabla(\bu^1_+-\bu^2_+)\|_{L_\infty((0, T), L_q(\Omega_+))}\\
&\quad+T^{1/p}\|\bH^1_+-\bH^2_+\|_{L_\infty((0, T), H^1_q(\Omega_+))}
\|(\bH^1_+, \bH^2_+)\|_{L_\infty((0, T), L_q(\Omega_+))})\\
&+[A_4(\Psi_{\bu^2})]_{T,2}(T^{1/p}(\|\pd_t(\Psi_{\bu^1}-\Psi_{\bu^2})\|_{L_\infty((0, T), H^1_q(\Omega_+))}
\|\nabla\bu^1_+\|_{L_\infty((0, T), L_q(\Omega_+))}\\
&\quad + \|\pd_t\Psi_{\bu^2}\|_{L_\infty((0, T), H^1_q(\Omega_+))}
\|\nabla(\bu^1_+-\bu^2_+)\|_{L_\infty((0, T), L_q(\Omega_+))})\\
&\quad+\|\pd_t(\bH^1_+-\bH^2_+)\|_{L_p((0, T), L_q(\Omega_+))}
\|(\bH^1_+, \bH^2_+)\|_{L_\infty((0, T), H^1_q(\Omega_+))}\\
&\quad + \|\Psi_{\bu^1}-\Psi_{\bu^2}\|_{L_\infty((0, T), H^1_q(\Omega_+))}
\|\pd_t\nabla\bu^1_+\|_{L_p((0, T), L_q(\Omega_+))}\\
&\quad + \|\Psi_{\bu^2}\|_{L_\infty((0, T), H^1_q(\Omega_+))}
\|\pd_t\nabla(\bu^1_+-\bu^2_+)\|_{L_p((0, T), L_q(\Omega_+))}\\
&\quad+\|(\bH^1_+-\bH^2_+)\|_{L_\infty((0, T), H^1_q(\Omega_+))}
\|(\pd_t\bH^1_+,\pd_t \bH^2_+)\|_{L_p((0, T), L_q(\Omega_+))}].
\end{align*}
Thus, by \eqref{4.4}, \eqref{4.4.2}, \eqref{4.4.3}, \eqref{4.5}, \eqref{4.5.1}, \eqref{nd:1}, 
\begin{align*}
\|\CN_{42}\|_{L_p((0, T), H^2_q(\Omega_+))}
&\leq C\{T^{1/p'}(E^1_T(\bu_+^1) + E^1_T(\bu^2_+))E^1_T(\bu^1_+-\bu^2_+)\\
& + (B + T^{\frac{s}{p'(1+s)}}(E^{2,+}_T(\bH^1_+) + E^{2,+}_T(\bH^2_+)))E^{2,+}_T(\bH^1_+-\bH^2_+)\},\\
\|\pd_t\CN_{42}\|_{L_p((0, T), L_q(\Omega_+))} &\leq C\{T^{1/p}({B + }E^1_T(\bu^1_+)
+ E^1_T(\bu^2_+))E^1_T(\bu^1_+-\bu^2_+)\\
&+{T^{1/p}(B+E^1_T(\bu^2_+))}(B+E^{2,+}_T(\bH^1_+) + E^{2,+}_T(\bH^2_+))
E^{2,+}_T(\bH^1_+-\bH^2_+)\\
&+T^{1/p'}(E^1_T(\bu^1_+)+ E^1_T(\bu^2_+))E^1_T(\bu^1_+-\bu^2_+)\\
&+(B+T^{\frac{s}{p'(1+s)}}(E^{2,+}_T(\bH^1_+) + E^{2,+}_T(\bH^2_+)))E^{2,+}_T(\bH^1_+-\bH^2_+)\},
\end{align*}
which, combined with \eqref{nd:4.1}, yields that
\begin{equation}\label{nd:4}\begin{aligned}
&\|\CN_4\|_{L_p((0, T), H^2_q(\Omega_+))}
+ \|\pd_t\CN_4\|_{L_p((0, T), L_q(\Omega_+))} \\
&\quad \leq C[T^{1/p'}\{
  E^1_T(\bu^1_+) + E^1_T(\bu^2_+) + (B+E^{2,+}_T(\bH^1_+))E^{2,+}_T(\bH^1_+)\}\\
&\hskip3.15cm + T^{1/p}({B+}E^1_T(\bu^1_+) + E^1_T(\bu^2_+) + (B+E^{2,+}_T(\bH^1_+))^2)]
E^1_+(\bu^1_+-\bu^2_+)\\
& \quad+[T^{1/p}(B+E^1_T(\bu^2_+))(B+E^{2,+}_T(\bH^1_+) + E^{2,+}_T(\bH^2_+))\\
&\hskip3.15cm +B+T^{\frac{s}{p'(1+s)}}(E^{2,+}_T(\bH^1_+) + E^{2,+}_T(\bH^2_+))]E^{2,+}_T(\bH^1_+-\bH^2_+)\}
\end{aligned}\end{equation}
where we have used $0 < T < 1$.

We now estimate $\bN_5$ and $\CN_5$.  Applying  \eqref{4.3} to the formula given in \eqref{2.13}, we have
\begin{equation}\label{n:5}\begin{aligned}
&\|\bN_{5\pm}(\bu, \bH)\|_{L_p((0, T), L_q(\Omega_\pm))}
 \leq C[\CA_{5\pm}(\Psi_\bu)]_{T,2}(T^{1/p}\|\bH_\pm\|^2_{L_\infty((0, T), H^1_q(\Omega_\pm))} \\
&\quad  +\|\Psi_\bu\|_{L_\infty((0, T), H^1_q(\Omega_\pm))}\|\bH_\pm\|_{L_p((0, T), H^2_q(\Omega_\pm))}
+ T^{1/p}\|\bu_+\|_{L_\infty((0, T), H^1_q(\Omega_+))}
\|\bH_\pm\|_{L_\infty((0, T), H^1_q(\Omega_+))})\\
& \leq C\{T^{1/p}(B+E^2_T(\bH))(B+E^1_T(\bu_+)+E^2_T(\bH))
+ T^{1/p'}E^1_T(\bu_+)E^2_T(\bH)\}.
\end{aligned}\end{equation}
To estimte $\CN_5$, we write $\CN_{5\pm} = \CN_{51\pm} + \CN_{52\pm}$ with
\begin{align*}
\CN_{51\pm} & = (\CA_{5\pm}(\Psi_{\bu^1})-\CA_{5\pm}(\Psi_{\bu^2}))(K^5_{11}, K^5_{21}, K^5_{31}, K^5_{41}, K^5_{51}), \\
\CN_{52\pm} & = \CA_{5\pm}(\Psi_{\bu^2})(K^5_{11}-K^5_{12}, K^5_{21}-K^5_{22}, K^5_{31}-K^5_{32}, 
K^5_{41}-K^5_{42}, K^5_{51}-K^5_{52}), 
\end{align*}
where $K^5_{1i} = \bH^i_\pm\otimes\nabla \bH^i_\pm$, 
$K^5_{2i} = \Psi_{\bu^i}\otimes\nabla^2\bH_{\pm}^i$, 
$K^5_{3i}=\nabla\Psi_{\bu^i}\otimes\nabla \bH^i_\pm$, 
$K^5_{4i} = \delta_\pm\nabla\bu^i_\pm\otimes\bH^i_\pm$, and 
$K^5_{5i}=\bu^i_\pm\otimes\bH^i_\pm$. 
By \eqref{4.6.1} and \eqref{n:5}, we have
\begin{equation}\label{nd:5.1} \|\CN_{51\pm}\|_{L_p((0, T), L_q(\Omega_\pm))}
\leq CT^{1/p'}(B+E^2_T(\bH^1))(B+E^1_T(\bu^1_+)+E^2_T(\bH^1))E^1_T(\bu^1_+-\bu^2_+),
\end{equation}
where we have used $0 < T < 1$. 
Furthermore, by \eqref{4.3}, \eqref{4.4}, and \eqref{4.5}
\begin{align*}
\|K^5_{11}-K^5_{12}\|_{L_p((0, T), L_q(\Omega_\pm))}
& \leq CT^{1/p}(B+E^2_T(\bH^1) + E^2_T(\bH^2))E^2_T(\bH^1-\bH^2); \\
\|K^5_{21}-K^5_{22}\|_{L_p((0, T), L_q(\Omega_\pm))} &\leq CT^{1/p'}\{E^2_T(\bH^1)E^1_T(\bu^1_+-\bu^2_+)
+E^1_T(\bu^2_+)E^2_T(\bH^1-\bH^2)\}; \\
\|K^5_{31}-K^5_{32}\|_{L_p((0, T), L_q(\Omega_\pm))} &\leq CT^{1/p'}\{E^2_T(\bH^1)E^1_T(\bu^1_+-\bu^2_+)
+E^1_T(\bu^2_+)E^2_T(\bH^1-\bH^2)\}; \\
\|K^5_{41}-K^5_{42}\|_{L_p((0, T), L_q(\Omega_\pm))}
& \leq CT^{1/p}\{(B+E^2_T(\bH^1))E^1_T(\bu^1_+-\bu^2_+) + 
(B+E^1_T(\bu^2_+))E^2_T(\bH^1-\bH^2)\}; \\
\|K^5_{51}-K^5_{52}\|_{L_p((0, T), L_q(\Omega_\pm))}
& \leq CT^{1/p}\{(B+E^2_T(\bH^1))E^1_T(\bu^1_+-\bu^2_+) +
(B+ E^1_T(\bu^2_+))E^2_T(\bH^1-\bH^2)\},
\end{align*}
which, combined with \eqref{4.4.3}, yields that 
\begin{align}
\|\CN_{52\pm}\|_{L_p((0, T), L_q(\Omega_\pm))}
&\leq C\{T^{1/p}(B+E^2_T(\bH^1) + E^2_T(\bH^2) + E^1_T(\bu^2_+))
+ T^{1/p'}E^1_T(\bu^2_+)\}E^2_T(\bH^1-\bH^2) \nonumber \\
&\quad+ C\{T^{1/p'}E^2_T(\bH^1) + 
T^{1/p}(B+E^2_T(\bH^1))\}E^1_T(\bu^1_+-\bu^2_+). 
\label{nd:5.2}
\end{align}
Combining \eqref{nd:5.1} and \eqref{nd:5.2} gives that 
\begin{equation}\label{nd:5}\begin{aligned}
\|\CN_5&\|_{L_p((0, T), L_q(\Omega))}  \leq C\{ T^{1/p'}((B+E^2_T(\bH^1))(B+E^1_T(\bu^1_+) + E^2_T(\bH^1))
+ E^2_T(\bH^1))\\
&\phantom{\|_{L_p((0, T), L_q(\Omega))}  \leq}\,\,
+ T^{1/p}(B+E^2_T(\bH^1))\}E^1_T(\bu^1_+-\bu^2_+)\\
&\quad\, +C\{T^{1/p'}E^1_T(\bu^2_+)+ T^{1/p}(B+E^2_T(\bH^1) + E^2_T(\bH^2) + E^1_T(\bu^2_+))\}
E^2_T(\bH^1-\bH^2). 
\end{aligned}\end{equation}

We now consider $\bN_6$ and $\CN_6$. We have to extend them to $t < 0$. 
{For this purpose, Let $E_\mp$ be an extension operator satisfying \eqref{jump:1}}. 
In view of \eqref{2.16}, we have
\begin{equation}\label{2.16*}\begin{aligned}
\bN_6(\bu, \bH) &= \CA_{61}(\Psi_\bu)(\alpha_+^{-1}\nabla E_-[\bH_+]-\alpha^{-1}_-\nabla E_+[\bH_-])
+ (\CB+\CA_{62}(\Psi_\bu))\nabla E_-[\bu_+]\otimes \nabla E_-[\bH_+]; \\
\CN_6 & = (\CA_{61}(\Psi_{\bu^1})-\CA_{61}(\Psi_{\bu^2}))
(\alpha_+^{-1}E_-[\bH_+^1]-\alpha^{-1}_-E_+[\bH_-^1])\\
&+( \CA_{62}(\Psi_{\bu^1})- \CA_{62}(\Psi_{\bu^2}))E_-[\bu_+^1]\otimes E_-[\bH_+^1]\\
& + \CA_{61}(\Psi_{\bu^2})(\alpha_+^{-1}\nabla (E_-[\bH_+^1]-E_-[\bH^2_+])-\alpha^{-1}_-
\nabla (E_+[\bH_-^1]-E_+[\bH_-^2]))\\
&+(\CB+\CA_{62}(\Psi_{\bu^2}))((E_-[\bu_+^1]-E_-[\bu_+^2])\otimes E_-[\bH_+^1]
+ E_-[\bu_+^2]\otimes(E_-[\bH_+^1]-E_-[\bH_+^2])).
\end{aligned}\end{equation}
on $\Gamma\times(0, T)$. 
{Define the extension operator $e_T$ by \eqref{4.8} 
and let $E_\pm$ and $E_{\BR^N}$ be extension operators 
satisfying \eqref{jump:1} and \eqref{4.11}, respectively. 
Let $\gamma_1$ and $\gamma_2$ be the positive constants appearing in 
respective Theorem \ref{thm:3.1.1} and Theorem \ref{thm:3.2.1} 
and set $\gamma_0=\max(\gamma_1,\gamma_2)$ below. 
Instead of \eqref{semigroup}, we set 
\begin{align*}
T(t)\bv_0=e^{(-\gamma_0-1+\Delta)t}\bv_0. 
\end{align*}
Then let $\CT$ and $\CE$ be extension operators 
satisfying \eqref{4.13*} and \eqref{4.14}, respectively. For $\ell=0,1$, we obtain 
\begin{equation*}\begin{aligned}
\|e^{\gamma_0t}T(\cdot)\bv_0\|_{L_p((0, \infty), H^{\ell+2}_q(\BR^N))}
+ \|e^{\gamma_0t}\pd_tT(\cdot)\bv_0\|_{L_p((0, \infty), H^\ell_q(\BR^N))}
\leq C\|\bv_0\|_{B^{\ell+2(1-1/p)}_{q,p}(\BR^N)} 
\end{aligned}\end{equation*}
and 
\begin{align}
\CT[E_\mp[\bw_{0\pm}]](0) & = E_\mp[\bw_{0\pm}] \quad\text{in $\Omega$},  
\nonumber\\
\|e^{\gamma_0t}\CT[E_\mp[\bw_{0\pm}]]\|_{L_p(\BR, H^{2+\ell}_q(\Omega))}
&+ \|e^{\gamma_0t}\pd_t\CT[E_\mp[\bw_{0\pm}]]\|_{L_p(\BR, H^\ell_q(\Omega))}
\leq C\|\bw_0\|_{B^{2(1-1/p)+\ell}_{q,p}(\dot\Omega)}.\label{4.13}
\end{align}}
{We now define an extension $\tilde \bN_6(\bu, \bH)$. Let $\tilde \bN_6(\bu, \bH)
= \tilde \bN_{61}(\bu, \bH) + \tilde \bN_{62}(\bu, \bH)$ with 
\begin{align*}
\tilde\bN_{61}(\bu, \bH) &= \CA_{61}(e_T[\Psi_\bu])(\alpha_+^{-1}\nabla\CE[E_-[\bH_+]]-\alpha^{-1}_-
\nabla\CE[E_+[\bH_-]]);\\
\tilde\bN_{62}(\bu, \bH) & = (\CB+\CA_{62}(e_T[\Psi_\bu]))\CE[E_-[\bu_+]]\otimes \CE[E_-[\bH_+]].
\end{align*}


To estimate $H^{1/2}_p(\BR, L_q(\dot\Omega))$ norm,   we use
the following {lemma}.
\begin{lem}\label{lem:4.5.3} Let $1 < p < \infty$ and $N < q < \infty$.
Let 
\begin{align*}
f  \in L_\infty(\BR, H^1_q(\dot\Omega)) \cap H^1_\infty(\BR, L_q(\dot\Omega)), \quad 
g \in H^{1/2}_p(\BR, H^1_q(\dot\Omega)) \cap L_p(\BR, H^1_q(\dot\Omega)).
\end{align*}
Then, we have
\begin{align*}
&\|fg\|_{H^{1/2}_p(\BR, L_q(\dot\Omega))} + \|fg\|_{L_p(\BR, H^1_q(\dot\Omega))}\\
&
\leq C(\|\pd_tf\|_{L_\infty(\BR, L_q(\dot\Omega))}+ \|f\|_{L_\infty(\BR, H^1_q(\dot\Omega))})^{1/2}
\|f\|_{L_\infty(\BR, H^1_q(\dot\Omega))}^{1/2}
(\|g\|_{H^{1/2}_p(\BR, L_q(\dot\Omega))}
+ \|g\|_{L_p(\BR, H^1_q(\dot\Omega))}).
\end{align*}
\end{lem}
\begin{proof} To prove Lemma \ref{lem:4.5.3}, we use the fact that
$$H^{1/2}_p(\BR, L_q(\dot\Omega)) \cap L_p(\BR, H^{1/2}_q(\dot\Omega))
=(L_p(\BR, L_q(\dot\Omega)), 
H^1_p(\BR, L_q(\dot\Omega)) \cap L_p(\BR, H^1_q(\dot\Omega)))_{[1/2]}
$$
where $(\cdot, \cdot)_{[1/2]}$ denotes a  complex interpolation functor of order
$1/2$. We have
\begin{align*}
&\|fg\|_{H^1_p(\BR, L_q(\dot\Omega))} + 
\|fg\|_{L_p(\BR, H^1_q(\dot\Omega))}\\
&\quad \leq C(\|\pd_tf\|_{L_\infty(\BR, L_q(\dot\Omega))}
\|g\|_{L_p(\BR, H^1_q(\dot\Omega))}
+ \|f\|_{L_\infty(\BR, H^1_q(\dot\Omega))}
\|g\|_{H^1_p(\BR, L_q(\dot\Omega))})\\
&\quad \leq C(\|\pd_tf\|_{L_\infty(\BR, L_q(\dot\Omega))}
+ \|f\|_{L_\infty(\BR, H^1_q(\dot\Omega))})
(\|g\|_{L_p(\BR, H^1_q(\dot\Omega))}
+\|g\|_{H^1_p(\BR, L_q(\dot\Omega))}).
\end{align*}
Moreover, 
$$\|fg\|_{L_p(\BR, L_q(\dot\Omega))}
\leq C\|f\|_{L_p(\BR, H^1_q(\dot\Omega))}\|g\|_{L_p(\BR, L_q(\dot\Omega))}.
$$
Thus, by complex interpolation, we have
\begin{align*}
&\|fg\|_{H^{1/2}_p(\BR, L_q(\dot\Omega))} + \|fg\|_{L_p(\BR, H^{1/2}_q(\dot\Omega))}\\
&
\leq C(\|\pd_tf\|_{L_\infty(\BR, L_q(\dot\Omega))}+ \|f\|_{L_\infty(\BR, H^1_q(\dot\Omega))})^{1/2}
\|f\|_{L_\infty(\BR, H^1_q(\dot\Omega))}^{1/2}
(\|g\|_{H^{1/2}_p(\BR, L_q(\dot\Omega))}
+ \|g\|_{L_p(\BR, H^{1/2}_q(\dot\Omega))}).
\end{align*}
Moreover, we have
$$\|fg\|_{L_p(\BR, H^1_q(\dot\Omega))} \leq C\|f\|_{L_\infty(\BR, H^1_q(\dot\Omega))}
\|g\|_{L_p(\BR, H^1_q(\dot\Omega))}.
$$
Thus, combining these two inequalities give the required estimate, which completes
 the proof of Lemma \ref{lem:4.5.3}. 
\end{proof}

\begin{lem}\label{lem:4.5.4}
Let $1 < p, q < \infty$.  Then, 
$$H^1_p(\BR, L_q(\dot\Omega)) \cap L_p(\BR, H^2_q(\dot\Omega))
\subset H^{1/2}_p(\BR, H^1_q(\dot\Omega))$$
and 
$$\|u\|_{H^{1/2}_p(\BR, H^1_q(\dot\Omega))} \leq C(\|u\|_{L_p(\BR, H^2_q(\dot\Omega))}
+ \|\pd_tu\|_{L_p(\BR, L_q(\dot\Omega))}).
$$
\end{lem}
\begin{proof} For a proof, see Shibata \cite[Proposition 1]{S2}. 
\end{proof}

By \eqref{2.1**}, we may assume that
\begin{equation}\label{2.1a}
\|e_T[\Psi_\bw]\|_{L_\infty(\BR, L_\infty(\Omega))} \leq \delta \quad\text{for $\bw \in \{\bu, \bu^1, \bu^2\}$}.
\end{equation}
By \eqref{4.4}, \eqref{mat:1} and the same argument as in proving \eqref{4.6.1}, we have
\begin{equation}\label{e4.4}\begin{aligned}
\|\CA_{6i}(e_T[\Psi_\bw])\|_{L_\infty(\BR, H^2_q(\Omega)} 
&\leq CT^{1/p'}{E^1_T(\bw_+)} \hspace{10mm} \text{for $\bw \in \{\bu, \bu^1, \bu^2\}$};
\\
\|\CA_{6i}(e_T[\Psi_{\bu^1}])-\CA_{6i}(e_T[\Psi_{\bu^2}])\|_{L_\infty(\BR, H^2_q(\Omega))}
&\leq CT^{1/p'}E^1_T(\bu^1_+-\bu^2_+). 
\end{aligned}\end{equation}
Employing the same as in proving \eqref{4.6.3} yields 
\begin{equation}\label{e4.15}\begin{aligned}
\|\pd_t(\CA_{6i}(e_T[\Psi_\bu])\|_{L_\infty(\BR, H^1_q(\Omega))}
& \leq C(B+E^1_T(\bu_+)), \\
\|\pd_t(\CA_{6i}(e_T[\Psi_{\bu^1}])-\CA_{6i}(e_T[\Psi_{\bu^2}))\|_{L_\infty(\BR, H^1_q(\Omega))}
 & \leq CE^1_T(\bu^1_+-\bu^2_+), 
\end{aligned}\end{equation}
and so noting that $e_T[\Psi_\bu]$  vanishes for $t \not \in (0, 2T)$,
by Lemma \ref{lem:4.5.3}, \eqref{e4.4} and \eqref{e4.15}, we have
\begin{align*}
&\|e^{-\gamma t}\tilde\bN^1_6(\bu, \bH)\|_{H^{1/2}_p(\BR, L_q(\Omega))}
+ \|e^{-\gamma t}\tilde\bN^1_6(\bu, \bH)\|_{L_p(\BR, H^1_q(\Omega))}
+ \gamma^{1/2}\|e^{-\gamma t}\tilde\bN^1_6(\bu, \bH)\|_{L_p(\BR, L_q(\Omega))}\\
&\quad \leq CT^{1/(2p')}(B+E^1_T(\bu_+)) \\
&\qquad\times
(\|\nabla(\CE[E_-[\bH_+]], \CE[E_+[\bH_-]])\|_{H^{1/2}_p(\BR, L_q(\Omega))}
+ \|\nabla(\CE[E_-[\bH_+]], \CE[E_+[\bH_-]])\|_{L_p(\BR, H^1_q(\Omega))})
\end{align*}
Thus, applying Lemma \ref{lem:4.5.4}, 
\eqref{4.14}, \eqref{4.13}, \eqref{4.10}, and \eqref{jump:1}, we have 
\begin{equation}\label{non:4.2}\begin{aligned}
\|e^{-\gamma t}\tilde\bN_{61}(\bu, \bH)\|_{H^{1/2}_p(\BR, L_q(\Omega))}
+ \|e^{-\gamma t}\tilde\bN_{61}(\bu, \bH)\|_{L_p(\BR, H^1_q(\Omega))}
+ \gamma^{1/2}\|e^{-\gamma t}\tilde\bN^1_6(\bu, \bH)\|_{L_p(\BR, L_q(\Omega))}\\
\leq CT^{1/(2p')}(B+E^1_T(\bu_+))(B+E^2_T(\bH)). 
\end{aligned}\end{equation}

We now estimate $\bN_6^2(\bu, \bH)$. For this purpose we use the following
esitmate:
\begin{equation}\label{complex:1}
\|f\|_{H^{1/2}_p(\BR, L_q(\dot\Omega))}
\leq C\|f\|_{H^1_p(\BR, L_q(\dot\Omega))}^{1/2}
\|f\|_{L_p(\BR, L_q(\dot\Omega))}^{1/2}, 
\end{equation}
which follows from  complex interpolation theory.  Let 
\begin{align*}
A^2_1 & = \CE[E_-[\bu_+]]\otimes  \CE[E_-[\bH_+]], 
\quad 
A^2_2  = \CA_{62}(\Psi_\bu)A^2_1.
\end{align*}
And then, $\tilde \bN_{62}(\bu, \bH) = {\mathcal{B}}A^2_1 + A^2_2$. 
We further divide $A^2_1$ into 
$A^2_1 = A^2_{11} + A^2_{12} + A^2_{21} + A^2_{22}$ with 
\allowdisplaybreaks
\begin{align*}
A^2_{11} &= \CT[E_-[\bu_{0+}]]\otimes \CT[E_-[\tilde\bH_{0+}]], \\
A^2_{12} &=\CT[E_-[\bu_{0+}]]\otimes e_T[E_-[\bH_+]-\CT[E_-[\tilde\bH_{0+}]]],  \\
A^2_{21} & = e_T[E_-[\bu_+] - \CT[E_-[\bu_{0+}]]]\otimes \CT[E_-[\tilde\bH_{0+}]], \\
A^2_{22} &= e_T[E_-[\bu_+] - \CT[E_-[\bu_{0+}]]]\otimes  e_T[E_-[\bH_+]-\CT[E_-[\tilde\bH_{0+}]]]
\end{align*}
By  \eqref{initial:4.1},  \eqref{4.3}, \eqref{4.5}, and \eqref{4.13}, 
\begin{align*}
&\|e^{-\gamma t}A^2_{11}\|_{H^i_p(\BR, L_q(\Omega_+))} \\
&\quad\leq C(\|e^{-\gamma t}\CT[E_-[\bu_{0+}]]\|_{H^i_p(\BR, L_q(\Omega_+))}
\|\CT[E_-[\tilde\bH_{0+}]]\|_{L_\infty(\BR, H^1_q(\Omega_+)} \\
&\hspace{20mm}+\|\CT[E_-[\bu_{0+}]]\|_{L_\infty(\BR, H^1_q(\Omega_+)}
\|e^{-\gamma t}\CT[E_-[\tilde\bH_{0+}]]\|_{H^i_p(\BR, L_q(\Omega_+))}) \\
&\quad\leq C{e}^{2(\gamma-\gamma_0)}B^2; \\
&\|e^{-\gamma t}\pd_tA^2_{12}\|_{L_p(\BR, L_q(\Omega_+))}\\
&\quad\leq C\{\|\pd_t\CT[E_-[\bu_{0+}]]\|_{L_p(\BR, L_q(\Omega_+))}
\|e_T[E_-[\bH_+] - \CT[E_-[\tilde\bH_{0+}]]]\|_{L_\infty(\BR, H^1_q(\Omega_+))}\\
&\quad+ \|\CT[E_-[\bu_{0+}]]\|_{L_\infty(\BR, H^1_q(\Omega_+))}
\|\pd_te_T[E_-[\bH_+] - \CT[E_-[\tilde\bH_{0+}]]]\|_{L_p(\BR, L_q(\Omega_+))}\}\\
&\quad\leq CB(B+E^{2,+}_T(\bH_+)); \\
&\|e^{-\gamma t}A^2_{12}\|_{L_p(\BR, L_q(\Omega_+))}\\
&\quad  \leq
T^{1/p}\|\CT[E_-[\bu_{0+}]]\|_{L_\infty(\BR, H^1_q(\Omega_+))}
\|e_T[E_-[\bH_+] - \CT[E_-[\tilde\bH_{0+}]]]\|_{L_\infty((0, 2T), L_q(\Omega_+))}\\
&\quad\leq CBT\|\pd_te_T[E_-[\bH_+] 
- \CT[E_-[\tilde\bH_{0+}]]\|_{L_p((0, 2T), L_q(\Omega_+))}\\
&\quad \leq CTB(B+E^{2,+}_T(\bH_+)), \\
&\|e^{-\gamma t}\pd_tA^2_{21}\|_{L_p(\BR, L_q(\Omega_+))}
\leq CB(B+E^{1}_T(\bu_+)); \\
&\|e^{-\gamma t}A^2_{21}\|_{L_p(\BR, L_q(\Omega_+))}
 \leq CTB(B+E^{1}_T(\bu_+)), \\
&\|e^{-\gamma t}\pd_tA^2_{22}\|_{L_p(\BR, L_q(\Omega_+))} \\
&\quad \leq C\{\|\pd_t e_T[E_-[\bu_+] - \CT[E_-[\tilde\bu_{0+}]]]\|_{L_p((0, 2T), L_q(\Omega_+))}
\|e_T[E_-[\bH_+] - \CT[E_-[\tilde\bH_{0+}]]]\|_{L_\infty((0, 2T), H^1_q(\Omega_+))}\\
&\quad +\| e_T[E_-[\bu_+] - \CT[E_-[\tilde\bu_{0+}]]]
\|_{L_\infty((0, 2T), H^1_q(\Omega_+))}
\|\pd_te_T[E_-[\bH_+] - \CT[E_-[\tilde\bH_{0+}]]]\|_{L_p((0, 2T), L_q(\Omega_+))}\\
&\quad \leq C(B+E^1_T(\bu_+))(B+E^{2,+}_T(\bH_+)); \\
&\|e^{-\gamma t}A^2_{22}\|_{L_p(\BR, L_q(\Omega_+))} \\
&\quad  \leq CT^{1/p}\| e_T[E_-[\bu_+] 
- \CT[E_-[\tilde\bu_{0+}]]]\|_{L_\infty((0, 2T), H^1_q(\Omega_+))}
\|e_T[E_-[\bH_+] - 
\CT[E_-[\tilde\bH_{0+}]]]\|_{L_\infty((0, 2T), L_q(\Omega_+))}\\
&\quad  \leq C(B+E^{1}_T(\bu_+))T
\|\pd_t e_T[E_-[\bH_+] - \CT[E_-[\tilde\bH_{0+}]]]\|_{L_p((0, 2T), L_q(\Omega_+))}\\
&\quad  \leq CT(B+E^1_T(\bu_+))(B+E^{2,+}_T(\bH_+))). 
\end{align*}
Using \eqref{complex:1}, 
we have
\begin{align}\label{non:4.4*}\begin{aligned}
\|e^{-\gamma t}A^2_1\|_{H^{1/2}_p(\BR, L_q(\Omega_+))}
&\leq C\sum_{i,j=1}^{2} \|e^{-\gamma t}A_{ij}\|_{H^1_p(\BR,L_q(\dot\Omega))}^{1/2}
\|e^{-\gamma t}A_{ij}\|_{L_p(\BR,L_q(\dot\Omega))}^{1/2} \\
&\leq C(e^{2(\gamma-\gamma_0)}B^2 +T^{1/2}(B+E^1_T(\bu_+))(B+E^{2,+}_T(\bH_+))). 
\end{aligned}\end{align}
And also, by \eqref{4.3}, \eqref{4.5}, \eqref{4.13}, and \eqref{initial:4.1}, 
\begin{align*}
\|e^{-\gamma t}A^2_{11}\|_{L_p(\BR, H^1_q(\Omega_+))}
& \leq C\|e^{-\gamma t}\CT[E_-[\bu_{0+}]]\|_{L_p(\BR, H^1_q(\Omega_+))}
\|\CT[E_-[\tilde\bH_{0+}]]\|_{L_\infty(\BR, H^1_q(\Omega_+))}\\
& \leq Ce^{2(\gamma-\gamma_0)}B^2; \\
\|e^{-\gamma t}A^2_{12}\|_{L_p(\BR, H^1_q(\Omega_+))} 
& \leq C\|\CT[E_-[\bu_{0+}]]\|_{L_\infty(\BR, H^1_q(\Omega_+))}
\|e_T[E_-[\bH_+]-\CT[E_-[\tilde\bH_{0+}]]]\|_{L_p((0, 2T), H^1_q(\Omega_+))} \\
&\leq CBT^{1/p}\|e_T[E_-[\bH_+]-\CT[E_-[\tilde\bH_{0+}]]]\|_{L_\infty((0, 2T), H^1_q(\Omega_+))}\\
& \leq CB(B+E^{2,+}_T(\bH_+))T^{1/p}; \\
\|e^{-\gamma t}A^2_{21}\|_{L_p(\BR, H^1_q(\Omega_+))} 
& \leq CB(B+E^1_T(\bu_+))T^{1/p};  \\
\|e^{-\gamma t}A^2_{22}\|_{L_p(\BR, H^1_q(\Omega_+))} 
& \leq C\|e_T[E_-[\bH_+]-\CT[E_-[\tilde\bH_{0+}]]]\|_{L_p((0, 2T), H^1_q(\Omega_+))}\\
&\quad\times
\|e_T[E_-[[\bu_+]-\CT[E_-[\bu_{0+}]]]\|_{L_\infty((0, 2T), H^1_q(\Omega_+))}
\\
&\leq CT^{1/p}(B+E^1_T(\bu_+))(B+E^{2,+}_T(\bH_+)),
\end{align*}
and so we have
$$\|e^{-\gamma t}A^2_1\|_{L_p(\BR, H^1_q(\Omega_+))}
\leq C(e^{2(\gamma-\gamma_0)}B^2 + T^{1/p}(B+E^1_T(\bu_+))(B+E^{2,+}_T(\bH_+))),
$$
which, combined with \eqref{non:4.4*}, yields that
\begin{equation}\label{non:4.4}\begin{aligned}
&\|e^{-\gamma t}A^2_1\|_{H^{1/2}_p(\BR, L_q(\Omega_+))}
+ \|e^{-\gamma t}A^2_1\|_{L_p(\BR, H^1_q(\Omega_+))}
+ \gamma^{1/2}\|e^{-\gamma t}A^2_1\|_{L_p(\BR, L_q(\Omega_+))}\\
&\quad
\leq C\{e^{2(\gamma-\gamma_0)}B^2 + (T^{1/2}+T^{1/p})(B+E^1_+(\bu_+))(B+E^{2,+}_T(\bH_+))\}.
\end{aligned}\end{equation}
By \eqref{e4.4}, \eqref{e4.15}, \eqref{non:4.4}, and Lemma \ref{lem:4.5.3}, 
\begin{equation}\label{non:4.5}\begin{aligned}
&\|e^{-\gamma t}A^2_2\|_{H^{1/2}_p(\BR, L_q(\Omega_+))}
+ \|e^{-\gamma t}A^2_2\|_{L_p(\BR, H^1_q(\Omega_+))} 
+ \gamma^{1/2}\|e^{-\gamma t}A^2_2\|_{L_p(\BR, L_q(\Omega_+))}\\
&\quad \leq C((B+E^1_T(\bu_+))T^{1/(2p')} {)} 
(\|A^2_1\|_{H^{1/2}_p(\BR, L_q(\Omega))} + \|A^2_1\|_{L_p(\BR, H^1_q(\Omega_+))})
\\
& \quad\leq C((B+E^1_+(\bu_+))T^{1/(2p')})
\{e^{2(\gamma-\gamma_0)}B^2 + (T^{1/2}+T^{1/p})(B+E^1_+(\bu_+))(B+E^{2,+}_T(\bH_+))\}.
\end{aligned}\end{equation}
Combining \eqref{non:4.2}, \eqref{non:4.4}, and \eqref{non:4.5} yields that
\begin{equation}\label{n:6}\begin{aligned}
&\|\tilde\bN_6(\bu, \bH)\|_{H^{1/2}_p(\BR, L_q(\Omega))} + 
\|\tilde\bN_6(\bu, \bH)\|_{L_p(\BR, H^1_q(\Omega))} 
+\gamma^{1/2}\|\tilde\bN_6(\bu, \bH)\|_{L_p(\BR, L_q(\Omega))}\\
&\quad \leq C[T^{1/(2p')}(B+E^1_T(\bu_+))(B+E^2_T(\bH))\\
&\quad + (1+(B+E^1_T(\bu_+))T^{1/(2p')})
\{e^{2(\gamma-\gamma_0)}B^2 + (T^{1/2}+T^{1/p})(B+E^1_+(\bu_+))(B+E^{2,+}_T(\bH_+))\}].
\end{aligned}\end{equation}
Recalling the formula of $\CN_6$ in \eqref{2.16*}, we define an extension, $\tilde \CN_6$, of $\CN_6$ by setting
$\tilde\CN_6 = \tilde\CN_{61} + \tilde\CN_{62} + \tilde\CN_{63} + \tilde\CN_{64}$ with 
\begin{align*}
\tilde\CN_{61} & = (\CA_{61}(e_T[\Psi_{\bu^1}])-\CA_{61}(e_T[\Psi_{\bu^2}]))
(\alpha_+^{-1}\nabla\CE[E_-[\bH_+^1]]-\alpha^{-1}_-\nabla\CE[E_+[\bH_-^1]])\\
\tilde \CN_{62} & = ( \CA_{62}(e_T[\Psi_{\bu^1}])- \CA_{62}(e_T[\Psi_{\bu^2}]))
\CE[E_-[\bu_+^1]]\otimes \CE[E_-[\bH_+^1]]\\
\tilde\CN_{63}&= \CA_{61}(e_T[\Psi_{\bu^2}])\{\alpha_+^{-1}\nabla(\CE[E_-[\bH_+^1]]-
\CE[E_-[\bH^2_+]])-\alpha^{-1}_-\nabla
(\CE[E_+[\bH_-^1]]-\CE[E_+[\bH_-^2]])\}\\
\tilde\CN_{64} & = (\CB+\CA_{62}(e_T[\Psi_{\bu^2}]))((\CE[E_-[\bu_+^1]]-
\CE[E_-[\bu_+^2]])\otimes \CE[E_-[\bH_+^1]]\\
&\phantom{ (\CB+\CA_{62}(e_T[\Psi_{\bu^2}]))}\,\,\,\,
+ \CE[E_-[\bu_+^2]]\otimes(\CE[E_-[\bH_+^1]]-\CE[E_-[\bH_+^2]])).
\end{align*}
Setting $\CD_{6i} = \CA_{6i}(e_T[\Psi_{\bu^1}]) - \CA_{6i}(e_T[\Psi_{\bu^2}])$ for 
notational simplicity, by \eqref{e4.4} and \eqref{e4.15} we have
\begin{equation}\label{4.16}
(\|\pd_t\CD_{6i}\|_{L_\infty(\BR, L_q(\dot\Omega))} + 
\|\CD_{6i}\|_{L_\infty(\BR, H^1_q(\dot\Omega))})^{1/2}
\|\CD_{6i}\|_{L_\infty(\BR, H^1_q(\dot\Omega))}^{1/2} \leq CT^{1/(2p')}E^1_T(\bu^1_+-\bu^2_+).
\end{equation}
 And, noting  Lemma \ref{lem:4.5.4}, we have 
$$\|\nabla \CE[E_\mp[\bH^1_\pm]]\|_{H^{1/2}_p(\BR, L_q(\Omega))}
+ \|\nabla \CE[E_\mp[\bH^1_\pm]]\|_{L_p(\BR, H^1_q(\Omega))} \leq C(B+E^2_T(\bH^1)),
$$
which, combined with Lemma \ref{lem:4.5.3}  
and \eqref{4.16}, yields that
\begin{equation}\label{nd.6.1} \begin{aligned}
\|e^{-\gamma t}\tilde\CN_{61}\|_{H^{1/2}_p(\BR, L_q(\dot\Omega))}
+ \|e^{-\gamma t}\tilde\CN_{61}\|_{L_p(\BR, H^1_q(\dot\Omega))}
+ \gamma^{1/2}\|e^{-\gamma t}\tilde\CN_{61}\|_{L_p(\BR, L_q(\dot\Omega))} \\
\leq CT^{1/(2p')}(B+E^2_T(\bH^1))E^1_T(\bu^1_+-\bu^2_+).
\end{aligned}\end{equation}
Analogously, by \eqref{4.16}, Lemma \ref{lem:4.5.3} and 
\eqref{non:4.4}, we have
\begin{equation}\label{nd.6.2} \begin{aligned}
&\|e^{-\gamma t}\tilde\CN_{62}\|_{H^{1/2}_p(\BR, L_q(\dot\Omega))}
+ \|e^{-\gamma t}\tilde\CN_{62}\|_{L_p(\BR, H^1_q(\dot\Omega))}
+ \gamma^{1/2}\|e^{-\gamma t}\tilde\CN_{62}\|_{L_p(\BR, L_q(\dot\Omega))} \\
&\quad \leq CT^{1/(2p')}\{e^{2(\gamma-\gamma_0)}B^2+(T^{1/2}+T^{1/p})(B+E^1_T(\bu^1_+))
(B+E^{2,+}_T(\bH^1_+))\}E^1_T(\bu^1_+-\bu^2_+).
\end{aligned}\end{equation}
Since
\begin{equation}\label{4.16*}
\CE[E_\mp[\bH^1_\pm]]-\CE[E_\mp[\bH^2_\pm]]
= e_T[\bH^1_\pm - \bH^2_\pm], 
\quad \CE[E_-[\bu^1_+]]-\CE[E_-[\bu^2_+]]=e_T(\bu^1_+ - \bu^2_+)
\end{equation}
as follows from \eqref{4.14}, by Lemma \ref{lem:4.5.3}, Lemma \ref{lem:4.5.4},
\eqref{e4.4}, and \eqref{e4.15}, we have
\begin{align}\label{nd.6.3}\begin{aligned}
\|e^{-\gamma t}\tilde\CN_{63}\|_{H^{1/2}_p(\BR, L_q(\dot\Omega))} + 
\|e^{-\gamma t}\tilde\CN_{63}\|_{L_p(\BR, H^1_q(\dot\Omega))} +
\gamma^{1/2}\|e^{-\gamma t}\tilde\CN_{63}\|_{L_p(\BR, L_q(\dot\Omega))} \\
\leq CT^{1/(2p')}(B+E^1_T(\bu^2_+))E^2_T(\bH^1-\bH^2).
\end{aligned}\end{align}
In view of \eqref{4.16*},  by \eqref{4.9} we have 
\begin{align*}
&\|e^{-\gamma t}(\CE[E_\mp[\bH^1_\pm]]-\CE[E_\mp[\bH^2_\pm]])\|_{L_p(\BR, L_q(\dot\Omega))}
\leq CT^{1/p}\|e_T[\bH^1_\pm - \bH^2_\pm]\|_{L_\infty((0, 2T), L_q(\dot\Omega))}\\
&\quad \leq CT\|\pd_t(e_T[\bH^1_\pm - \bH^2_\pm]\|_{L_p((0, 2T), L_q(\dot\Omega))}
\leq CTE^{2,\pm}_T(\bH^1_\pm - \bH^2_\pm). 
\end{align*}
Analogously, 
\begin{align*}\|e^{-\gamma t}(\CE[E_-[\bu^1_+]]-\CE[E_-[\bu^2_+]])\|_{L_p(\BR, L_q(\dot\Omega))}
&\leq CTE^1_T(\bu^1_+-\bu^2_+);\\
\|e^{-\gamma t}(\CE[E_\mp[\bH^1_\pm]]-\CE[E_\mp[\bH^2_\pm]])\|_{L_p(\BR, H^1_q(\dot\Omega))}
& \leq CT^{1/p}\|e_T[\bH^1_\pm - \bH^2_\pm]\|_{L_\infty((0, 2T), H^1_q(\dot\Omega))}\\
&\leq CT^{1/p}E^{2,\pm}_T(\bH^1_\pm-\bH^2_\pm); \\
\|e^{-\gamma t}(\CE[E_-[\bu^1_+]]-\CE[E_-[\bu^2_+]])\|_{L_p(\BR, H^1_q(\dot\Omega))}
&\leq CT^{1/p}E^1_T(\bu^1_+-\bu^2_+).
\end{align*}
Thus, by \eqref{complex:1},
\begin{align*}
\|e^{-\gamma t}(\CE[E_-[\bu^1_+]]-\CE[E_-[\bu^2_+]])
\otimes \CE[E_-[\bH^1_+]]\|_{H^{1/2}_p(\BR, L_q(\dot\Omega))}
\leq CT^{1/2}(B+E^{2,+}_T(\bH^1_+))E^1_T(\bu^1_+-\bu^2_+):\\
\|e^{-\gamma t}\CE[E_-[\bu^1_+]]\otimes
(\CE[E_-[\bH^1_+]]-\CE[E_-[\bH^2_+]])\|_{H^{1/2}_p(\BR, L_q(\dot\Omega))}
\leq CT^{1/2}(B+E^1_T(\bu^2_+))E^{2,+}_T(\bH^1_+-\bH^2_+):\\
\|e^{-\gamma t}(\CE[E_-[\bu^1_+]]-\CE[E_-[\bu^2_+]])
\otimes \CE[E_-[\bH^1_+]]\|_{L_p(\BR, H^1_q(\dot\Omega))}
\leq CT^{1/p}(B+E^{2,+}_T(\bH^1_+))E^1_T(\bu^1_+-\bu^2_+):\\
\|e^{-\gamma t}\CE[E_-[\bu^1_+]]\otimes
(\CE[E_-[\bH^1_+]]-\CE[E_-[\bH^2_+]])\|_{L_p(\BR, H^1_q(\dot\Omega))}
\leq CT^{1/p}(B+E^1_T(\bu^2_+))E^{2,+}_T(\bH^1_+-\bH^2_+),
\end{align*}
which, combined with Lemma \ref{lem:4.5.3}, \eqref{e4.4}, and \eqref{e4.15}, yields that 
\begin{equation}\label{nd.6.4}\begin{aligned}
&\|e^{-\gamma t}\tilde\CN_{64}\|_{H^{1/2}_p(\BR, L_q(\dot\Omega))}
+ \|e^{-\gamma t}\tilde\CN_{64}\|_{L_p(\BR, H^1_q(\dot\Omega))}
+ \gamma^{1/2}\|e^{-\gamma t}\tilde\CN_{64}\|_{L_p(\BR, L_q(\dot\Omega))}\\
&\quad\leq C(1+T^{1/(2p')}(B+E^1_T(\bu^2_+)))\\
&\qquad \times (T^{1/2}+T^{1/p})((B+E^{2,+}_T(\bH^1_+))E^1_T(\bu^1_+-\bu^2_+)
+ (B+E^1_T(\bu^2_+))E^{2,+}_T(\bH^1_+-\bH^2_+)).
\end{aligned}\end{equation}
Combining \eqref{nd.6.1}, \eqref{nd.6.2}, \eqref{nd.6.3}, and \eqref{nd.6.4} yields that
\begin{equation}\label{nd:6}\begin{aligned}
&\|e^{-\gamma t}\tilde\CN_6\|_{H^{1/2}_p(\BR, L_q(\dot\Omega))}
+ \|e^{-\gamma t}\tilde\CN_6\|_{L_p(\BR, H^1_q(\dot\Omega))}
+ \gamma^{1/2}\|e^{-\gamma t}\tilde\CN_6\|_{L_p(\BR, L_q(\dot\Omega))}\\
& \quad\leq C\{T^{1/(2p')}(B+E^2_T(\bH^1) +e^{2(\gamma-\gamma_0)}B^2)\\
&\qquad+(T^{1/2}+T^{1/p})\{1+T^{1/(2p')}(B + E^1_T(\bu^1_+)+E^1_T(\bu^2_+))\}(B+E^{2,+}_T(\bH^1_+))\}
E^1_T(\bu^1_+-\bu^2_+)\\
&\quad + C\{T^{1/(2p')}(B+E^1_T(\bu^2_+)) \\
&\qquad+ (T^{1/2}+T^{1/p})(1+T^{1/(2p')}(B+E^1_T(\bu^2_+)))(B+E^1_T(\bu^2_+))\}
E^{2}_T(\bH^1-\bH^2).
\end{aligned}\end{equation}
We now consider $N_7$ and $\CN_7$.  In view of \eqref{2.18}, we define extensions of 
$N_7$ and $\CN_7$ to $\BR\setminus(0, T)$ by setting
\begin{align*}
\tilde N_7(\bu, \bH)  = &\,\,\CA_7(e_T[\Psi_\bu])\nabla(\mu_+\CE[E_-[\bH_+]]-\mu_-\CE[E_+[\bH_-]]), \\
\tilde\CN_7  = &(\CA_7(e_T[\Psi_{\bu^1}]) - \CA_7(e_T[\Psi_{\bu^2}]))
\nabla(\mu_+\CE[E_-[\bH^1_+]]-\mu_-\CE[E_+[\bH^1_-]]) \\
+ &\,\,\CA_7(e_T[\Psi_{\bu^2}])\nabla\{\mu_+(\CE[E_-[\bH^1_+]]-\CE[E_-[\bH^2_+]])
-\mu_-(\CE[E_+[\bH^1_-]]-\CE[E_+[\bH^2_-]])\}. 
\end{align*}
Employing the same argument as in proving \eqref{non:4.2}, \eqref{nd.6.1}, and \eqref{nd.6.3}, we have 
\begin{align}
&\|e^{-\gamma t}\tilde N_7(\bu, \bH)\|_{H^{1/2}_p(\BR, L_q(\dot\Omega))}
+ \|e^{-\gamma t}\tilde N_7(\bu, \bH)\|_{L_p(\BR, H^1_q(\dot\Omega))}
+ \gamma^{1/2}\|e^{-\gamma t}\tilde N_7(\bu, \bH)\|_{L_p(\BR, L_q(\dot\Omega))}\nonumber \\
&\quad \leq CT^{1/(2p')}(B+E^1_T(\bu_+))(B+E^2_T(\bH)), \label{n:7}\\
&\|e^{-\gamma t}\tilde \CN_7\|_{H^{1/2}_p(\BR, L_q(\dot\Omega))}
+ \|e^{-\gamma t}\tilde \CN_7\|_{L_p(\BR, H^1_q(\dot\Omega))} 
+ \gamma^{1/2}\|e^{-\gamma t}\tilde \CN_7\|_{L_p(\BR, L_q(\dot\Omega))} \nonumber \\
&\quad\leq CT^{1/(2p')}\{(B+E^2_T(\bH^1))E^1_T(\bu^1_+-\bu^2_+)
+(B+E^1_T(\bu^2_+))E^2_T(\bH^1-\bH^2)\}. \label{nd:7}
\end{align}

We finally define extensions of $(N_8(\bu, \bH), \bN_9(\bu, \bH)) = \CA_8(\Psi_\bu)[[\bH]]$ 
and  $\CN_8$ 
by setting
\begin{align*}
&(\tilde N_8(\bu, \bH), \tilde \bN_9(\bu, \bH)) = \CA_8(e_T[\Psi_\bu])(\mu_+\CE[E_-[\bH_+]]-\mu_-\CE_-[E_+[\bH_-]]); \\
&\tilde\CN_8  = (\CA_8(e_T[\Psi_{\bu^1}]) - \CA_8(e_T[\Psi_{\bu^2}]))
(\mu_+\CE[E_-[\bH^1_+]]-\mu_-\CE_-[E_+[\bH^1_-]]) \\
&+\CA_8(e_T[\Psi_{\bu^2}])\{\mu_+(\CE[E_-[\bH^1_+]]-\CE[E_-[\bH^2]])-
\mu_-(\CE_-[E_+[\bH^1_-]]-\CE_-[E_+[\bH^2_-]]) \}.
\end{align*}
By \eqref{e4.4}, \eqref{e4.15}, and \eqref{4.13}, we have 
\begin{align}
&\|e^{-\gamma t}(\tilde N_8(\bu, \bH), \tilde \bN_9(\bu, \bH))\|_{L_p(\BR, H^2_q(\dot\Omega))} +
\|e^{-\gamma t}\pd_t(\tilde N_8(\bu, \bH), \tilde \bN_9(\bu, \bH))\|_{L_p(\BR, L_q(\dot\Omega))} \nonumber \\
&\quad  \leq C\{T^{1/p'}E^1_T(\bu_+)E^2_T(\bH) 
+ T^{1/p}(B+E^1_T(\bu_+))(B+E^2_T(\bH))\}; \label{n:8}\\
&\|e^{-\gamma t}\tilde\CN_8\|_{L_p(\BR, H^2_q(\dot\Omega))} +
 \|e^{-\gamma t}\pd_t\tilde \CN_8\|_{L_p(\BR, L_q(\dot\Omega))} \nonumber\\
&\quad \leq C{(T^{1/p'}E^2_T(\bH^1))+T^{1/p}(B+E^2_T(\bH^1))}E^1_T(\bu^1_+-\bu^2_+) 
\nonumber\\
&\hskip1cm + C\{T^{1/p'}E^1_T(\bu^2_+)
+ T^{1/p}{(B+E^1_T(\bu^2_+))}\}E^2_T(\bH^1-\bH^2). \label{nd:8}
\end{align}

\section{A proof of Theorem \ref{thm:main}}\label{sec:5}

We shall prove Theorem \ref{thm:main} by contraction mapping principle. 
For this purpose, we define an underlying space $\bU_{T,L}$ for a large
number $L > 1$ and a small time $T \in (0, 1)$ by setting 
\begin{equation}\label{under:0}\begin{aligned}
\bU_{T, L} & = \{(\bu_+, \tilde\bH_\pm) \mid \bu_+ 
\in H^1_p((0, T), H^1_q(\Omega_+)^N) \cap L_p((0, T), H^3_q(\Omega_+)^N), \\
&\tilde\bH_\pm \in H^1_p((0, T), L_q(\Omega_\pm)^N) \cap 
L_p((0, T), H^2_q(\Omega_\pm)^N), \\
&\bu_+|_{t=0} = \bu_{0+} \enskip\text{in $\Omega_+$}, \quad 
\tilde\bH_\pm|_{t=0} = \tilde\bH_{0\pm} \enskip\text{in $\Omega_\pm$}, \\
&E_T^1(\bu_+)\leq L, \quad  E^{2,\pm}_T(\tilde\bH_\pm) \leq  L\}.
\end{aligned}\end{equation}
Let $B$ be a positive number  for which initial
data $\bu_{0+}$ and $\tilde\bH_{0\pm}$ for equations \eqref{mhd.3}
satisfy the condition \eqref{initial:4.1}. 

Let $(\bu_+, \tilde\bH_\pm)$ be an element of $\bU_{T, L}$, and let 
$\bN_{5\pm}(\bu, \tilde\bH)$, 
$\bN_6(\bu, \tilde\bH)$, $N_7(\bu, \tilde\bH)$, $N_8(\bu, \tilde\bH)$,
and $\bN_9(\bu, \tilde\bH)$ be respective non-linear terms defined
in \eqref{2.13}, \eqref{2.16}, \eqref{2.17}, \eqref{2.18*}, and \eqref{2.19*}.
Let $\bH$ be a solution of equations:
\allowdisplaybreaks{
\begin{align}
\mu\pd_t\bH - \alpha^{-1}\Delta\bH& = \bN_5(\bu, \tilde\bH)&
\quad&\text{in $\dot\Omega\times(0, T)$}, \nonumber \\
[[\alpha^{-1}\curl\bH]]\bn &= \bN_6(\bu, \tilde\bH)&
\quad&\text{on $\Gamma\times(0, T)$}, \nonumber \\
[[\mu\dv\bH]] &=N_7(\bu, \tilde\bH)&
\quad&\text{on $\Gamma\times(0, T)$}, \nonumber \\
[[\mu \bH\cdot\bn]] &= N_8(\bu, \tilde\bH)&
\quad&\text{on $\Gamma\times(0, T)$}, \nonumber \\
[[\bH-<\bH, \bn>\bn]] &=\bN_9(\bu, \tilde\bH)&\quad 
\quad&\text{on $\Gamma\times(0, T)$}, \nonumber \\
\bn_-\cdot\bH_- = 0, \quad
(\curl\bH_-)\bn_- &= 0& \quad
&\text{on $S_\pm\times(0, T)$},  \nonumber \\
\bH|_{t=0} &= \tilde\bH_0&\quad
&\text{in $\dot\Omega$}.\label{eq:5.1}
\end{align}}
Next, let $\bN_1(\bu_+, \bH_+)$, $N_2(\bu_+)$, $\bN_3(\bu_+)$
 and $\bN_4(\bu_+, \bH_+)$ be respective non-linear
terms given in \eqref{2.12}, \eqref{2.11}, and \eqref{2.15} by replacing
$\tilde\bH_+$ with $\bH_+$, where $\bH_+=\bH|_{\Omega_+}$ and 
$\bH$ is a solution of equations \eqref{eq:5.1}. And then, let $\bv$ be 
a solution of equations:
\begin{equation}\label{eq:5.2}
\begin{aligned}
\rho\pd_t\bv_+ - \DV\bT(\bv_+, \fq) &= \bN_1(\bu_+, \bH_+)
&\quad &\text{in $\Omega_+\times(0, T)$}, \\
 \dv\bv_+= N_2(\bu_+) &= \dv\bN_3(\bu_+)
&\quad &\text{in $\Omega_+\times(0, T)$}, \\
\bT(\bv_+, \fq)\bn&= \bN_4(\bu_+, \bH_+)
&\quad &\text{on $\Gamma\times(0, T)$}, \\
\bv_+|_{t=0} & = \bu_{0+}
&\quad&\text{in $\Omega_+$}.
\end{aligned}
\end{equation}
Recalling that $E^1_T(\bu_+) \leq L$,  in view of  \eqref{4.3}, \eqref{4.4}, and \eqref{2.1**}
we choose $T > 0$ so small that
$$\|\Psi_\bu\|_{L_\infty((0, T), L_\infty(\Omega))} 
\leq C\|\Psi_\bu\|_{L_\infty((0, T), H^2_q(\Omega))} 
\leq CT^{1/p'}L \leq \delta.
$$
Moreover, in view of \eqref{4.4.2}, we choose $T>0$ in such a way that
$$T^{1/p'}(E^1_T(\bu_+)+B) \leq T^{1/p'}(L+B) \leq 1.$$
Let
$\tilde\bh=(\tilde\bN_6(\bu, \tilde\bH), N_7(\bu, \tilde\bH))$.
and 
$\tilde\bk=(\tilde N_8(\bu, \tilde\bH), \tilde\bN_9(\bu, \tilde\bH))$, and 
let $F_H(\bff, \tilde\bh, \tilde\bk)$ a symbol given in Theorem \ref{thm:3.2.1}. 
By \eqref{n:5}, \eqref{n:6}, \eqref{n:7}, and \eqref{n:8},
we have
\begin{align*}
F_H(\bN_5(\bu, \tilde\bH), \tilde\bh, \tilde\bk) &\leq C[T^{1/p}(B+L)^2 
+ T^{1/p'}(B+L)^2+T^{1/(2p')}(B+L)^2\\
&+(1+(B+L)T^{1/(2p')})\{e^{2(\gamma-\gamma_0)}B^2
+(T^{1/2}+T^{1/p})(B+L)^2\}]
\end{align*}
for any $\gamma\geq\gamma_0$. We fix $\gamma$ as $\gamma= \gamma_0$. 
Let $\alpha = \min(1/p. 1/{p'}, 1/{2p'}, 1/2, s/{p'(1+s)})$.  Since $0 < t< 1$, 
there exsit positive constants $M_1$ and $M_2$ for which 
$$F_H(\bN_5(\bu, \tilde\bH), \tilde\bh, \tilde\bk) 
\leq M_1B^2 + M_2((L+B)^2  + (L+B)^3)T^\alpha.$$
Applying Theorem \ref{thm:3.2.1} to equations \eqref{eq:5.1} yields that 
$E^2_T(\bH) \leq Ce^{\gamma_0T}(B + F_H(\bN_5(\bu, \tilde\bH), \tilde\bh, \tilde\bk))$
for some constant $C_1$.
Choosing $T>0$ in such a way that $\gamma_0T \leq 1$ gives that 
\begin{equation}\label{5.2}
E^2_T(\bH) \leq C_1e[{B + M_1B^2} + M_2({(L+B)^2 + (L+B)^3})T^\alpha]
\end{equation}
In particular, we choose $T > 0$ so small that 
$M_2({(L+B)^2 + (L+B)^3})T^\alpha \leq {B + M_1B^2}$ and $L > 0$ so large 
that ${4}C_1e{(B + M_1B^2)} \leq L$, and then by \eqref{5.2} 
\begin{equation}\label{5.3} 
E^{2,\pm}_T(\bH^\pm) \leq L{/2}.
\end{equation}

We next consider equations \eqref{eq:5.2}.  Let $F_v$ be a symbol given in
Theorem \ref{thm:3.1.1}.  By \eqref{n:1}, \eqref{n:2}, \eqref{n:3}, \eqref{5.2},
and \eqref{5.3},
\begin{align*}
&F_v(\bN_1(\bu_+, \bH_+), N_2(\bu_+), \bN_3(\bu_+), \bN_4(\bu_+, \bH))\\
&\leq C(T^{1/p'}L^2 + T^{1/p}(L+B)^2 + T^{s/(p'(1+s))}{(L^2+(L+B)L^2)} +{(1+B)}BE^{2,+}_T(\bH_+))\\
&\leq C[T^{1/p'}L^2 + T^{1/p}(L+B)^2 + T^{s/(p'(1+s))}{(L^2+(L+B)L^2)}\\
&\qquad +{(1+B)}BCe\{M_1B^2 + M_2((L+B)^2 + (L+B)B^2 + (L+B)^3)T^\alpha\}],
\end{align*}
which yields that 
$$F_v(\bN_1(\bu_+, \bH_+), N_2(\bu_+), \bN_3(\bu_+), \bN_4(\bu_+, \bH))
\leq M_3{(1+B)^2}B^2 + M_4{(1+B)^2}((L+B)^2 + (L+B)^3)T^\alpha
$$
for some constants  $M_3$ and $M_4$.
Thus, applying Theorem \ref{thm:3.1.1} with $0 < T < 1$ to equations \eqref{eq:5.2}, we have
$$E^1_T(\bv) \leq C_2e^{\gamma_1T}\{B + M_3{(1+B)^2}B^2 + M_4{(1+B)^2}((L+B)^2 + (L+B)^3)T^\alpha\}.
$$
for some constant $C_2$.
Recalling that $\gamma_1 \leq \gamma_0$ and $\gamma_0T \leq 1$, choosing
$T>0$ so small that $M_4{(1+B)^2}((L+B)^2 + (L+B)^3)T^\alpha \leq 
B + M_3{(1+B)^2}B^2$ and choosing $L > 0$ so large that 
$L \geq {4}C_2e(B + M_3{(1+B)^2}B^2)$, we have   
$E^1_T(\bv) \leq L{/2}$, which implies that 
$(\bv, \bH_\pm) \in \bU_{T, L}$.  In particular, we set
$L = {4}e\max(C_1{(B+M_1B^2)}, C_2(B + M_3{(1+B)^2}B^2))$.
Let $\Phi$ be a map acting on $(\bu, \tilde\bH) \in \bU_{T, L}$
by setting $\Phi(\bu, \tilde\bH) = (\bv, \bH)$, and then $\Phi$ is a map from
$\bU_{T,L}$ into itself. 

We now prove that $\Phi$ is a contraction map. Let $(\bu_+^i, \tilde\bH^i)
\in \bU_{T,L}$ ($i=1,2$) and set $(\bv^i_+, \bH^i) = \Phi(\bu_+^i, \tilde\bH^i)$.  
In view of \eqref{2.1**}, \eqref{4.4.2}, and \eqref{2.1a}, choosing
$T > 0$ smaller if necessary, we may assume that 
$$\|\Psi_{\bu^i}\|_{L_\infty((0, T), L_\infty(\Omega))} \leq \delta,
\quad
\|e_T[\Psi_{\bu^i}]\|_{L_\infty(\BR, L_\infty(\Omega))} \leq \delta, 
\quad T^{1/p'}(E^1_T(\bu_i) + B) \leq 1
$$
for $i=1,2$.
Set 
\begin{align*} \bv_+ &= \bv^1_+-\bv^2_+, \quad \bH = \bH^1-\bH^2, 
\quad \CN_i = \bN_i(\bu^1_+, \bH^1_+) - \bN_i(\bu^2_+, \bH^2_+),
\quad \CN_2 = N_2(\bu^1_+)-N_2(\bu^2_+), \\
\CN_3& = \bN_3(\bu^1_+)-\bN_3(\bu^2_+), 
\quad \CN_j = \bN_j(\bu^1_+, \tilde\bH^1_+) - \bN_j(\bu^2_+,\tilde  \bH^2_+),
\quad \CN_k = N_k(\bu_+^1, \tilde\bH^1) - N_k(\bu^2_+, \tilde\bH^2)
\end{align*}}
for $i= 1, 4$, $j=5, 6, 9$ and $k=7,8$. Noticing that 
$\bv_+^1|_{t=0}=\bv_+^2|_{t=0} = \bu_{0+}$, and $\bH^1_\pm|_{t=0}=\bH^2_\pm|_{t=0}=\tilde\bH_{0\pm}$, 
by \eqref{eq:5.1}, \eqref{eq:5.2} we see that $\bH$ satisfies the following 
equations:
\begin{equation}\label{eqh:5.1}\begin{aligned}
\mu\pd_t\bH - \alpha^{-1}\Delta\bH = \CN_5&&
\quad&\text{in $\dot\Omega\times(0, T)$},\\
[[\alpha^{-1}\curl\bH]]\bn = \CN_6, \quad 
[[\mu\dv\bH]] =\CN_7&&
\quad&\text{on $\Gamma\times(0, T)$}, \\
[[\mu \bH\cdot\bn]] = \CN_8, \quad 
[[\bH-<\bH, \bn>\bn]] =\CN_9&&\quad 
\quad&\text{on $\Gamma\times(0, T)$}, \\
\bn_\pm\cdot\bH_\pm = 0, \quad
(\curl\bH_\pm)\bn_\pm = 0&& \quad
&\text{on $S_\pm\times(0, T)$}, \\
\bH|_{t=0} = 0&&\quad
&\text{in $\dot\Omega$}.
\end{aligned}\end{equation}
and that $\bv_+$ satisfies the following equations:
\begin{equation}\label{eqh:5.2}
\begin{aligned}
\rho\pd_t\bv_+ - \DV\bT(\bv_+, \fq) = \CN_1&
&\quad &\text{in $\Omega_+\times(0, T)$}, \\
 \dv\bv_+= \CN_2 = \dv\CN_3&
&\quad &\text{in $\Omega_+\times(0, T)$}, \\
\bT(\bv_+, \fq)\bn= \CN_4&
&\quad &\text{on $\Gamma\times(0, T)$}, \\
\bv_+=0&
&\quad&\text{on $S_+\times(0, T)$}, \\
\bv_+|_{t=0}  = 0&
&\quad&\text{in $\Omega_+$}.
\end{aligned}
\end{equation}
Set $\tilde\CH = (\CN_6, \CN_7)$ and $\tilde\CK=(\CN_8, \CN_9)$.
By \eqref{nd:5}, \eqref{nd:6}, \eqref{nd:7}, and \eqref{nd:8}, we have
\begin{align*}
F_H(&\CN_5, \tilde\CH, \tilde\CK) \leq C[\{T^{1/p'}((B+L)^2 + (B+L))
+ T^{1/p}(B+L) + T^{1/(2p')}(B+L)\}E^1_T(\bu^1-\bu^2) \\
&\quad  +\{T^{1/p'}L + T^{1/p}(B+L) + T^{1/(2p')}(B+L)\}E^2_T(\tilde\bH^1-\tilde\bH^2) \\
&\quad +
\{T^{1/(2p')}(B+L+e^{\gamma-\gamma_0}B^2)+(T^{1/2}+T^{1/p})(1+T^{1/(2p')}(B+L))
(B+L)\}E^1_T(\bu^1-\bu^2)\\
&\quad +\{T^{1/(2p')}(B+L)+(T^{1/2}+T^{1/p})(1+T^{1/(2p')}(B+L))(B+L)\}
E^2_T(\tilde\bH^1-\tilde\bH^2)].
\end{align*}
for any $\gamma \geq \gamma_0$.  Thus, choosing $\gamma=\gamma_0$ and 
noting $0 < T < 1$, we have
\begin{equation}\label{5.4}
F_H(\CN_5, \tilde\CH, \tilde\CK) \leq C(B+L+(B+L)^2)T^\alpha(
E^1_T(\bu^1_+-\bu^2_+) + E^2_T(\tilde\bH^1-\tilde\bH^2)).
\end{equation}
Applying Theorem \ref{thm:3.2.1} to equations \eqref{eqh:5.1} and using
\eqref{5.4} gives that 
\begin{equation}\label{diff.1}
E^1_T(\bH^1-\bH^2) \leq M_5(B+L+(B+L)^2)T^\alpha(E^1_T(\bu^1_+-\bu^2_+)
+E^2_T(\tilde\bH^1-\tilde\bH^2))
\end{equation}
for some constant $M_5$, where we have used $\gamma_2\leq \gamma_0$ and 
$\gamma_0T\leq 1$. Moreover, by \eqref{nd.1}, \eqref{nd:2}, and \eqref{nd:4}
\begin{align*}
F_v(\CN_1, \CN_2, \CN_3, \CN_4) \leq &\,C\{
{(T^{1/p'}+T^{1/p})((L+B)+(L+B)^2)}\}E^1_T(\bu^1_+-\bu^2_+) \\
+ &\,C{\{T^{1/p}(B+L)^2+(1+T^{1/p}(B+L))(B+T^{s/p'(1+s)}L)\}}E^2_T(\bH^1-\bH^2)\},
\end{align*}
which, combined with \eqref{diff.1}, leads to 
\begin{align*}
F_v(\CN_1, \CN_2, \CN_3, \CN_4) \leq &\,{C(B+L+(B+L)^2)T^\alpha
E^1_T(\bu^1_+-\bu^2_+)} \\
{+}&\,{CM_5(B+L+(B+L)^2)^2T^\alpha(E^1_T(\bu^1_+-\bu^2_+) 
+ E^2_T(\tilde\bH^1-\tilde\bH^2))}.
\end{align*}
Thus, applying Theorem \ref{thm:3.1.1} to equations \eqref{eqh:5.2} leads to 
\begin{equation}\label{diff.2}\begin{aligned}
E^1_T(\bv^1_+-\bv^2_+) \leq M_6\{&{(B+L+(B+L)^2)}\\
&{+(B+L+(B+L)^2)^2}\}T^\alpha
(E^1_T(\bu^1_+-\bu^2_+) + E^2_T(\tilde\bH^1-\tilde\bH^2))
\end{aligned}\end{equation}
for some $M_6$, where we have used $\gamma_1 \leq \gamma_0$ and $\gamma_0T\leq 1$.
Choosing $T>0$ so small that 
$M_5(B+L+(B+L)^2)T^\alpha \leq 1/4$ in \eqref{diff.1} and 
$M_6\{{(B+L+(B+L)^2)+(B+L+(B+L)^2)^2}\}T^\alpha \leq 1/4$ in 
\eqref{diff.2} gives that 
$$E^1_T(\bv^1_+-\bv^2_+) + E^2_T(\bH_1-\bH_2) \leq (1/2)
(E^1_T(\bu^1_+-\bu^2_+) 
+ E^2_T(\tilde\bH^1-\tilde\bH^2)),
$$
which shows that the $\Phi$ is a contraction map.  Thus, the Banach fixed point
theorem yields the unique existence of a fixed point, $(\bu_+, \tilde\bH_\pm)
\in \bU_{T,L}$, of the map $\Phi$, which is a unique solution of 
equations \eqref{mhd.3}.  This completes the proof of Theorem \ref{thm:main}.
\appendix
\section{A Proof of Theorem \ref{thm:3.1.1}}
To prove the maximal $L_p$-$L_q$ regularity, acording to Shibata \cite{S1, S2, S20}  
a main step is to prove the 
existence of  $\CR$-solver for the following model problem: 
\begin{equation}\label{Ap:1}\begin{aligned}
\lambda\bu - \DV(\mu\bD(\bu)-\fp\bI) =\bff&&\quad&\text{in $\BR^N_+$}, \\
\dv\bu = g = \dv\bg&&\quad&\text{in $\BR^N_+$}, \\
(\mu\bD(\bu)-\fp\bI)\bn=\bh&&\quad&\text{on $\BR^N_0$}, 
\end{aligned}\end{equation}
where $\BR^N_+ = \{x = (x_1, \ldots, x_N) \mid x_N > 0\}$, 
$\BR^N_0 = \{x = (x_1, \ldots, x_{N-1}, 0)
\mid (x_1, \ldots, x_{N-1}) \in \BR^{N-1}\}$, and $\bn = (0, \ldots, 0, -1)$.  The 
$\lambda$ is a complex parameter ranging in $\Sigma_{\epsilon, \lambda_0}
= \{ \lambda \in \BC\setminus\{0\} \mid |\arg\lambda| 
\leq \pi-\epsilon, |\lambda| \geq \lambda_0\}$
with $0 < \epsilon < \pi/2$ and $\lambda_0>0$.  In \cite{S1, S20}, it was proved the existence
of $\CR$-solvers, $\CS(\lambda)$, $\CP(\lambda)$,  for equations \eqref{Ap:1} 
which satisfy the following properties:
\begin{itemize}
\item[\thetag1]~
$
\CS(\lambda) \in {\rm Hol}\,(\Sigma_{\epsilon, \lambda_0}, \CL(\CX_q(\BR^N_+), 
H^2_q(\HS)^N))$, \quad
$\CP(\lambda) \in {\rm Hol}\,(\Sigma_{\epsilon, \lambda_0},
\CL(\CX_q(\HS), H^1_q(\HS) + \hat H^1_{q,0}(\HS)))$. 
\item[\thetag2]~ Problem \eqref{Ap:1} admits unique solutions 
$\bu=\CS(\lambda)F_\lambda(\bff, g, \bg, \bh)$ and 
$\fp = \CP(\lambda)F_\lambda(\bff, g, \bg, \bh)$
for any $(\bff, g, \bg, \bh) \in X_q$ and $\lambda \in \Sigma_{\epsilon, \lambda_0}$, 
where
 $F_\lambda(\bff, g, \bg, \bh) = (\bff, g, \lambda^{1/2}g, \lambda\bg, \bh, \lambda^{1/2}\bh),
 $. 
 \item[\thetag3]~\hskip2cm$
 \CR_{\CL(\CX_q(\HS), H^{2-j}_q(\HS)^N)}(\{(\tau\pd_\tau)^\ell(\lambda^{j/2}\CS(\lambda)) 
\mid
 \lambda \in \Sigma_{\epsilon, \lambda_0}\}) \leq r_b \quad(j=0,1,2)$,  \vskip0.1pc
 \hskip2.85cm $\CR_{\CL(\CX_q(\HS), L_q(\HS)^N)}
(\{(\tau\pd_\tau)^\ell(\nabla\CP(\lambda)) \mid
 \lambda \in \Sigma_{\epsilon, \lambda_0}\}) \leq r_b$\vskip0.1pc
 for $\ell=0,1$ with some constant $r_b$ depending on $\lambda_0$ and $\epsilon$. 
 \end{itemize}
 Here, $\lambda = \gamma + i\tau \in \BC$,  $\CR_{\CL(X, Y)}\CT$ denotes the 
$\CR$ norm of an operator family $\CT \subset \CL(X, Y)$, $\CL(X, Y)$ being the 
set of all bounded linear operators from $X$ into $Y$, 
 \begin{align*}
 \CX_q(\HS) & = \{(F_1, \ldots, F_6) \mid F_1, F_4, F_6 \in L_q(\HS)^N, \quad
 F_2 \in H^1_q(\HS), \quad F_3 \in L_q(\HS), \quad F_5 \in H^1_q(\HS)^N\}; \\
 X_q(\HS) & = \{(\bff, g, \bg, \bh) \mid \bff \in L_q(\HS)^N, \enskip g \in H^1_q(\HS), 
 \enskip \bg \in L_q(\HS)^N, \enskip 
 \bh \in H^1_q(\HS)^N, \enskip g = \dv\bg\}.
 \end{align*}
 The $F_1$, $F_2$, $F_3$, $F_4$, $F_5$, and $F_6$ are corresponding variables
 to $\bff$, $g$, $\lambda^{1/2}g$, $\lambda\bg$, $\bh$, and $\lambda^{1/2}\bh$, 
 respectively. The norm of $\CX_q(\HS)$ is defined by  setting
 $$\|(F_1, \ldots, F_6)\|_{\CX_q(\HS)}
 = \|(F_1, F_3, F_4, F_6)\|_{L_q(\HS)} + \|(F_2, F_5)\|_{H^1_q(\HS)}.
 $$
 In particular, we know an unique existence of solutions $\bu \in H^2_q(\HS)^N$ and
 $\fp \in H^1_q(\HS) + \hat H^1_{q,0}(\HS)$
\footnote{Here, we  just give an idea of obtaining third order regularities. An idea 
also is found in \cite[Appendix 6.2]{DHMT}. To prove 
Theorem \ref{thm:3.1.1} exactly from the $\CR$-bounded solution operators
point of view, we have to start returning  the non-zero $\bff$, $g$
and $\bg$ situation to the situation where $\bff=g=\bg=0$, which needs an idea.
We will give an exact proof of Theorem \ref{thm:3.1.1} in a forthcomming 
paper.}  of equations \eqref{Ap:1} 
 possessing the estimate:
 \begin{equation}\label{ap:2}\begin{aligned}
 &\|\lambda\bu\|_{L_q(\HS)} + \|\bu\|_{H^2_q(\HS)} + \|\nabla\fp\|_{L_q(\HS)}\\
 &\quad \leq C(\|\bff\|_{L_q(\HS)} + \|(g, \bh)\|_{H^1_q(\HS)}
 + \|\lambda^{1/2}(g, \bh)\|_{L_q(\HS)} + \|\lambda\bg\|_{L_q(\HS)}).
 \end{aligned}\end{equation}
 We now prove that $\bu \in H^3_q(\HS)^N$ and $\nabla\fp \in H^1_q(\HS)^N$ provided that
 $\bff \in H^1_q(\HS)^N$, $g \in H^2_q(\HS)$, $\bg \in H^1_q(\HS)^N$, and 
 $\bh \in H^2_q(\HS)^N$. Moreover, $\bu$ and $\fp$ satisfy the estimate:
 \begin{equation}\label{ap:3}\begin{aligned}
 &\|\lambda\bu\|_{H^1_q(\HS)} + \|\bu\|_{H^3_q(\HS)} + \|\nabla\fp\|_{H^1_q(\HS)} \\
 &\quad\leq C(\|\bff\|_{H^1_q(\HS)} + \|(g, \bh)\|_{H^2_q(\HS)}
 +\|\lambda(g, \bh)\|_{L_q(\HS)} + \|\lambda \bg\|_{H^1_q(\HS)}). 
 \end{aligned}\end{equation}
 In fact, differentiating equations \eqref{Ap:1} with respect to tangential variables $x_j$ ($j=1, \ldots, N-1$)
 and noting that $\pd_j\bu$ and $\pd_j\fp$ satisfy equations replacing 
 $\bff$, $g=\dv\bg$, and $\bh$ with $\pd_j\bff$, $\pd_jg=\dv\pd_j\bg$ and $\pd_j\bh$,
 by \eqref{ap:2} and the uniquness of solutions we see that 
 $\pd_j\bu \in H^2_q(\HS)^N$ and $\nabla \pd_j\fp \in L_q(\HS)^N$, and 
 \begin{equation}\label{ap:4}\begin{aligned}
 &\|\lambda\pd_j\bu\|_{L_q(\HS)} + \|\pd_j\bu\|_{H^2_q(\HS)} + \|\nabla\pd_j\fp\|_{L_q(\HS)}\\
 &\quad \leq C(\|\pd_j\bff\|_{L_q(\HS)} + \|(\pd_jg, \pd_j\bh)\|_{H^1_q(\HS)}
 + \|\lambda^{1/2}(\pd_jg, \pd_j\bh)\|_{L_q(\HS)} + \|\lambda\pd_j\bg\|_{L_q(\HS)})
 \end{aligned}\end{equation}
 for $j=1, \ldots, N-1$. To estimate $\pd_N\bu$ and $\pd_N\fp$, we start with estimating 
$\pd_Nu_N$. 
 In fact, from the divergence equations it follows that  
 $\pd_Nu_N = -\sum_{j=1}^{N-1}\pd_ju_j + g$, 
 and so 
 $$\lambda\pd_Nu_N = -\sum_{j=1}^{N-1}\lambda\pd_ju_j + \lambda g,\quad 
 \pd_N^3u_N = -\sum_{j=1}^{N-1}\pd_N^2\pd_ju_j + {\pd_N^2}g,
 $$
 which, combined with \eqref{ap:4} yields that 
 \begin{equation}\label{ap:5}\begin{aligned}
 &\|\lambda\pd_Nu_N\|_{L_q(\HS)} + \|\pd_N^3u_N\|_{L_q(\HS)} \\
 &\quad \leq C{\bigg\{} \sum_{j-1}^{N-1}(\|\pd_j\bff\|_{L_q(\HS)} + \|(\pd_jg, \pd_j\bh)\|_{H^1_q(\HS)}
 + \|\lambda^{1/2}(\pd_jg, \pd_j\bh)\|_{L_q(\HS)} + \|\lambda\pd_j\bg\|_{L_q(\HS)})\\
&\hspace{15mm} {+\|(\lambda g,\pd_N^2g)\|_{L_q(\HS)} \bigg\}}.
 \end{aligned}\end{equation}
 From the {the $N$-th component of the first} equation of equations \eqref{Ap:1} and $\dv \bu = g$, we have  
 $$\lambda u_N - \mu\Delta u_N - \mu\pd_Ng + \pd_N\fp = f_N,
 $$
 and so, we see that $\pd_N^2\fp \in L_q(\HS)$ and 
 \begin{equation}\label{ap:6}\|\pd_N^2\fp\|_{L_q(\HS)} \leq  \|\pd_Nf_N\|_{L_q(\HS)}
 + \mu\|\pd_N^2g\|_{L_q(\HS)} + 
 \|\lambda\pd_Nu_N\|_{L_q(\HS)} + \mu\|u_N\|_{H^3_q(\HS)}.
 \end{equation}
 From equations \eqref{Ap:1}, we have
 $$\begin{aligned}
 \lambda u_j - \mu\Delta u_j & = f_j - \pd_jp+ \mu\pd_jg&\quad&\text{in $\HS$}, \\
 \pd_Nu_j & = -\pd_ju_N + \mu^{-1}h_j&\quad&\text{on $\BR^N_0$}.
 \end{aligned}$$
 Differentiating the first equation of the above set of equations with respect to
 $x_N$ and setting $\pd_Nu_j=v$, we have
 $$\begin{aligned}
 \lambda v- \mu\Delta v& = \pd_Nf_j - \pd_j\pd_N\fp+ \mu\pd_N\pd_jg&\quad&\text{in $\HS$}, \\
 v & = -\pd_ju_N + \mu^{-1}h_j&\quad&\text{on $\BR^N_0$}.
 \end{aligned}$$
 Thus, setting $w = v+\pd_ju_N - \mu^{-1}h_j$, we have
 $$\begin{aligned}
 \lambda w- \mu\Delta w& = \pd_Nf_j - \pd_j\pd_N\fp+ \mu\pd_N\pd_jg
 +(\lambda-\Delta)(\pd_ju_N - \mu^{-1}h_j)&\quad&\text{in $\HS$}, \\
w & =0&\quad&\text{on $\BR^N_0$}.
 \end{aligned}$$
 Thus, by a known estimate for the Dirichlet problem, we have
 \begin{equation}\label{ap:7}\begin{aligned}
 &\|\lambda\pd_Nu_j\|_{L_q(\HS)} + \|\pd_Nu_j\|_{H^2_q(\HS)}\\
&\quad  \leq C\{\|\pd_Nf_j\|_{L_q(\HS)} + \|\pd_j\pd_N\fp\|_{L_q(\HS)}
 + \|\pd_j\pd_Ng\|_{L_q(\HS)}
 + \|\lambda\pd_ju_N\|_{L_q(\HS)} \\
&\quad  + \|\lambda h_j\|_{L_q(\HS)} + \|\pd_ju_N\|_{H^2_q(\HS)}
 + \|h_j\|_{H^2_q(\HS)}\}.
 \end{aligned}\end{equation}
 Noting that $\|\lambda^{1/2}(\pd_jg, \pd_j\bh)\|_{L_q(\HS)} \leq C(\|\lambda(g, \bh)\|_{L_q(\HS)}
 + \|(g, \bh)\|_{H^2_q(\HS)})$ and combining \eqref{ap:2} and \eqref{ap:4}--\eqref{ap:7}, we have
 \eqref{ap:3}.

Localizing the problem and using the argument above, we can show Theorem \ref{thm:3.1.1}. 


\end{document}